\documentclass{amsart}

\usepackage{etex}
\usepackage{amsmath,amssymb,amsthm,amsfonts,mathrsfs}
\usepackage[frame,cmtip,arrow,matrix,line,graph,curve]{xy}
\usepackage{graphpap,color,paralist,pstricks}
\usepackage[mathscr]{eucal}
\usepackage{mathabx}
\usepackage[pdftex,colorlinks,backref=page,citecolor=blue]{hyperref}
\usepackage{tikz}
\usetikzlibrary{calc,decorations.markings}
\usepackage{epic,eepic}
\usepackage{yfonts}
\usepackage{enumitem} 
\usepackage{bbm}
\usepackage{tikz-cd}

\theoremstyle{definition}
\newtheorem{theorem}{Theorem}[section]
\newtheorem{definition}[theorem]{Definition}

\newtheorem{lemma}[theorem]{Lemma}
\newtheorem{proposition}[theorem]{Proposition}
\newtheorem{corollary}[theorem]{Corollary}

\newtheorem{question}[theorem]{Question}
\newtheorem*{theorem*}{Theorem}

\theoremstyle{remark}
\newtheorem{remark}[theorem]{Remark}
\newtheorem{example}[theorem]{Example}

\renewcommand{\AA}{\mathbb A}
\newcommand{\QQ}{\mathbb Q}
\newcommand{\RR}{\mathbb R}
\newcommand{\CC}{\mathbb C}
\newcommand{\PP}{\mathbb P}
\newcommand{\ZZ}{\mathbb Z}
\newcommand{\GG}{\mathbb G}
\newcommand{\be}{\mathbf e}
\newcommand{\kk}{\mathbbm k}
\newcommand{\M}{\mathrm M}

\newcommand{\U}{\mathrm U}
\renewcommand{\H}{\mathrm{H}}
\newcommand{\cS}{\mathcal S}
\newcommand{\cQ}{\mathcal Q}
\newcommand{\cG}{\mathcal G}
\newcommand{\V}{V}

\newcommand{\ind}{\mathbb I}
\newcommand*\one{\mathbf{1}} 
\newcommand{\msum}{\uplus}

\newcommand{\rk}{\operatorname{rk}}

\begin{document}

\title{Stellahedral geometry of matroids}

\author{Christopher Eur, June Huh, Matt Larson}

\address{Harvard University}
\email{ceur@math.harvard.edu}

\address{Princeton University and Korea Institute for Advanced Study}
\email{huh@princeton.edu}

\address{Stanford University}
\email{mwlarson@stanford.edu}

\begin{abstract}
We use the geometry of the stellahedral toric variety to study matroids.  We identify the valuative group of matroids with the cohomology ring of the stellahedral toric variety, and show that valuative, homological, and numerical equivalence relations for matroids coincide.  We establish a new log-concavity result for the Tutte polynomial of a matroid, answering a question of Wagner and Shapiro--Smirnov--Vaintrob on Postnikov--Shapiro algebras, and calculate the Chern--Schwartz--MacPherson classes of matroid Schubert cells.  The central construction is the ``augmented tautological classes of matroids,'' modeled after certain toric vector bundles on the stellahedral toric variety.
\end{abstract}

\maketitle

\renewcommand{\baselinestretch}{.95}\normalsize
\setcounter{tocdepth}{1}
\tableofcontents
\renewcommand{\baselinestretch}{1.2}\normalsize

\section{Introduction}

Let $E = \{1, \dotsc, n\}$.  For $S\subseteq E$, we write $\be_S$ for the sum of the standard basis vectors $\sum_{i \in S} \be_i$ in the vector space $\mathbb{R}^E$.  A \emph{matroid} $\M$ on $E$ is a collection $\mathscr{B}$ of subsets of $E$, called the \emph{bases} of $\M$, such that every edge of the convex hull
\[
P(\M) \coloneq \operatorname{conv}\{\be_B \, | \, B\in \mathscr{B}\} \subseteq \RR^E
\]
is parallel to $\be_i - \be_j$ for some $i$ and $j$ in $E$.
 By definition, the coordinate sum of any point in the \emph{base polytope} $P(\M)$ is a constant integer $\rk(\M)$, called the \emph{rank} of $\M$, which is equal to $|B|$ for any $B\in \mathscr{B}$.
The condition on the edges of the base polytope is equivalent to the \emph{basis exchange property} {appearing in the work of Whitney \cite{Whitney35} that introduced matroids}:
\begin{quote}
For any $B_1,B_2 \in \mathscr{B}$ and any $i \in B_1 \setminus B_2$, there is $j \in B_2 \setminus B_1$ such that $(B_1 \setminus i) \cup j \in \mathscr{B}$.
\end{quote}
The above definition of matroids via base polytopes arose from  the study of moment map images of torus orbit closures in Grassmannians by Gelfand, Goresky, MacPherson, and Serganova in \cite{GGMS}.
See \cite[Chapter 1]{Kun86} for an excellent historical overview of early contributions, and \cite{Ard22} and \cite{Eur} for snapshots of recent advances in the theory of matroids.
For a general introduction to matroids, and for any undefined matroid terms, we refer to \cite{Oxl11}.

For a nonnegative integer $r \le n$, we consider the free abelian group generated by the set of matroids of rank $r$ on $E$:
\[
\mathrm{Mat}_r(E) \coloneq \Big\{ \sum_i c_i \M_i \ \Big| \ \text{$c_i$ is an integer and $\M_i$ is a rank $r$ matroid on $E$}\Big\}.
\]
We study three equivalence relations on $\mathrm{Mat}_r(E)$---valuative, homological, and numerical.

\begin{definition}
Let $\mathbf 1_{P(\M)}$ be the indicator function of the base polytope of $\M$, which is the function $\RR^E\to \ZZ$ defined by $\mathbf 1_{P(\M)}(x) = 1$ if $x\in P(\M)$ and $\mathbf 1_{P(\M)}(x) = 0$ otherwise.
An element $\sum_i c_i \M_i$ is said to be \emph{valuatively equivalent to zero} if the function $\sum_i c_i \mathbf 1_{P(\M_i)}$ is zero.
\end{definition}

Figure \ref{fig:zero} illustrates an  element of $\mathrm{Mat}_2\, ([4])$ that is valuatively equivalent to zero.
The valuative group of rank $r$ matroids on $E$, denoted $\mathrm{Val}_r(E)$, is the group $\mathrm{Mat}_r(E)$ modulo the subgroup of elements valuatively equivalent to zero.
A homomorphism of abelian groups $ \mathrm{Mat}_r(E) \to G$ is said to be \emph{valuative} if it factors through the valuative group. 
Many matroid invariants, including the Tutte polynomial, 
the Kazhdan--Lusztig polynomial, 
the motivic zeta function, 
the Chern--Schwartz--MacPherson cycle, 
and the volume polynomial of the Chow ring, 
turn out to be valuative. 
See \cite{AFR10,AS22,Ard22} for extensive lists and history of the study of valuative matroid invariants.

\begin{figure}[h!]
\centering
\begin{tikzpicture}[
	x = {(-.7cm,-0.6cm)},
	y = {(.9cm,-.2cm)},
	z = {(0cm,1cm)},
	edge/.style={line width=1pt, line cap=round, black},
	edge2/.style={line width=0.4pt, line cap=round, black},
	edge3/.style={line width=0.8pt, line cap=round, black},
	face/.style={draw=none, opacity=0.2, fill=blue},
	scale = 1.2
	]

\coordinate (s) at (0,1,0.2);
\coordinate (t) at (0,0,-0.7);
\node[label = {\huge$-$}]  at ($1.7*(s)+(t)$) {};
\node[label = {\huge$-$}]  at ($4.7*(s)+(t)$) {};
\node[label = {\huge$+$}]  at ($7.7*(s)+(t)$) {};

\coordinate (v0) at (0,0,-.07);
\coordinate (v1) at (0,1,-.02);
\coordinate (v2) at (1,0,0);
\coordinate (v3) at (.95,.98,0);
\coordinate (v4) at ($0.1*(s)+(0,0,.35)$);
\coordinate (v5) at ($0.1*(s)+(0,0,-1.32)$);

\draw[edge2] (v0)--(v1);
\draw[edge] (v1)--(v3);
\draw[edge2] (v2)--(v0);
\draw[edge] (v3)--(v2);
\draw[edge2] (v0)--(v4);
\draw[edge] (v1)--(v4);
\draw[edge] (v2)--(v4);
\draw[edge] (v3)--(v4);
\draw[edge2] (v0)--(v5);
\draw[edge] (v1)--(v5);
\draw[edge] (v2)--(v5);
\draw[edge] (v3)--(v5);

\coordinate (s1) at ($3*(s)$);
\coordinate (p0) at ($(v0)+(s1)$);
\coordinate (p1) at ($(v1)+(s1)$);
\coordinate (p2) at ($(v2)+(s1)$);
\coordinate (p3) at ($(v3)+(s1)$);
\coordinate (p4) at ($(v4)+(s1)$);

\draw[edge2] (p0)--(p1);
\draw[edge] (p1)--(p3);
\draw[edge2] (p2)--(p0);
\draw[edge] (p3)--(p2);
\draw[edge2] (p0)--(p4);
\draw[edge] (p1)--(p4);
\draw[edge] (p2)--(p4);
\draw[edge] (p3)--(p4);

\coordinate (s2) at ($6*(s)$);
\coordinate (q0) at ($(v0)+(s2)$);
\coordinate (q1) at ($(v1)+(s2)$);
\coordinate (q2) at ($(v2)+(s2)$);
\coordinate (q3) at ($(v3)+(s2)$);
\coordinate (q4) at ($(v5)+(s2)$);

\draw[edge] (q0)--(q1);
\draw[edge] (q1)--(q3);
\draw[edge] (q2)--(q0);
\draw[edge] (q3)--(q2);
\draw[edge2] (q0)--(q4);
\draw[edge] (q1)--(q4);
\draw[edge] (q2)--(q4);
\draw[edge] (q3)--(q4);

\coordinate (s3) at ($9*(s)$);
\coordinate (u0) at ($(v0)+(s3)$);
\coordinate (u1) at ($(v1)+(s3)$);
\coordinate (u2) at ($(v2)+(s3)$);
\coordinate (u3) at ($(v3)+(s3)$);

\draw[edge] (u0)--(u1);
\draw[edge] (u1)--(u3);
\draw[edge] (u2)--(u0);
\draw[edge] (u3)--(u2);

\end{tikzpicture}
    \caption{An element of $\mathrm{Mat}_2\, ([4])$ that is valuatively equivalent to zero}
    \label{fig:zero}
\end{figure}
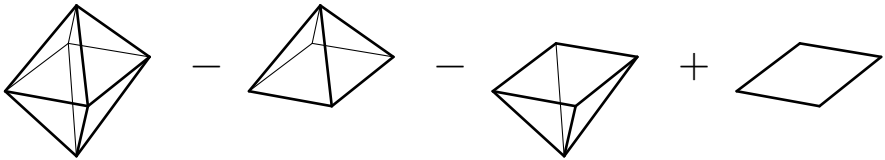

For the homological equivalence relation, we use the \emph{augmented Bergman fan} $\Sigma_\M$ of $\M$, which is an $r$-dimensional  simplicial fan in $\RR^E$ obtained by gluing together the order complex of the lattice of flats and the independence complex of $\M$. For an explicit description, see Definition~\ref{defn:augmentedBergman}.  
The augmented Bergman fan, introduced in \cite{BHMPW}, is a central object in the proof of the Dowling--Wilson top-heavy conjecture and the nonnegativity of the matroid Kazhdan--Lusztig polynomial \cite{BHMPW2}.  
The constant weight $1$ is balanced on the augmented Bergman fan, 
defining a Minkowski weight  $[\Sigma_\M]$ in the sense of \cite{FultonSturmfels}.
We review the definition of Minkowski weights and their identification with homology classes on toric varieties in Section~\ref{subsec:MW}.

\begin{definition}
An element $\sum_i c_i \M_i$ is said to be \emph{homologically equivalent to zero} if the Minkowski weight $\sum_i c_i [\Sigma_{\M_i}]$ is zero.
\end{definition}

For the numerical equivalence, we use the bilinear \emph{intersection pairing}
\[
\mathrm{Mat}_r(E) \times \mathrm{Mat}_{n-r}(E) \longrightarrow \ZZ, \quad (\M,\M') \longmapsto \deg(\M \wedge \M'),
\]
where the integer  $\deg(\M \wedge \M')$, for a rank $r$ matroid $\M$ and a rank $n-r$ matroid $\M'$ on $E$, is 
\[
\deg(\M \wedge \M') = \begin{cases}
1 & \text{if there are bases $B$ of $\M$ and $B'$ of $\M'$ such that $B \cap B' = \emptyset$},\\
0 & \text{if otherwise}.
\end{cases}
\]
We will identify this intersection pairing with an instance of the intersection product on the homology of a certain $n$-dimensional smooth projective variety; see Theorem~\ref{thm:intersectionaugmented} and Section \ref{SectionValuativeGroup}.

\begin{definition}
An element $\sum_i c_i \M_i$ is said to be \emph{numerically equivalent to zero} if it is in the kernel of the intersection pairing.
\end{definition}

Our first main result states that these three equivalence relations coincide.

\begin{theorem}\label{mainthm:equivalences}
The following conditions are equivalent for any $\eta \in \mathrm{Mat}_r(E)$.
\begin{enumerate}[label = (\arabic*)]\itemsep 5pt
\item\label{val} $\eta$ is valuatively equivalent to zero. 
\item\label{hom}  $\eta$ is homologically equivalent to zero. 
\item\label{num}  $\eta$ is numerically equivalent to zero.  
\end{enumerate}
\end{theorem}

We establish this equivalence via the combinatorics and algebraic geometry of the \emph{stellahedron} $\Pi_E$ of $E$, which is an $n$-dimensional simple polytope in $\RR^E$ with the following equivalent descriptions.
\begin{enumerate}[label=$\bullet$]\itemsep 5pt
\item The \emph{permutohedron} of $E$ is the convex hull of the permutations
\[
{\underline \Pi}_E\coloneq \operatorname{conv}\{ w \cdot (1,2,\ldots, n) \mid w \text{ is a permutation of $E$}\} \subseteq \RR^E.
\]
Writing $\mathbb{R}^E_{\ge 0}$ for the nonnegative orthant, the stellahedron of $E$ is
\[
\Pi_E = \big\{ u \in \RR_{\geq 0}^E \, \big| \, \text{ there exists $v \in {\underline \Pi}_E$ such that $v-u \in \RR_{\geq 0}^E$}\big\}. 
\]
This description shows that the permutohedron $\underline{\Pi}_E$ is the facet of $\Pi_E$ on which the standard inner product with $\be_E$ is maximized.
\item The \emph{independence polytope} of a matroid $\M$ is the convex hull
\[
I(\M) = \operatorname{conv}\{\be_I \, | \, I\subseteq B \text{ for some basis $B$ of $\M$}\} \subseteq \RR^E.
\]
Writing $\U_{r,E}$ for the uniform matroid of rank $r$ on $E$, whose bases are all size $r$ subsets of $E$, the stellahedron of $E$ is the Minkowski sum
\[
\Pi_E=\sum_{r=0}^n I(\U_{r,E}).
\]
This description shows that the standard $n$-dimensional simplex $I(\U_{1,E})$ and the standard $n$-dimensional cube $I(\U_{n,E})$ are Minkowski summands of the $n$-dimensional stellahedron $\Pi_E$.
Figure \ref{fig:stella} illustrates the case $E=[3]$.
\end{enumerate}
We remark that the stellahedron $\Pi_E$ is a realization of the graph associahedron of the star graph with the set of endpoints $E$; see for example \cite[\S10.4]{PRW08}.
 We refer to \cite{CD} and \cite{Devadoss} for discussions of graph associahedra and their realizations.\footnote{In \cite{FS05, PRW08, Postnikov}, an $n$-dimensional graph associahedron is realized as a generalized permutohedron in $\mathbb{R}^{n+1}$. For the star graph with the set of endpoints $E$, the stellahedron $\Pi_E$ and the projection of that graph associahedron to $\mathbb{R}^E$ have the same normal fan.}

\begin{figure}[h!]
\centering
\begin{tikzpicture}[
	x = {(-.7cm,-0.6cm)},
	y = {(.9cm,-.3cm)},
	z = {(0cm,1cm)},
	edge/.style={line width=1pt, line cap=round, black},
	edge2/.style={line width=0.4pt, line cap=round, black},
	edge3/.style={line width=0.8pt, line cap=round, black},
	face/.style={draw=none, opacity=0.2, fill=blue},
	scale = .75
	]

\coordinate (s) at (0,1,0.3);
\coordinate (t) at (0,0,-0.5);
\node[label = {\Huge$=$}]  at ($4.2*(s)+(t)$) {};
\node[label = {\huge$+$}]  at ($7.75*(s)+(t)$) {};
\node[label = {\huge$+$}]  at ($11.35*(s)+(t)$) {};

\coordinate (s0) at ($0*(0,-1,-0.3)$);
\coordinate (v'0) at ($.9*(0,0,0)$);
\coordinate (v'1) at ($.9*(0,0,2.85)$);
\coordinate (v'2) at ($.9*(0,2,3)$);
\coordinate (v'3) at ($.9*(1.1,2,3)$);
\coordinate (v'4) at ($.9*(2,1.2,3)$);
\coordinate (v'5) at ($.9*(2,0,3)$);
\coordinate (v'6) at ($.9*(0,3,2)$);
\coordinate (v'7) at ($.9*(0,2.95,0)$);
\coordinate (v'8) at ($.9*(1,3,2)$);
\coordinate (v'9) at ($.9*(1.92,3,0)$);
\coordinate (v'10) at ($.9*(1.92,3,1.05)$);
\coordinate (v'11) at ($.9*(2.9,0,0)$);
\coordinate (v'12) at ($.9*(3,0,2)$);
\coordinate (v'13) at ($.9*(3,1.2,2)$);
\coordinate (v'14) at ($.9*(3,2,0)$);
\coordinate (v'15) at ($.9*(3,2,1.05)$);

\coordinate (v0) at ($(v'0)+(s0)$);
\coordinate (v1) at ($(v'1)+(s0)$);
\coordinate (v2) at ($(v'2)+(s0)$);
\coordinate (v3) at ($(v'3)+(s0)$);
\coordinate (v4) at ($(v'4)+(s0)$);
\coordinate (v5) at ($(v'5)+(s0)$);
\coordinate (v6) at ($(v'6)+(s0)$);
\coordinate (v7) at ($(v'7)+(s0)$);
\coordinate (v8) at ($(v'8)+(s0)$);
\coordinate (v9) at ($(v'9)+(s0)$);
\coordinate (v10) at ($(v'10)+(s0)$);
\coordinate (v11) at ($(v'11)+(s0)$);
\coordinate (v12) at ($(v'12)+(s0)$);
\coordinate (v13) at ($(v'13)+(s0)$);
\coordinate (v14) at ($(v'14)+(s0)$);
\coordinate (v15) at ($(v'15)+(s0)$);

\draw[edge] (v1)--(v2);
\draw[edge] (v2)--(v3);
\draw[edge] (v3)--(v4);
\draw[edge] (v4)--(v5);
\draw[edge] (v5)--(v1);
\draw[edge] (v2)--(v6);
\draw[edge] (v6)--(v7);
\draw[edge] (v3)--(v8);
\draw[edge] (v6)--(v8);
\draw[edge] (v7)--(v9);
\draw[edge] (v8)--(v10);
\draw[edge] (v9)--(v10);
\draw[edge] (v5)--(v12);
\draw[edge] (v12)--(v11);
\draw[edge] (v4)--(v13);
\draw[edge] (v12)--(v13);
\draw[edge] (v13)--(v15);
\draw[edge] (v10)--(v15);
\draw[edge] (v11)--(v14);
\draw[edge] (v14)--(v9);
\draw[edge] (v14)--(v15);
\draw[edge2] (v1)--(v0);
\draw[edge2] (v11)--(v0);
\draw[edge2] (v7)--(v0);

\coordinate (s1) at ($6*(0,1,0.3)$);
\coordinate (a0) at (0,0,0);
\coordinate (a1) at (0,0,.98);
\coordinate (a2) at (0,1,0);
\coordinate (a3) at (1,0,0);

\coordinate (d0) at ($(a0) + (s1)$);
\coordinate (d1) at ($(a1) + (s1)$);
\coordinate (d2) at ($(a2) + (s1)$);
\coordinate (d3) at ($(a3) + (s1)$);

\draw[edge] (d1)--(d2);
\draw[edge] (d2)--(d3);
\draw[edge] (d1)--(d3);
\draw[edge2] (d1)--(d0);
\draw[edge2] (d2)--(d0);
\draw[edge2] (d3)--(d0);

	x = {(-.7cm,-0.6cm)},
	y = {(.9cm,-.3cm)},
	z = {(0cm,1cm)},
	edge/.style={line width=1.1pt, line cap=round, black},
	edge2/.style={line width=0.6pt, line cap=round, black},
	face/.style={draw=none, opacity=0.2, fill=blue},
	]
\coordinate (s2) at ($13*(0,1,0.3)$);
\coordinate (b'0) at (0,0,0);
\coordinate (b'1) at (0,0,0.98);
\coordinate (b'2) at (0,1,0);
\coordinate (b'3) at (1,0,0);
\coordinate (b'4) at (0,1,1);
\coordinate (b'5) at (.95,.98,0);
\coordinate (b'6) at (1,0,1);
\coordinate (b'7) at (1,1,1);

\coordinate (b0) at ($(b'0)+(s2)$);
\coordinate (b1) at ($(b'1)+(s2)$);
\coordinate (b2) at ($(b'2)+(s2)$);
\coordinate (b3) at ($(b'3)+(s2)$);
\coordinate (b4) at ($(b'4)+(s2)$);
\coordinate (b5) at ($(b'5)+(s2)$);
\coordinate (b6) at ($(b'6)+(s2)$);
\coordinate (b7) at ($(b'7)+(s2)$);

\draw[edge2] (b1)--(b0);
\draw[edge2] (b2)--(b0);
\draw[edge2] (b3)--(b0);
\draw[edge] (b1)--(b4);
\draw[edge] (b6)--(b7);
\draw[edge] (b4)--(b7);
\draw[edge] (b1)--(b6);
\draw[edge] (b3)--(b5);
\draw[edge] (b3)--(b6);
\draw[edge] (b5)--(b7);
\draw[edge] (b2)--(b4);
\draw[edge] (b2)--(b5);

\coordinate (s2) at ($9.5*(0,1,0.3)$);
\coordinate (c'0) at (0,0,0);
\coordinate (c'1) at (0,0,0.98);
\coordinate (c'2) at (0,1,0);
\coordinate (c'3) at (1,0,0);
\coordinate (c'4) at (0,1,1);
\coordinate (c'5) at (.95,.98,0);
\coordinate (c'6) at (1,0,1);

\coordinate (c0) at ($(c'0)+(s2)$);
\coordinate (c1) at ($(c'1)+(s2)$);
\coordinate (c2) at ($(c'2)+(s2)$);
\coordinate (c3) at ($(c'3)+(s2)$);
\coordinate (c4) at ($(c'4)+(s2)$);
\coordinate (c5) at ($(c'5)+(s2)$);
\coordinate (c6) at ($(c'6)+(s2)$);

\draw[edge2] (c1)--(c0);
\draw[edge2] (c2)--(c0);
\draw[edge2] (c3)--(c0);
\draw[edge] (c1)--(c4);
\draw[edge] (c1)--(c6);
\draw[edge] (c3)--(c5);
\draw[edge] (c3)--(c6);
\draw[edge] (c2)--(c4);
\draw[edge] (c2)--(c5);
\draw[edge3] (c4)--(c5);
\draw[edge3] (c6)--(c5);
\draw[edge3] (c6)--(c4);
\end{tikzpicture}
    \caption{The stellahedron of $[3]$ as the sum of three independence polytopes}
    \label{fig:stella}
\end{figure}
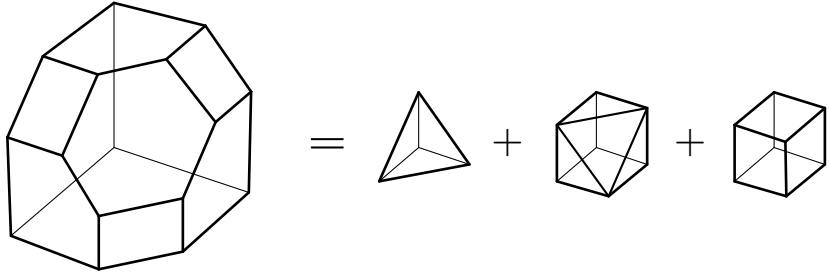

The \emph{stellahedral fan} $\Sigma_E$ is the normal fan of the stellahedron $\Pi_E$.  It is a simplicial fan that is unimodular with respect to the lattice $\ZZ^E \subseteq \RR^E$.  
The \emph{stellahedral variety} of $E$ is the associated smooth projective toric variety $X_E$. 
In this introduction, all varieties will be over the complex numbers. We follow the conventions of \cite{Ful93} and \cite{CLS11} for toric varieties.  The compact complex manifold $X_E$ is the central geometric object behind Theorem \ref{mainthm:equivalences}.

Let $T$ be the open torus $(\mathbb{C}^*)^E$ of the stellahedral variety $X_E$.
The two descriptions of the stellahedron have the following geometric consequences:
\begin{enumerate}[label=$\bullet$]\itemsep 5pt
\item 
The \emph{permutohedral variety} ${\underline X}_E$, the toric variety of the permutohedron ${\underline \Pi}_E$, admits a $T$-equivariant embedding
\[
\iota_E\colon\underline{X}_E \longrightarrow X_E,
\]
corresponding to the permutohedral facet $\underline{\Pi}_E$ of $\Pi_E$.
\item There is a birational toric morphism to the $n$-dimensional projective space 
\[
\pi_E\colon X_E \longrightarrow \PP^E,
\] 
corresponding to the Minkowski summand $I(\U_{1,E})$ of $\Pi_E$.
\item There is a birational toric morphism to the $n$-dimensional product of projective lines 
\[
\pi_{1^E}\colon X_E \longrightarrow (\PP^1)^E,
\]
 corresponding to the Minkowski summand $I(\U_{n,E})$ of $\Pi_E$.
\end{enumerate}
Summarizing, we have $T$-equivariant maps
\[
\begin{tikzcd}
{\underline X}_E \arrow[r, hook, "\iota_E"]& X_E \arrow[dr, "\pi_{1^E}"'] \arrow[dl, "\pi_E"] & \\
\mathbb{P}^E
& & (\mathbb{P}^1)^E.
\end{tikzcd}
\]
The image of ${\underline X}_E $ in $\mathbb{P}^E$ is the hyperplane at infinity $\mathbb{P}(\mathbb{C}^E)$, and the image of 
${\underline X}_E $ in $(\mathbb{P}^1)^E$ is the point $\infty^E$. 
Note that $\PP^E$ and $(\PP^1)^E$ are equivariant compactifications of the additive group $\CC^E$.   
In Section \ref{SectionStellahedral}, we observe
that the stellahedral variety $X_E$ is also a $\CC^E$-equivariant compactification of $\CC^E$, 
and that both maps to $\PP^E$ and $(\PP^1)^E$ are equivariant with respect to $\CC^E$.  

\begin{theorem}\label{thm:valuativeiso}
For every integer $r$, the assignment $\M \mapsto [\Sigma_\M]$ defines an isomorphism 
\[
\mathrm{Val}_r(E) \overset\sim\to H_{2r}(X_E,\ZZ)
\]
 from the valuative group of matroids on $E$ to the homology of the stellahedral variety of $E$.
\end{theorem}

Theorem~\ref{thm:valuativeiso} explains the coincidence of the valuative and the homological equivalence relations in 
Theorem~\ref{mainthm:equivalences}.
In Corollary~\ref{cor:DerksenFink}, we use Theorem~\ref{thm:valuativeiso} to give a geometric interpretation of a result of Derksen and Fink on a combinatorial basis of the valuative group  \cite{DerksenFink}.
The restriction of $[\Sigma_\M]$ to the permutohedral variety $\underline{X}_E$ is given by the Minkowski weight $[\underline{\Sigma}_\M]$, which is the constant balanced weight $1$ on the Bergman fan $\underline{\Sigma}_\M$ if the matroid is loopless and the constant balanced weight $0$ if otherwise.
Thus, Theorem~\ref{thm:valuativeiso} 
also recovers a result of Hampe that identifies the homology of ${\underline X}_E$ with the valuative group of loopless matroids \cite{Ham17}.

Poincar\'e duality for $X_E$ endows the homology of $X_E $ with the intersection product that is dual to the cup product on the cohomology of $X_E$.  We identify this intersection product with matroid intersection.  
Recall that the \emph{matroid intersection} of matroids $\M$ and $\M'$ on  $E$ is a matroid $\M \wedge \M'$ on $E$ whose bases are the minimal members of the family 
\[
\{B \cap B' \, | \, \text{$B$ is a basis of $\M$ and $B'$ is a basis of $\M'$}\}.
\]
  In particular, $\M  \wedge \M'$ has rank zero if and only if $\M$ and $\M'$ have bases $B$ and $B'$ that are disjoint.  Let us denote by $\operatorname{crk}(\M) = n-r$ the \emph{corank} of a rank $r$ matroid $\M$ on $E$.

\begin{theorem}\label{thm:intersectionaugmented}
The intersection product on $X_E$ satisfies
\[
[\Sigma_{\M}] \cdot [\Sigma_{\M'}] = \begin{cases}
[\Sigma_{\M \wedge \M'}] & \text{if $\operatorname{crk}(\M) + \operatorname{crk}(\M') = \operatorname{crk}(\M\wedge \M')$},\\
0 & \text{if otherwise}.
\end{cases}
\]
\end{theorem}

Theorem~\ref{thm:intersectionaugmented}, together with Poincar\'e duality for $X_E$, explains the coincidence of the homological and the numerical equivalence relations
 in Theorem~\ref{mainthm:equivalences}. 
By restricting to the permutohedral variety ${\underline X}_E$, we recover the following description of the intersection product on the homology of ${\underline X}_E$, previously established by Speyer in \cite[Proposition 4.4]{Spe08}. 

\begin{corollary}\label{cor:intersectbergman}
The intersection product on $\underline{X}_E$ satisfies
\[
 [\underline\Sigma_{\M}] \cdot  [\underline\Sigma_{\M'}] = \begin{cases}
 [\underline \Sigma_{\M \wedge \M'}] & \text{if $\M \wedge \M'$ is loopless}, \\ 0 &\text{if otherwise}.
\end{cases}
\]
\end{corollary}

Recall that a \emph{realization} of $\M$ over $\mathbb{C}$ is an $r$-dimensional linear subspace $L\subseteq \CC^E$ such that 
\[
\mathscr{B}=\big\{B\subseteq E \, \big| \, \text{the projection $\CC^E\twoheadrightarrow \CC^B$ restricts to an isomorphism $L\overset\sim\to \CC^B$}\big\}.
\]
The \emph{augmented wonderful variety} $W_L$ is the closure of $L$ in $X_E$.  
We show in Corollary \ref{cor:ctop} that the homology class of the augmented wonderful variety in the stellahedral variety is given by
\[
[W_L]=[\Sigma_\M] \in H_{2r}(X_E,\mathbb{Z}).
\]
The intersection of $W_L$ and $\underline{X}_E$ is the \emph{wonderful variety} $\underline{W}_L$ of de Concini and Procesi \cite{dCP95}, which is the closure of the projective hyperplane arrangement complement $\mathbb{P}(L) \cap (\mathbb{C}^*)^E/\mathbb{C}^*$ in $\underline{X}_E$. 
The main geometric objects behind the displayed identity and the proofs of Theorems~\ref{thm:valuativeiso} and \ref{thm:intersectionaugmented} are certain $T$-equivariant vector bundles on $X_E$ which we call ``augmented tautological bundles.'' 
For a linear subspace $L\subseteq \CC^E$, these are $T$-equivariant vector bundles $\cQ_L$ and $\cS_L$ on $X_E$ 
that have the following properties:
\begin{enumerate}[label=$\bullet$]\itemsep 5pt
\item The augmented wonderful variety $W_L$ is the vanishing locus of a distinguished global section of $\cQ_L$ (Theorem~\ref{thm:vanishingsection}). Consequently, the normal bundle $\mathcal N_{W_L/X_E}$ is isomorphic to the restriction 
of $\cQ_L$ to $W_L$ (Corollary~\ref{cor:koszul}).
\item The logarithmic tangent bundle $\mathcal T_{W_L}(-\log \partial W_L)$ of $W_L$, viewed as a compactification of $L=W_L\setminus \partial W_L$, is isomorphic to the restriction 
of $\cS_L$ to $W_L$ (Theorem~\ref{thm:logtangent}).
\end{enumerate}
See Definition~\ref{defn:augmentedTauto} for the construction of the augmented tautological bundles.  By restricting these bundles $\cQ_L$ and $\cS_L$ to the permutohedral variety $\underline X_E$, one recovers the ``tautological bundles'' $\underline \cQ_L$ and $\underline \cS_L$ (Definition~\ref{defn:tauto}) introduced in \cite{BEST21}.

In general, for an arbitrary matroid $\M$ with possibly no realization over $\mathbb{C}$, instead of vector bundles on $X_E$ we have $T$-equivariant $K$-classes $[\cQ_\M]$ and $[\cS_{\M}]$ on $X_E$.
These classes, which we call ``augmented tautological classes,'' 
satisfy the following properties:
\begin{enumerate}[label=$\bullet$]\itemsep 5pt
\item If $L \subseteq \mathbb{C}^E$ is a realization of $\M$, then $[\cQ_\M] = [\cQ_L]$ and $[\cS_\M] = [\cS_L]$ as $T$-equivariant $K$-classes (Proposition~\ref{prop:Kclasses}).
\item The assignments $\M\mapsto [\cQ_\M]$ and $\M \mapsto [\cS_\M]$ are both valuative maps from $\mathrm{Mat}_r(E)$ to the Grothendieck ring of $T$-equivariant vector bundles on $X_E$ (Proposition~\ref{prop:tautoval}).
\item By restricting $[\cQ_\M]$ and $[\cS_{\M}]$  to the permutohedral variety $\underline X_E$, one recovers the ``tautological classes of matroids'' $[\underline\cQ_\M]$ and $[\underline\cS_\M]$ introduced in \cite{BEST21}.
\end{enumerate}
The Chern classes of augmented tautological classes relate well to independence polytopes and augmented Bergman classes of matroids:
\begin{enumerate}[label=$\bullet$]\itemsep 5pt
\item Under the correspondence between base-point-free divisor classes on toric varieties and polytopes \cite[Section 6.2]{CLS11}, 
the first Chern class $c_1(\cQ_\M)$ of $[\cQ_\M]$ corresponds to the independence polytope $I(\M^\perp)$ of the dual $\M^\perp$ of $\M$.
\item  The top Chern class $c_{n-r}(\cQ_\M) \cap [X_E]$ of $[\cQ_\M]$ is the augmented Bergman class $[\Sigma_\M]$.
\end{enumerate}
The augmented tautological classes behave particularly well with respect to the following \emph{exceptional isomorphisms} between the Grothendieck ring of vector bundles $K(X_E)$ and the cohomology ring $H^{\bullet}(X_E,\mathbb{Z})$.
For any $K$-class $[\mathcal E]$, we write $c(\mathcal E)$ 
for its total Chern class and $[\det \mathcal E]$ for the $K$-class of its determinant line bundle.

\begin{theorem}\label{thm:exceptIsom} \
\begin{enumerate}[label = (\arabic*)]\itemsep 5pt
\item There is a unique ring isomorphism 
\[
\phi \colon K(X_E) \overset\sim\to H^{\bullet}(X_E,\mathbb{Z})
\]
that satisfies
$
\phi ( [\det \cQ_L] ) = c(\cQ_L)
$
for any linear subspace  $L\subseteq \CC^E$.
\item There is a unique ring isomorphism 
\[
\zeta \colon K(X_E) \overset\sim\to H^{\bullet}(X_E,\mathbb{Z})
\]
that satisfies
$
\zeta([\mathcal O_{W_L}]) = [W_L]$
for any linear subspace  $L\subseteq \CC^E$.
\end{enumerate}
\end{theorem}

Recall that the classical Hirzebruch--Riemann--Roch formula requires the use of rational coefficients.
We show that the isomorphisms $\phi$ and $\zeta$ satisfy the following Hirzebruch--Riemann--Roch-type formula with integer coefficients. 
We write the sheaf Euler characteristic map and the degree map by
\[
\chi \colon K(X_E)\to \ZZ \quad \text{and} \quad \int_{X_E} \colon H^{\bullet}(X_E,\mathbb{Z}) \to \ZZ.
\]
For each $i$ in $E$, let  
 $\pi_i\colon X_E \to \PP^1$ be the $i$-th factor of the map $\pi_{1^E} \colon X_E\to (\PP^1)^E$. 

\begin{theorem}\label{thm:fakeHRR}
For any $\xi \in K(X_E)$, the exceptional isomorphisms $\phi$ and $\zeta$ satisfy 
\[
\chi \big(\xi  \big) = \int_{X_E} \phi\big( \xi \big)\cdot c \big(\bigoplus_{i\in E}\pi_i^*\mathcal O_{\PP^1}(1)\big) = \int \zeta\big(\xi \big)\cdot c\big(\pi_E^*\mathcal O_{\PP^E}(-1)\big)^{-1}.
\]
\end{theorem}

Despite apparent similarities, these identities are not consequences of the classical Hirzebruch--Riemann--Roch theorem, since $\phi$ and $\zeta$ differ from the Chern character map. 
The integral classes $c \big(\bigoplus_{i\in E}\pi_i^*\mathcal O_{\PP^1}(1)\big)$ and $c\big(\pi_E^*\mathcal O_{\PP^E}(-1)\big)^{-1}$ play the role of the Todd class for $\phi$ and $\zeta$.
The isomorphisms $\phi$ and $\zeta$ 
 are closely related to the isomorphism $K({\underline X}_E) \overset\sim\to H^{\bullet}({\underline X}_E,\mathbb{Z})$ in \cite[Theorem D]{BEST21} in two different ways; see Remark~\ref{rem:compareHRR}.
 
We prove the existence of the isomorphisms in Theorem~\ref{thm:exceptIsom} in Section \ref{sec:exceptIsom}, and use it to prove Theorems~\ref{thm:valuativeiso} and \ref{thm:intersectionaugmented} in Section~\ref{subsec:polytopealgebra}.
The uniqueness of the isomorphisms in Theorem~\ref{thm:exceptIsom} is then derived from Theorem~\ref{thm:valuativeiso} in Section~\ref{subsec:polytopealgebra}.
We prove Theorem~\ref{thm:fakeHRR} in Section \ref{subsec:fakeHRR}.

Theorem~\ref{thm:fakeHRR} reveals remarkable numerical properties of the augmented tautological classes.
Recall that the \emph{Tutte polynomial} of a matroid $\M$ on $E$, introduced by Tutte \cite{Tut67} for graphs and by Crapo \cite{Cra69} for matroids, is the bivariate polynomial
\[
T_\M(x,y) = \sum_{S\subseteq E} (x-1)^{\rk_\M(E) - \rk_\M(S)}(y-1)^{|S| - \rk_\M(S)},
\]
where $\rk_\M \colon 2^E \to \ZZ$ here denotes the \emph{rank function} of $\M$. 
We give the following geometric interpretations of the Tutte polynomial as intersection numbers of the Chern and Segre classes of augmented tautological classes.
For a $K$-class $[\mathcal E]$ and a formal variable $u$, we set
\[
 c(\mathcal E, u) = \sum_{i} c_i(\mathcal E)u^i \quad \text{and} \quad s(\mathcal E, u) = \sum_{i} s_i(\mathcal E)u^i,
 \]
 where  $c_i(\mathcal E)$ is the $i$-th Chern class of $[\mathcal E]$ and $s_i(\mathcal E)$ is the $i$-th Segre class of $[\mathcal E]$.

\begin{theorem}\label{thm:tutterank}
For any rank $r$ matroid $\M$ on $E$, we have
\[
 T_\M(u+1, v+1)=
\int_{X_E} c(\mathcal S_\M,u) \cdot v^{n-r} \cdot c(\mathcal Q_\M,v^{-1}) \cdot c\big(\bigoplus_{i\in E} \pi_i^*\mathcal O_{\PP^1}(1)\big).
\]
\end{theorem}

Eliminating $\cS$ using $\cQ^\vee$, we get the following identity for the homogeneous polynomial
\[
t_\mathrm{M}(x,y,z,w)\coloneq (y + z)^{r} (x + w)^{n-r} T_\M\left( \frac{x + y}{y + z}, \frac{x + y + z + w}{x + w}\right).
\]

\begin{theorem}\label{thm:tutteintersection}
For any rank $r$ matroid $\M$ on $E$, we have
\[
t_\M(x,y,z,w)=\int_{X_E} s\big(\pi_E^*\mathcal O_{\PP^E}(-1), x\big) \cdot  c \big(\bigoplus_{i\in E} \pi_i^* \mathcal{O}_{\PP^1}(1), y \big) \cdot s(\mathcal{Q}_\M^{\vee}, z) \cdot c(\mathcal{Q}_\M, w). 
\]
\end{theorem}

The second formula implies the following 
analytic property of the Tutte polynomial.

\begin{theorem} \label{thm:tuttelorentzian}
For any rank $r$ matroid $\M$ on $E$, the polynomial $t_\M(x,y,z,w)$
is a denormalized Lorentzian polynomial in the sense of \cite{BH20, BLP}.
\end{theorem}

See Section \ref{subsec:positive} for a short review of Lorentzian polynomials, and see Remark~\ref{rmk:strengthen} for a strengthening of Theorem~\ref{thm:tuttelorentzian}.  If $\M$ has a realization $L\subseteq \CC^E$, Theorem \ref{thm:tuttelorentzian} follows from Theorem \ref{thm:tutteintersection} and the fact that the vector bundle $\cQ_L$ is globally generated.  For an arbitrary, not necessarily realizable, matroid $\M$, we establish Theorem~\ref{thm:tuttelorentzian} by constructing tropical models of augmented tautological classes, and then by applying tools from tropical Hodge theory as developed in \cite[Section 5]{ADH22}.

\begin{remark}
Consider the homogeneous polynomial 
 \[
 \underline{t}_\M(x,y,z,w) \coloneq (x + y)^{-1}(y + z)^{r} (x + w)^{n-r} T_\M\left( \frac{x + y}{y + z}, \frac{x + y}{x + w}\right).
 \]
In \cite[Theorems A and B]{BEST21}, the authors show the identity
\[
 \underline{t}_\M(x,y,z,w) =\int_{\underline{X}_E} s\big(\pi_E^*\mathcal O_{\mathbb{P}(\mathbb{C}^E)}(-1), x\big) \cdot c(\underline\cQ_{\U_{1,E}}, y) \cdot s(\underline{\mathcal{Q}}_\M^{\vee}, z) \cdot c(\underline{\mathcal{Q}}_\M, w) 
 \]
and show that this polynomial is a denormalized Lorentzian polynomial.  The authors do not know whether this result can be deduced directly from Theorem \ref{thm:tutteintersection} and \ref{thm:tuttelorentzian}, or vice versa.
\end{remark}

Specializing Theorem~\ref{thm:tuttelorentzian} by setting $x=1$, $y =0$, $z=q$, $w = 0$, we obtain the following corollary, which appeared in \cite[Problem 6.10]{Wagner98} and \cite[Conjecture 2]{SSV22} in the context of Postnikov--Shapiro algebras of graphs \cite{PostnikovShapiro}.

\begin{corollary}\label{cor:PostnikovShapiro}
For any rank $r$ matroid $\M$, the coefficients of the polynomial $q^{r} \, T_{\M}(q^{-1}, 1+q)$ form a log-concave sequence with no internal zeroes. 
\end{corollary}

We conclude with the study of the geometry of {matroid Schubert varieties} via augmented tautological bundles.  For a realization $L\subseteq \CC^E$ of a matroid $\M$, its \emph{matroid Schubert variety} $Y_L$ is the closure of $L$ in $(\PP^1)^E$.  
Matroid Schubert varieties play a central role in the proof of the Dowling--Wilson top-heavy conjecture in the realizable case \cite{HW},
and their intersection cohomologies are the main objects of study in the proof of the general case  \cite{BHMPW2}.
Matroid Schubert varieties satisfy several features analogous to those of classical Schubert varieties in flag varieties; see \cite{BHMPW2}.  Two such features are as follows:
\begin{enumerate}[label=$\bullet$]\itemsep 5pt
\item The map $\pi_{1^E}\colon X_E \to (\PP^1)^E$ restricts to a resolution of singularities $W_L \to Y_L$ for any $L \subseteq \mathbb{C}^E$.  The boundary $\partial W_L = W_L \setminus L$ is a simple normal crossings divisor on $W_L$. 
\item The standard affine paving of $(\PP^1)^E$ restricts to an affine paving of a matroid Schubert variety $Y_L$, whose $k$-dimensional cells are 
\[
U^F=\{p \in Y \, | \, \text{$p_i=\infty$ if and only if $i \notin F$}\},
\]
one for each rank $k$ flat $F$ of $\M$.
Writing $y_F$ for the homology class of the closure of $U^F$, which is another matroid Schubert variety,
we have
\[
H_\bullet(Y_L,\ZZ) \simeq \bigoplus_{F \in \mathscr{L}(\M)} \ZZ \, y_F,
\]
where $\mathscr{L}(\M)$ is the lattice of flats of $\M$.
\end{enumerate}
As mentioned before, 
the restriction of $\cS_L$ to the augmented wonderful variety $W_L$ is isomorphic to the log-tangent bundle $\mathcal T_{W_L}(-\log \partial W_L)$. 
This allows us to deduce the following remarkably simple formula for 
 the Chern--Schwartz--MacPherson (CSM) classes of matroid Schubert cells in their varieties.
See Section~\ref{ssec:CSM} for a brief review of CSM classes.

\begin{theorem}\label{thm:csm}
The Chern--Schwartz--MacPherson class of $\mathbf{1}_L$ in $Y_L$ is the sum over all flats
\[
c_{SM}(\mathbf 1_L) = \sum_{F \in \mathscr{L}(\M)} y_F \in H_\bullet(Y_L, \ZZ).
\]
\end{theorem}

In particular, the CSM class of $L$ in $Y_L$ is effective.
The analogous effectivity of CSM classes of classical Schubert cells in their varieties was established in \cite{aluffi2017shadows}.

We include an appendix that discusses notions of valuativity and polytope algebras. We mostly collect statements from the literature, but we also give an isomorphism between a certain polytope algebra and the $K$-ring of a smooth projective toric variety.

\subsection*{Notation}
Let $\kk$ be an algebraically closed field of arbitrary characteristic.  A variety is an irreducible and reduced scheme of finite type separated over $\kk$.  When $\kk = \CC$, the singular homology groups in even degrees and the Chow homology groups coincide for smooth projective toric varieties and augmented wonderful varieties, so we will use the two groups interchangeably in such cases, and similarly for the singular cohomology ring and the Chow cohomology ring.
We denote by $\langle \cdot, \cdot \rangle$ the standard pairing on $\kk^E$ or $\ZZ^E$.

\subsection*{Acknowledgements} We thank Alex Fink and Ravi Vakil for helpful conversations, Mario Sanchez for a helpful discussion on the proof of Lemma~\ref{lem:surjInd}, and the referees for their careful reading and suggestions.  The first author is partially supported by the US National Science Foundation (DMS-2001854).  The second author is partially supported by Simons Investigator Grant and NSF Grant DMS- 2053308.  The third author is supported by an NDSEG graduate fellowship.

\section{Torus-equivariant geometry preliminaries}
We collect some facts about the torus-equivariant $K$-ring and torus-equivariant Chow ring of a smooth projective toric variety.  The reader may skip this section and refer back as needed.

Let $X_{\Sigma}$ be the smooth projective toric variety with fan $\overline\Sigma$, and let $T = \mathbb{G}_m^E$ be the torus with character lattice $\operatorname{Char}(T) = \ZZ^E$.  Suppose that $T$ acts on $X_\Sigma$ via a surjective map of tori with connected kernel to the dense open torus of $X_\Sigma$, so that the corresponding map of cocharacter lattices is $\ZZ^E \to \ZZ^E/(\operatorname{lin}\cap \ZZ^E)$ for some linear subspace $\operatorname{lin}\subset \ZZ^E \otimes \RR$.
This data is encoded by the $n$-dimensional complete fan $\Sigma$ in $\RR^E$ with lineality space $\operatorname{lin}$ such that $\Sigma/\operatorname{lin} = \overline\Sigma$.

\subsection{Localization theorems}
Let $K_T(X_{\Sigma})$ be the $T$-equivariant $K$-ring of $X_{\Sigma}$, the Grothendieck ring of $T$-equivariant vector bundles on $X_{\Sigma}$. Let $K(X_{\Sigma})$ denote the $K$-ring of $X_{\Sigma}$. By forgetting the equivariant structure, one has a surjective map $K_T(X_{\Sigma}) \to K(X_{\Sigma})$.
By taking the $T$-equivariant sheaf Euler characteristic, one has a $K_T(\operatorname{pt})$-module homomorphism $\chi^T \colon K_T(X_{\Sigma}) \to K_T(\operatorname{pt})$.  We identify $K_T(\operatorname{pt}) = \ZZ[\operatorname{Char}(T)]$ with the Laurent polynomial ring $\mathbb{Z}[T_1^{\pm 1}, \ldots, T_n^{\pm 1}]$ where $T_i$ is the standard character of $i\in E$ under the identification $\operatorname{Char}(T) = \ZZ^E$.

Let $A_T^{\bullet}(X_{\Sigma})$ denote the equivariant Chow ring of $X_{\Sigma}$, as defined in \cite{EG1998}, and let $A^{\bullet}(X_{\Sigma})$ denote the Chow ring of $X_{\Sigma}$. Similar to the $K$-rings, one has a surjective map $A_T^{\bullet}(X_{\Sigma}) \to A^{\bullet}(X_{\Sigma})$ and a $A_T^{\bullet}(\operatorname{pt})$-module homomorphism $\int^T \colon A_T^{\bullet}(X_{\Sigma}) \to A_T^{\bullet}(\operatorname{pt})$.  We identify $A^\bullet_T(\operatorname{pt})$ with the polynomial ring $\ZZ[t_1, \ldots, t_n]$.
Let $\int \colon A^\bullet(X_\Sigma) \to \ZZ$ be the (non-equivariant) degree map.

Let $\Sigma(k)$ denote the set of cones of dimension $k$ of $\Sigma$. For each maximal cone $\sigma$ of $\Sigma$, we have a map $K_T(X_{\Sigma}) \to K_T(\operatorname{pt}_\sigma) = \mathbb{Z}[T_1^{\pm 1}, \ldots, T_n^{\pm 1}]$ given by pulling back to or \emph{localizing at} the corresponding fixed point $\operatorname{pt}_\sigma$. Similarly, we have a map $A_T^{\bullet}(X_{\Sigma}) \to A_T^\bullet(\operatorname{pt}_\sigma) = \mathbb{Z}[t_1, \ldots, t_n]$.
These maps can be combined into maps $K_T(X_{\Sigma}) \to K_T(X_{\Sigma}^T) = \prod_{\sigma \in \Sigma(n)} K_T(\operatorname{pt})$ and $A^{\bullet}_T(X_{\Sigma}) \to A_{T}^{\bullet}(X_{\Sigma}^T) = \prod_{\sigma \in \Sigma(n)} A_T^{\bullet}(\operatorname{pt})$, where $X_\Sigma^T$ denotes the set of $T$-fixed points of $X_\Sigma$. For a character $v = (v_1, \ldots, v_n) \in \ZZ^E$, we denote $T^v = T_1^{v_1} \dotsb T_n^{v_n}$ and $t_v = v_1 t_1+ \cdots+ v_nt_n$. Then we have the following localization theorem.

\begin{theorem}\label{thm:localization}
Let $X_{\Sigma}$ as above. Then
\begin{enumerate}[label = (\arabic*)]\itemsep 5pt
\item \label{localization:K} \cite[Corollary 5.11]{VezzosiVistoli} The restriction map $K_T(X_{\Sigma}) \to K_T(X_{\Sigma}^T)$ is injective, and its image is the subring of $\prod_{\sigma \in \Sigma(n)} K_T(\operatorname{pt})$ given by
$$\left\{f \in \prod_{\sigma \in \Sigma(n)} K_T(\operatorname{pt}) \ \middle| \ \begin{matrix} f_{\sigma} - f_{\sigma'}\equiv 0 \ \operatorname{mod} \ 1 - T^{v} \\ \text{ whenever }\dim \sigma \cap \sigma' = d - 1 \text{ with } \RR(\sigma \cap \sigma' ) = \ker v \end{matrix} \right\}.$$
Moreover, the map $K_T(X_\Sigma) \to K(X_\Sigma)$ forgetting the equivariant structure is surjective, with kernel $I_K$ equal to the ideal generated by $f - f(1, \ldots, 1)$ where $f$ is a global Laurent polynomial, i.e.,\ $f_\sigma$ for all $\sigma\in \Sigma(n)$ equals a common Laurent polynomial.
\item \label{localization:A} \cite{PayneCohomology} The restriction map $A_T^{\bullet}(X_{\Sigma}) \to A_T^{\bullet}(X_{\Sigma}^T)$ is injective, and its image is the subring of $\prod_{\sigma \in \Sigma(n)} A_T^{\bullet}(\operatorname{pt})$ given by
$$\left\{f \in \prod_{\sigma \in \Sigma(n)} A_T^{\bullet}(\operatorname{pt}) \ \middle| \ \begin{matrix} f_{\sigma} - f_{\sigma'}\equiv0 \ \operatorname{mod} \ {t_v} \\ \text{ whenever }\dim \sigma \cap \sigma' = d - 1 \text{ with } \RR(\sigma \cap \sigma') = \ker v \end{matrix} \right\}.$$
Moreover, the map $A_T^\bullet(X_\Sigma) \to A^\bullet(X_\Sigma)$ forgetting the equivariant structure is surjective, with kernel $I_A$ equal to the ideal generated by $f - f(0, \ldots, 0)$ where $f$ is a global polynomial, i.e.,\ $f_\sigma$ for all $\sigma\in \Sigma(n)$ equals a common polynomial.
\end{enumerate}
\end{theorem}

\subsection{Duality, rank, symmetric powers, exterior powers, Chern classes, and Segre classes}\label{subsec:equivConsts}
We now recall the description of several operations on the equivariant $K$-ring of a toric variety in terms of localization at fixed points. Let $[\mathcal{E}] \in K_T(X_{\Sigma})$ be an equivariant $K$-class, localizing to $[\mathcal{E}]_{\sigma} = \sum_{i=1}^{k_{\sigma}} a_{\sigma, i}T^{m_{\sigma, i}}$ at a torus-fixed point corresponding to a maximal cone $\sigma \in \Sigma(n)$. 

There is a ring involution $D_K$ on $K_T(X_{\Sigma})$ defined by sending the class of an equivariant vector bundle to the class of the dual vector bundle. The dual class $D_K([\mathcal{E}]) := [\mathcal{E}]^{\vee}$ has 
$$D_K([\mathcal{E}])_{\sigma} = \sum_{i=1}^{k_{\sigma}} a_{\sigma, i}T^{-m_{\sigma, i}}.$$
There is a corresponding ring involution, denoted $D_A$, on $A_T^{\bullet}(X_{\Sigma})$, defined by $D_A(t_i) \mapsto -t_i$ at each torus-fixed point. This multiplies by $(-1)^k$ on $A^k_T(X_{\Sigma})$. These involutions descend to $K(X_{\Sigma})$ and $A^{\bullet}(X_{\Sigma})$. 

As toric varieties are integral, every coherent sheaf on a toric variety has a rank. As the rank is additive in short exact sequences, this defines a ring homomorphism $\operatorname{rk} \colon K_T(X_{\Sigma}) \to \mathbb{Z}$, which descends to $K(X_{\Sigma}) \to \mathbb{Z}$. The rank of $[\mathcal{E}]$ is $\sum_{i=1}^{k_{\sigma}} a_{\sigma, i}$, which is independent of the choice of $\sigma$. 

The operation that assigns to each equivariant vector bundle its $j$-th symmetric or exterior power extends naturally to $K(X_{\Sigma})$ and $K_T(X_{\Sigma})$. Explicitly, with $u$ a formal variable, we have that
$$\sum_{j=0}^{\infty}  \textstyle\bigwedge^j \displaystyle[\mathcal{E}]_{\sigma} u^j = \prod_{i=1}^{k_{\sigma}}(1 + T^{m_{\sigma, i}}u)^{a_{\sigma, i}}, \text{ and } \sum_{j=0}^{\infty} \operatorname{Sym}^j[\mathcal{E}]_{\sigma}u^j = \prod_{i=1}^{k_{\sigma}} \left(\frac{1}{1 - T^{m_{\sigma, i}}u}\right )^{a_{\sigma, i}}.$$

The function that sends a vector bundle to its equivariant total Chern class extends to a function $c^T \colon K_T(X_{\Sigma}) \to A_T^{\bullet}(X_{\Sigma})$, which is multiplicative in the sense that $c^T(\mathcal E + \mathcal F) = c^T(\mathcal E) \cdot c^T(\mathcal F)$.
The equivariant Chern polynomial $c^T(\mathcal{E}, u)$ is the polynomial $c_0^T(\mathcal{E}) + c_1^T(\mathcal{E})u + c_2^T(\mathcal{E})u^2 + \dotsb$, where $u$ is a formal variable. Define similarly the Chern polynomial $c(\mathcal{E}, u) \in A^{\bullet}(X_{\Sigma})[u]$. 
The equivariant total Chern class localizes to
$$c^T(\mathcal{E}, u)_{\sigma} = \sum_{j=0}^{\infty} c_j^T(\mathcal{E})_{\sigma} u^j = \prod_{i=1}^{k_{\sigma}}(1 + ut_{m_{\sigma, i}})^{a_{\sigma, i}},$$
where $u$ is a formal variable. 

If $\mathcal{E}$ is a vector bundle on $X_{\Sigma}$, then $\mathcal{E}$ has a Segre class in $A^{\bullet}(X_{\Sigma})$, characterized by the property that $c(\mathcal{E}) s(\mathcal{E}) = 1$.  We define the \emph{equivariant Segre class} to be the inverse of $c^T(\mathcal{E})$ in $A^{\bullet}_T(X_{\Sigma})[c^T(\mathcal{E})^{-1}]$. Because $c(\mathcal{E})$ is a unit in $A^{\bullet}(X_{\Sigma})$,  there is a natural map $A^{\bullet}_T(X_{\Sigma})[c^T(\mathcal{E})^{-1}] \to A^{\bullet}(X_{\Sigma})$, and the image of $s^T(\mathcal{E})$ is $s(\mathcal{E})$. Define the (equivariant) Segre polynomial in the same way as the (equivariant) Chern polynomial.

\section{Stellahedral varieties}\label{SectionStellahedral}

We describe the stellahedral fan $\Sigma_E$ and its variety $X_E$ in several different ways, and we record several useful properties of $X_E$ we will need.  The closely related permutohedral fan ${\underline\Sigma}_E$ and its variety ${\underline X}_E$ will often appear and aid the discussion.

\subsection{The stellahedral fan via compatible pairs}\label{subsec:stellafan}

We describe the stellahedral fan in terms of its cones.
We start by describing the closely related permutohedral fan, which both serves as a motivation for and appears as a  substructure in the stellahedral fan.

\begin{definition}
The \emph{permutohedral fan} ${\underline\Sigma}_E$ is a fan in $\RR^E/\RR\be_E$ that consists of cones $\underline \sigma_{\mathscr F}$ for each chain $\mathscr F: F_1 \subsetneq \cdots \subsetneq F_k$ of nonempty proper subsets of $E$ where
\[
\underline\sigma_{\mathscr F} = \operatorname{cone}\{\overline\be_{F_1}, \ldots, \overline\be_{F_k}\}.
\]
Here we denoted $\overline u$ for the image of $u\in \RR^E$ in $\RR^E/\RR\be_E$.
\end{definition}

That this definition of ${\underline\Sigma}_E$ is equivalent to its description as the normal fan of the permutohedron ${\underline \Pi}_E =\operatorname{conv}\{ w \cdot (1,2,\ldots, n) \mid w \text{ is a permutation of $E$}\} \subseteq \RR^E$ is a standard fact about Coxeter reflection groups; see for instance \cite{BB05}.
We now give a similar description of the stellahedral fan $\Sigma_E$ in terms of ``compatible pairs'' as given in \cite[\S2]{BHMPW}.

\begin{definition}
A pair $(I,\mathscr  F)$ consisting of a subset $I\subseteq E$ and a chain $\mathscr F: F_1\subsetneq F_2 \subsetneq \cdots\subsetneq F_k$ of proper subsets of $E$ is said to be \emph{compatible} if $I$ is a subset of every element of $\mathscr F$.  We write $I\leq \mathscr  F$ in this case.
\end{definition}

Both the subset $I$ and the chain $\mathscr F$ are allowed to be empty.  In contrast to the permutohedral case, the empty set is allowed to be an element in the chain $\mathscr F$. Make the following a definition.

\begin{proposition}\label{prop:stellafan}\cite[Proposition 2.6]{BHMPW}
The stellahedral fan $\Sigma_E$ is a simplicial fan that consists of cones $\sigma_{I\leq\mathscr F}$ for each compatible pair $I\leq \mathscr F$ where
\[
\sigma_{I\leq \mathscr F} = \operatorname{cone}\{\be_i \mid i\in I\} + \operatorname{cone}\{-\be_{E\setminus F} \mid F\in \mathscr F\}.
\]
\end{proposition}

We denote the rays of the fan $\Sigma_E$ by
\[
\rho_i = \sigma_{\{i\}\leq \emptyset} = \operatorname{cone}(\be_i) \text{ for each $i\in E$} \quad\text{and}\quad \rho_S = \sigma_{\emptyset \leq \{S\}} = \operatorname{cone}(-\be_{E\setminus S}) \text{ for each $S\subsetneq E$}.
\]

The proposition gives the following corollary concerning the stars of the stellahedral fan.  Recall that for a fan $\Sigma$ in $\RR^E$, the \emph{star} of a cone $\sigma\in\Sigma$ is a fan, denoted $\operatorname{star}_\sigma\Sigma$, in $\RR^E/\RR\sigma$ whose cones are the images of the cones in $\Sigma$ containing $\sigma$.

\begin{corollary} \label{cor:stars}\cite[Proposition 2.7]{BHMPW}
Let $I = \{i_1, \ldots, i_j\} \leq \mathscr F: F_1 \subsetneq \cdots \subsetneq F_k$ be a compatible pair, and by convention set $F_{k+1} = E$ (so $F_1 = E$ if $\mathscr F$ is an empty chain).  Then, the isomorphism
\[
\RR^E / \RR\sigma_{I\leq \mathscr F} = \RR^E/\RR\{\be_{i_1},\ldots, \be_{i_j}, -\be_{E\setminus F_1}, \ldots, -\be_{E\setminus F_k}\} \simeq \RR^{F_1\setminus I} \times \prod_{i=1}^k \RR^{F_{i+1}\setminus F_i}/\RR\be_{F_{i+1}\setminus F_i}\]
induces an isomorphism of fans
\[
\operatorname{star}_{\sigma_{I\leq \mathscr F}}\Sigma_E \simeq \Sigma_{{F_1\setminus I}} \times \prod_{i = 1}^k \underline\Sigma_{{F_{i+1}\setminus F_i}}.
\]
\end{corollary}

\begin{example}\label{eg:permutostar}
When $(I,\mathscr F) = (\emptyset, \{\emptyset\})$ corresponding to the ray $\rho_\emptyset = \operatorname{cone}(-\be_E)$, we have that $\operatorname{star}_{\rho_\emptyset}\Sigma_E \simeq {\underline\Sigma}_E$.  In particular, we recover that the permutohedral variety ${\underline X}_E$ arise as the $T$-invariant divisor of $X_E$ corresponding to the ray $\rho_\emptyset$, as noted in the introduction.  From the map $\ZZ^E \to \ZZ^E/\ZZ\rho_\emptyset = \ZZ^E/\ZZ\be_E$, we have that the open dense torus of $\underline X_E$ is the projectivization $\PP T = (\kk^*)^E/\kk^*$ of $T$.
\end{example}

We will often use Example~\ref{eg:permutostar} to recover or relate the ``augmented'' structures on stellahedral varieties to the ``non-augmented'' versions on permutohedral varieties.  We will use the more general star structures of the stellahedral fan in \S\ref{subsec:basic}, where we study the restriction of augmented tautological bundles to various torus-invariant subvarieties of the stellahedral variety.

\subsection{Refinements and coarsenings}\label{subsec:refinecoarsen}

We record how the stellahedral fan $\Sigma_E$ arises as either a refinement or a coarsening of certain fans.  First, we note that $\Sigma_E$ is an iterated stellar subdivision of coarser fans in two distinguished ways.  Both statements can be verified via Proposition~\ref{prop:stellafan}.

\begin{proposition}\label{prop:iteratedblowup}
Let $\Sigma_E$ be the stellahedral fan of $E$.  The following hold. 
\begin{enumerate}[label = (\alph*)]\itemsep 5pt
\item Let $\Sigma_n$ be the fan in $\RR^E$ whose maximal cones are the cones generated by the cardinality-$n$ subsets of $\{\be_1, \be_2, \ldots, \be_{n}, -\be_E\}$.  Then $\Sigma_E$ is obtained from $\Sigma_n$ by performing the stellar subdivision of all maximal cones of $\Sigma_n$ that contain the vector $-\be_E$, then performing the stellar subdivision of the inverse images of codimension 1 cones that contain $-\be_E$, and so on.
\item Let $(\Sigma_1)^E$ be the fan in $\RR^E$ whose maximal cones are the $2^{n}$ orthants of $\RR^E$.  Then $\Sigma_E$ is obtained from $(\Sigma_1)^E$ by performing the stellar subdivision of the negative orthant, then performing the stellar subdivision of the codimension-1 faces of the negative orthant, and so on.
\end{enumerate}
\end{proposition}

Since the toric varieties of $\Sigma_E$ and $(\Sigma_1)^E$ are $\PP^E$ and $(\PP^1)^E$, respectively, the above two descriptions of $\Sigma_E$ can be rephrased to say that the stellahedral variety $X_E$ is an iterated blow-up along smooth centers from $\PP^E$ and from $(\PP^1)^E$.  The two maps $\pi_E\colon X_E \to \PP^E$ and $\pi_{1^E}\colon X_E \to (\PP^1)^E$ are the blow-down maps.
For $i\in E$, let $\pi_i\colon X_E \to \PP^1$ be the composition of $\pi_{1^E}$ with the projection to the $i$-th $\PP^1$.  These maps from $X_E$ to projective spaces  give the following distinguished divisor classes on $X_E$.

\begin{definition}\label{defn:alphayis}
With notations as above, we denote
\[
\alpha = \pi_E^*(\text{hyperplane class of $\PP^E$}) \quad\text{and}\quad y_i = \pi_i^*(\text{hyperplane class of $\PP^1$}).
\]
\end{definition}

We now describe the stellahedral fan $\Sigma_E$ as a coarsening of a permutohedral fan.  This description of $\Sigma_E$ will be useful for our discussion of the tropical geometry of augmented wonderful varieties in \S\ref{subsec:tropical} and for producing a basis for $\Sigma_E$ in \S\ref{subsec:schubert}.

\medskip
Denote by $\widetilde E = E\sqcup \{0\}$.  
Let $p$ be the isomorphism of lattices
\[
p\colon \ZZ^{\widetilde E} /\ZZ \be_{\widetilde E} \to \ZZ^E \textnormal{ given by } (a_0, a_1, \ldots, a_n) \mapsto (a_1 - a_0, \ldots, a_n - a_0).
\]
That is, for $S\subseteq \widetilde E$ we have $\overline \be_S\mapsto \be_S$ if $0\notin S$ and $\overline\be_S \mapsto -\be_{E\setminus S}$ if $0\in S$.
To show that the stellahedral fan $\Sigma_E$ of $E$ is the image under $p$ of a coarsening of the permutohedral fan $\underline\Sigma_{{\widetilde E}}$ of $\widetilde E$, we use the following notions from \cite{dCP95, FY04} in an equivalent formulation given in \cite[\S7]{Postnikov}.
A \emph{building set} is a collection $\cG$ of subsets of $\widetilde E$ such that $\{i\} \in \cG$ for any $i\in \widetilde E$, and if $S$ and $S'$ are in $\cG$ with $S\cap S' \neq \emptyset$ then so is $S\cup S'$.
The \emph{nested complex} $\mathcal N$ of a building set $\cG$ is a simplicial complex on vertices $\cG$ whose faces are collections $\{X_1, \ldots, X_k\} \subseteq \cG$ such that for every subcollection $\{X_{i_1}, \ldots, X_{i_\ell}\}$ with $\ell \geq 2$ consisting only of pairwise incomparable elements, one has $\bigcup_{j=1}^\ell X_{i_j} \notin \cG$.
When $\widetilde E\in \cG$, the set of cones
\[
\big\{\operatorname{cone}\{\overline\be_{X_1}, \ldots, \overline\be_{X_k}\}\subseteq \RR^{\widetilde E}/\RR\be_{\widetilde E} \mid \{X_1, \ldots, X_k\}\subseteq \cG\setminus\{\emptyset, \widetilde E\} \text{ a face of $\mathcal N$} \big\}
\]
is a smooth fan in $\RR^{\widetilde E}/\RR\be_{\widetilde E}$ that coarsens the permutohedral fan $\underline\Sigma_{{\widetilde E}}$.

\begin{proposition}\label{prop:stellacoarsen}
The collection 
$
\cG = \{S \cup 0 \mid S \subseteq E\} \cup E
$
is a building set whose fan projects isomorphically onto the stellahedral fan $\Sigma_E$ under $p$.
\end{proposition}

\begin{proof}
Both the facts that $\cG$ is a building set and that the faces of $\mathcal N$ are
$
\{S_1 \cup 0, \ldots, S_k\cup 0\} \cup I
$,
where $\emptyset \subseteq S_1 \subsetneq \cdots \subsetneq S_k \subseteq E$ and $\emptyset \subseteq I \subseteq S_1$, are straightforward to check.  The rest of the proposition follows from Proposition~\ref{prop:stellafan}.
\end{proof}

\subsection{Polymatroids}
A standard correspondence between polyhedra and divisors on toric varieties \cite[\S6.2]{CLS11} (see also \cite[\S2.4]{ACEP20}) states the following:
For a lattice polytope $Q$ and the toric variety $X_Q$ defined by its normal fan $\Sigma_Q$, the base-point-free torus-invariant divisors on $X_Q$ are in bijection with \emph{deformations} of $Q$, which are lattice polytopes whose normal fans coarsen $\Sigma_Q$.
We show that specializing this to the stellahedral variety $X_E$ gives a correspondence between the set of base-point-free divisor classes on $X_E$ and a family of polytopes called ``polymatroids'' introduced in \cite{edmonds1970submodular}.

\begin{definition}\label{defn:polymat}
For vectors $u,v\in \RR^E$, let us denote $u\geq v$ if $u-v\in \RR^E_{\geq 0}$.
A \emph{polymatroid} on $E$ is a nonempty polytope $P$ in the nonnegative orthant $\RR^E_{\geq 0}$ satisfying the following two properties:
\begin{enumerate}\itemsep 5pt
\item If $v\in \RR^E_{\geq 0}$ such that $u\geq v$ for some $u\in P$, then $v\in P$.
\item For any $v\in \RR^E_{\geq 0}$, every maximal $u\in P$ such that $u\leq v$ has the same coordinate sum $\langle u, \be_E \rangle$.
\end{enumerate}
An \emph{integral polymatroid} is a polymatroid whose vertices lie in $\ZZ^E$.
\end{definition}

We will use the following ``strong normality'' of integral polymatroids in the proof of Proposition~\ref{prop:orbitflag}.

\begin{proposition}\cite[Chapter 18.6, Theorem 3]{Wel76}\label{prop:strongnormal}
Let $P_1, \ldots, P_k$ be integral polymatroids on $E$. Then any lattice point $q\in \ZZ^E$ in the Minkowski sum $P_1 + \cdots + P_k$ is a sum $p_1 + \cdots + p_k$ of lattice points $p_i \in P_i\cap \ZZ^E$.  In particular, an integral polymatroid $P$ is a normal polytope.
\end{proposition}

This property of polymatroids implies that the closure of the image of the map
\[
T \to \PP^{|P_1 \cap \ZZ^E|-1} \times \cdots \times \PP^{|P_k \cap \ZZ^E|-1} \text{ defined by } t\mapsto ([t^m]_{m\in P_1\cap \ZZ^E}, \dots, [t^m]_{m\in P_k\cap \ZZ^E})
\]
is isomorphic to the toric variety of the normal fan of $P_1 + \cdots + P_k$.
For a general discussion of normal polytopes in toric geometry, see \cite[Chapter 2]{CLS11}.

To relate polymatroids to base-point-free divisor classes on $X_E$, we will need the following equivalent description of (integral) polymatroids.
A function $f \colon 2^E \to \RR$ with $f(\emptyset) = 0$ is said to be \emph{non-decreasing} and \emph{submodular} if
\begin{enumerate}\itemsep 5pt
\item[] (non-decreasing) $f(S)\leq f(S')$ whenever $S\subseteq S' \subseteq E$, and
\item[] (submodular) $f(S\cup S') + f(S\cap S') \leq f(S) + f(S')$ for all $S,S'\subseteq E$.
\end{enumerate}

\begin{theorem}\cite[(8)]{edmonds1970submodular}
Polymatroids on $E$ are in bijection with non-decreasing and submodular functions $f \colon 2^E \to \RR$ with $f(\emptyset) = 0$.  The bijection is given by
\begin{align*}
\text{a polytope $P$} \quad&\mapsto\quad f \colon 2^E \to \RR \text{ where } f_P(S) = \max\{ \langle u, \be_S\rangle \mid u\in P\} \text{ for $S\subseteq E$}\\
\text{a function } f \colon 2^E\to \RR \quad&\mapsto\quad P = \{u\in \RR_{\geq 0}^E \mid \langle \be_S, u \rangle \leq f(S) \text{ for all }S\subseteq E\}.
\end{align*}
A polymatroid $P$ is integral if and only if the function $f$ is $\ZZ$-valued.\footnote{In some previous works \cite{DerksenFink, cameron2018flag}, the terminology ``polymatroid''  refers to associating the polytope $\underline P = \{u\in \RR_{\geq 0}^E \mid \langle \be_S, u \rangle \leq f(S) \text{ for all proper $S\subsetneq E$ and } \langle \be_E, u \rangle = f(E)\}$ to a non-decreasing and submodular function $f$ with $f(\emptyset) = 0$.  Our polytope $P$ is equal to $\{u\in \RR_{\geq 0}^E \mid \text{there exists $v\in \underline P$ such that } v - u \in \RR_{\geq 0}^E\}$, and hence contains $\underline P$ as a face.}
\end{theorem}

\begin{example}\label{eg:indep}
The independence polytope $I(\M)$ of a matroid $\M$ is an integral polymatroid where the function $f$ is the rank function $\rk_\M$.  It follows that $\operatorname{rk}_\M$ is a non-decreasing and submodular function.  Conversely, the rank function characterization of matroids implies that an integral polymatroid contained in the Boolean cube $[0,1]^E$ is the independence polytope of a matroid.  See \cite{edmonds1970submodular} for details.
\end{example}

The following proposition implies that, up to translation, polymatroids are exactly the deformations of the stellahedron.

\begin{proposition}\label{prop:nefpolymatroid}
For a proper subset $\emptyset\subseteq S\subsetneq E$, let $D_S$ be the torus-invariant divisor on $X_E$ corresponding to the ray $\sigma_{\emptyset\leq \{S\}} = \operatorname{cone}(-\be_{E\setminus S})$ of $\Sigma_E$.  Let $[D_S]$ be its divisor class in $A^1(X_E)$.  Then the map defined by 
\[
(\text{integral polymatroid $P$ defined by $f \colon 2^E \to \ZZ$}) \mapsto \sum_{\emptyset\subseteq S \subsetneq E} f(E\setminus S)[D_S] \in A^1(X_E)
\]
is a bijection between the set of integral polymatroids on $E$ and the set of base-point-free divisor classes on $X_E$.
\end{proposition}

For the proof we will need the following consequence of Proposition~\ref{prop:stellafan}, which follows from \cite[Theorem 6.1.7]{CLS11}.

\begin{corollary}\label{cor:nonfaces} (cf.\ \cite[Proposition 2.10]{BHMPW})
A collection of rays in $\Sigma_E$ is a minimal collection of rays that do not form a cone in $\Sigma_E$ if and only if the collection is either
\[
\{\rho_i, \rho_S\} \text{ for $i\notin S \subsetneq E$} \qquad\text{or}\qquad \{\rho_S, \rho_{S'}\} \text{ for incomparable $S,S'\subsetneq E$}.
\]
\end{corollary}

\begin{proof}[Proof of Proposition~\ref{prop:nefpolymatroid}]
We begin by noting that the primitive vectors  in the rays of $\Sigma_E$ are $\{\be_i\mid i \in E\} \cup \{-\be_{E\setminus S} \mid S\subsetneq E\}$.  Because the cone spanned by $\{\be_i \mid i\in E\}$ is a maximal cone in $\Sigma_E$, the presentation of the class group $A^1(X_E)$ in terms of torus-invariant divisors, as given in \cite[Theorem 4.1.3]{CLS11}, implies that any divisor class $[D] \in A^1(X_E)$ can be written uniquely as $[D] = \sum_{S \subsetneq E} c_S [D_S]$ with $c_S \in \mathbb{Z}$.  Let us set $c_E = 0$ by convention, and let $D = \sum_{S \subsetneq E} c_S D_S$ be a divisor.  We now need check that the line bundle $\mathcal{O}_{X_E}(D)$ of the divisor $D$ on $X_E$ is base-point-free if and only if the function $f \colon 2^E \to \ZZ$ given by $S \mapsto c_{E \setminus S}$ defines a polymatroid on $E$.\\
\indent For this end, we will use a criterion for base-point-freeness on toric varieties in terms of piecewise linear functions.  Following the conventions of \cite{CLS11}, the divisor $D = \sum_{S \subsetneq E} c_S D_S$ corresponds to the piecewise linear function $\varphi_D$ on $\mathbb{R}^E$ defined by assigning the value $0$ to $\be_i$ for $i\in E$ and the value $-c_S$ to $-\be_{E\setminus S}$ for $S\subsetneq E$.  
Applying a criterion for base-point-freeness \cite[Theorem 6.4.9]{CLS11} to the stellahedral fan along with Corollary~\ref{cor:nonfaces}, one has that $\mathcal O_{X_E}(D)$ is base-point-free if and only if the following two conditions are satisfied:
\begin{enumerate}\itemsep 5pt
\item For $i\in E$ and a subset $S\subsetneq E$ not containing $i$, one has
\[
\varphi_D(\be_i - \be_{E\setminus S}) \geq \varphi_D(\be_i) + \varphi_D(-\be_{E\setminus S}).
\]
Equivalently, since $i\notin S$ implies that $\be_i - \be_{E\setminus S} = -\be_{E\setminus (S\cup i)}$, noting that $\varphi_D(\be_i) = 0$ and $-\varphi_D(-\be_{E\setminus S}) = c_S$ gives
\[
c_{S\cup i} \leq c_{S}.
\]
\item For incomparable proper subsets $S$ and $S'$ of $E$, one has
\[
\varphi_D(-\be_{E\setminus S}-\be_{E\setminus S'}) \geq \varphi_D(-\be_{E\setminus S})  + \varphi_D(-\be_{E\setminus S'}).
\]
Equivalently, since $-\be_{E\setminus S}-\be_{E\setminus S'} = -\be_{E\setminus (S\cap S')}-\be_{E\setminus (S\cup S')}$, and because $\varphi_D$ is linear on $\operatorname{cone}\{-\be_{E\setminus (S\cap S')},-\be_{E\setminus (S\cup S')}\}$, noting that $-\varphi_D(-\be_{E\setminus S}) = c_{S}$ gives
\[
c_{S\cap S'} + c_{S\cup S'} \leq c_{S} + c_{S'}.
\]
Here, note that when $S\cup S' = E$, our convention that $c_{E}=0$ is consistent because $\varphi_D(-e_{E\setminus E}) = \varphi_D(0) = 0$.

\end{enumerate}
In terms of the function $f\colon S\mapsto c_{E\setminus S}$, the first condition is equivalent to $f(S) \leq f(S\cup i)$, and the second condition is equivalent to $f(S\cup S') + f(S\cap S') \leq f(S) + f(S')$.
\end{proof}

For an integral polymatroid $P$, let $D_P = \sum_{S\subsetneq E} f(E\setminus S)D_S$ be the corresponding divisor on $X_E$.
Let $X_P$ be the toric variety of the normal fan of $P$, considered as a fan in $\RR^E$ so that $X_P$ is considered as a $T$-variety.  Note that $X_P$ may have dimension less than $n$, so the action of $T$ on $X_P$ may have a nontrivial kernel.

\begin{example}\label{eg:twomaps}
For any matroid $\M$, we have that the divisor $D_{I(\M)}$ induces a toric morphism $X_E \to X_{I(\M)}$.  In particular, we recover the two distinguished maps from $X_E$ in the introduction:
When $P$ is the simplex $I(\mathrm U_{1,E})$, whose normal fan is $\Sigma_n$, we obtain the map $\pi_E\colon X_E \to \PP^E$.   When $P$ is the boolean cube $I(\mathrm U_{n,E})$, whose normal fan is $(\Sigma_1)^E$, we obtain the map $\pi_{1^E}\colon X_E \to (\PP^1)^E$.
\end{example}

\subsection{Orbit-closure in a flag variety and additive-equivariance}\label{subsec:flagvar}

We have so far described the structure of $X_E$ as a toric variety, i.e., in terms of the $T$-action.  Here we show that $X_E$ admits an action by a larger group that contains the additive group $\mathbb{G}_a^E$.  Let us begin with the 1-dimensional case.

The multiplicative group $\mathbb{G}_m$ acts on the additive group $\mathbb{G}_a$ via $t \cdot b = tb$ for $t \in \mathbb{G}_m$ and $b \in \mathbb{G}_a$. Let $\mathbb{G} = \mathbb{G}_m \ltimes \mathbb{G}_a$ be semi-direct product.
Concretely, the groups $\mathbb{G}_m$, $\mathbb{G}_a$, and $\mathbb{G}$ embed into $GL_2$ as follows.
\begin{equation*}
\mathbb{G}_m, \mathbb{G}_a, \mathbb{G} \hookrightarrow GL_2 \quad\text{via} \quad
t \mapsto \begin{pmatrix}t & 0 \\ 0 & 1\end{pmatrix},\
b \mapsto \begin{pmatrix} 1 & b \\ 0 & 1\end{pmatrix},\
(t,b) \mapsto \begin{pmatrix} t & b \\ 0 & 1\end{pmatrix}.
\end{equation*}
We denote by $\V = \kk^2$ the resulting $\GG$-representation.  The group $\mathbb{G}$ thus acts on $\PP(V) = \mathbb{P}^1$ by
\[
(t,b) \cdot [x : y] = [tx + by : y]
\]
with two orbits $\{[x:1] \mid b\in \kk\}\simeq \AA^1$ and $\{[1:0]\}$, denoted $\{\infty\}$.
When we treat $\PP^1$ as the toric variety of the fan in $\RR^1$ consisting of the three cones $\{\RR_{\geq0}, \RR_{\leq 0}, \{0\}\}$, the orbit $\mathbb{A}^1_o$ is identified with the toric affine chart of $\mathbb{P}^1$ corresponding to $\RR_{\geq 0}$.
In particular, letting $D_{[0,1]}$ be the toric divisor on $\PP^1$ corresponding to the interval $[0,1]\subset \RR^1$,
we may identify $V = H^0(\mathbb{P}^1, \mathcal O_{\PP^1}(1))^\vee$ by giving $T$-linearization of $\mathcal O_{\PP^1}(1)$ as $\mathcal O_{\PP^1}(\infty) = \mathcal O_{\PP^1}(D_{[0,1]})$.

Let us now show that the stellahedral variety $X_E$ admits a $\GG^E$-action.  We do this by realizing $X_E$ as a $\GG^E$-orbit closure in a flag variety.  While there are several alternate ways to exhibit the $\GG^E$-action on $X_E$, as listed in Remark~\ref{rem:additivealts}, the orbit closure description will be useful for defining the augmented tautological bundles in the next section.

From the $\GG$-action on $V = \kk^2$, we endow $V^E \simeq \kk^E \oplus \kk^E$ with the $\GG^E$-action given by $(\mathbf{t}, \mathbf{b}) \mapsto \begin{pmatrix} \text{diag}(\mathbf{t}) & \text{diag}(\mathbf{b}) \\ 0 & I \end{pmatrix}$. Let $\Delta \colon \kk^E \to V^E$ be the diagonal embedding. 

\begin{proposition}\label{prop:orbitflag}
Let $\mathscr L= \{L_1\subseteq \cdots \subseteq L_\ell\}$ be a flag of linear subspaces of $\kk^E$ realizing matroids $\M_1, \ldots, \M_\ell$, and let $P$ be the polymatroid $I(\M_1) + \cdots + I(\M_\ell)$.
Then the $\mathbb{G}^E$-orbit closure of $[\Delta(\mathscr{L})]$ in $Fl(\dim(L_{1}), \dotsc,  \dim(L_{\ell}); V^E)$ is identified with $X_P$. 
\end{proposition}

\begin{proof}
We first consider the case when $\ell = 1$, so we are taking the $\mathbb{G}^E$-orbit closure of $[\Delta(L_1)]$ in $Gr(\dim(L_1); V^E)$. Let $A$ be a matrix whose rows form a basis for $L_1$, so the rows of $\begin{pmatrix} A & A \end{pmatrix}$ form a basis for $\Delta(L_1)$. Then the $\mathbb{G}^E$-action on $Gr( \dim(L_1); V^E)$ is given by
$$(\mathbf{t}, \mathbf{b}) \cdot \left[ \begin{pmatrix} A & A \end{pmatrix} \right] = \left[ \begin{pmatrix} A & A \end{pmatrix} \begin{pmatrix} \text{diag}(\mathbf{t}) & \text{diag}(\mathbf{b}) \\ 0 & I \end{pmatrix}^t \right] =  \left[ \begin{pmatrix} (\mathbf{t} + \mathbf{b})A & A\end{pmatrix} \right].$$
This implies that the $T$-orbit closure coincides with the $\mathbb{G}^E$-orbit closure. 

The normalization of $\overline{T \cdot [\Delta(L_1)]}$ is a toric variety, so it is defined over $\operatorname{Spec} \mathbb{Z}$. We may therefore consider the moment polytope of its complexification, which is given a polarization via the Pl\"{u}cker embedding of the Grassmannian. The vertices of the moment polytope are given by the $T$-weights of the non-zero maximal minors of $\begin{pmatrix} A & A \end{pmatrix}$, where $T$ acts by scaling the first $n$ columns. Every non-zero maximal minor of $\begin{pmatrix} A & A\ \end{pmatrix}$ is given by a subset $S_1$ of the first $n$ rows and a subset $S_2$ of the second $n$ rows such that $S_1 \sqcup S_2$ is a basis for $\M_1$. The $T$-weight of this minor is $\be_{S_1}$, so the moment polytope is $I(\M_1)$.

Let $S$ be the set of non-loops of $\M_1$. The vertices of $I(\M_1)$ generate the lattice $\mathbb{Z}^S$, which implies that the character lattice of the embedded torus in the normalization of $\overline{T \cdot [\Delta(L_1)]}$ is $\mathbb{Z}^S$. Every lattice point in $I(\M_1)$ is a vertex, so the restriction map $H^0(Gr(\dim(L_1); V^E); \mathcal{O}(1)) \to H^0(\overline{T \cdot [\Delta(L_1)]}, \mathcal{O}(1))$ is surjective. By Proposition~\ref{prop:strongnormal}, $\overline{T \cdot [\Delta(L_1)]}$ is projectively normal and therefore normal, so $\overline{T \cdot [\Delta(L_1)]}$ is isomorphic to $X_{I(\M_1)}$. 

We now treat the general case. There is an embedding $Fl(\dim(L_{1}), \dotsc, \dim(L_{\ell}); V^E) \hookrightarrow \prod_{i=1}^{\ell} Gr(\dim(L_i); V^E)$, and the computation above implies that the $T$-orbit closure of $[\Delta(\mathscr{L})]$ is also the $\mathbb{G}^E$-orbit closure. By Proposition~\ref{prop:strongnormal}, the Segre embedding of $\overline{T \cdot [\Delta(\mathcal{L})]}$ corresponds to the Minkowski sum of polytopes (with the complete linear series), which implies that the moment polytope of $\overline{T \cdot [\Delta(\mathscr{L})]}$ is $P$. Using that $P$ is a normal polytope, we get that $\overline{T \cdot[\Delta(\mathscr{L})]}$ is isomorphic to $X_{P}$.
\end{proof}

The flag of matroids realized by a general full flag $\mathscr{L} = \{L_1 \subsetneq L_2 \subsetneq \cdots \subsetneq L_n = \kk^E\}$ over an infinite field $\kk$ are exactly the uniform matroids $\U_{1,E}, \ldots, \U_{n,E}$.  Since the stellahedron $\Pi_E$ is the Minkowski sum $I(\U_{1, E}) + \dotsb + I(\U_{n, E})$, we have the following corollary.

\begin{corollary}\label{cor:stellaorbit}
The $\mathbb{G}^E$-orbit closure of a general full flag of linear subspaces $\mathscr L$, viewed as a point in $Fl(1, \ldots, n; V^E)$ via $\Delta$, is identified with $X_E$. In particular, $X_E$ has the structure of a $\mathbb{G}^E$-variety. 
\end{corollary}

\begin{remark}\label{rem:additivealts}
With $\PP^1$ as a $\GG$-variety described above,  $\mathbb{G}^E$ acts on $(\mathbb{P}^1)^E$ with $2^{n}$ orbits.  In \S\ref{subsec:refinecoarsen}, we described $X_E$ as the iterated blow-up of the strict transforms of the proper $\mathbb{G}^E$-orbit closures in increasing order of dimension. The functoriality of the blow-up then gives $X_E$ a $\mathbb{G}^E$-action, and the blow-down map $X_E \to (\mathbb{P}^1)^E$ is $\GG^E$-equivariant.  
\\
\indent Alternatively, one notes that $\PP^E$, viewed as the projective completion $\PP(\kk^E \oplus \kk)$ of $\kk^E$, is a $\GG^E$-equivariant compactification of $\kk^E$ with the obvious action of $\GG^E$.  The proper $\GG^E$-orbit closures in $\PP^E$ are then exactly the coordinate subspaces of $\PP^E$ contained in the hyperplane at infinity $\PP(\kk^E) \subseteq \PP^E$.  In \S\ref{subsec:refinecoarsen}, we described $X_E$ as the iterated blow-up of the strict transforms of these proper $\GG^E$-orbit closures in the increasing order of dimension.  Again, the functoriality of the blow-up gives $X_E$ a $\GG^E$-action with an equivariant blow-down map $X_E \to \PP^E$.
\\
\indent Lastly, one may also appeal to \cite[Theorem 3.4 \& 4.1]{Arzhantsev2017} to show that any toric variety $X_P$ of the normal fan $\Sigma_P$ of a polymatroid $P$ on $E$ admits a $\GG_a^E$-action that is compatible with the torus-action:  One verifies that $\{-\be_i \mid i\in E\}$ form a ``complete collection of Demazure roots'' of $\Sigma_P$ as defined in (\emph{loc.\ cit.}).
\end{remark}

\section{Augmented tautological bundles and classes}

\subsection{Well-definedness}

We now construct the augmented tautological bundles and augmented tautological classes.
Recall the notation $V^E = \kk^E \oplus \kk^E$.  Recall that for any polymatroid $P$ (such as an independence polytope), one has a $T$-equivariant map $X_E \to X_P$ because the normal fan $\Sigma_P$ coarsens $\Sigma_E$.
Let us prepare with the following trivial case.

\begin{lemma}\label{lem:O1s}
Consider the map $X_E \to Gr(n;V^E)$ obtained as the composition of $X_E \to X_{I(\U_{n,E})}$ with the map $X_{I(U_{n,E})} \to Gr(n;V^E)$ given by setting $\ell = 1$ and $L_1 = \kk^E$ in Proposition~\ref{prop:orbitflag}.  The pullback to $X_E$ of the tautological subbundle $\mathcal{S}$ on $Gr(n; V^E)$ is isomorphic to $\bigoplus_{i\in E} \pi_i^*\mathcal{O}_{\mathbb{P}^1}(-1)$, equipped with the unique $T$-linearization that is trivial on the $\GG^E$-orbit $\mathbb{A}^E \subseteq X_E$.  
\end{lemma}

\begin{proof}
By construction, the pullback of $\mathcal{S}$ to $X_E$ is a subbundle of $\mathcal{O}^{\oplus 2n}_{X_E}$, and $\bigoplus_{i\in E} \pi_i^*\mathcal{O}_{\mathbb{P}^1}(-1)$ (with the unique $T$-linearization that is trivial on $\mathbb{A}^E$) is a subbundle of $\mathcal{O}^{\oplus 2n}_{X_E}$ whose fiber over any point in $\mathbb{A}^E$ is the diagonal $\Delta(\kk^E)$. It follows from the construction of the map $X_E \to Gr(n; V^E)$ that the pullback of $\mathcal{S}$ has the fiber over any point of $\mathbb{A}^E$ equal to $\Delta(\kk^E)$; the result follows because we may check whether two subbundles of $\mathcal{O}^{\oplus 2n}_{X_E}$ are equal on a dense open subset.\\
\indent Alternatively, we had given $\mathcal O_{\PP^1}(1)$ the $T$-linearization as the line bundle $\mathcal O_{\PP^1}(D_{[0,1]})$, which is trivial on the $\GG$-orbit $\AA^1$ of $\PP(V)$.  This resulted in the identification of $V$ with $H^0(\mathbb{P}^1, \mathcal O_{\PP^1}(1))^\vee$.  Since $I(\U_{n,E}) = [0,1]^E$, we find that $(\PP^1)^E \simeq X_{I(\U_{n,E})} \to Gr(n;V^E)$ is the map induced by the $E$-fold product of the injection of vector bundles $\mathcal O_{\PP^1}(-1) \to \mathcal O_{\PP^1} \otimes V$.
\end{proof}

Given a linear subspace $L\subseteq \kk^E$, we now construct vector bundles fitting into a short exact sequence that is modeled after $0\to L \to \kk^E \to \kk^E/L \to 0$.  Because we would like at least one of the vector bundles to be globally generated, the vector bundles $\cS_L$ and $\cQ_L$ will be defined so that they fit into the short exact sequence
$
0 \to \cS_L \to \bigoplus_{i\in E} \pi_i^*\mathcal{O}_{\mathbb{P}^1}(1) \to \cQ_L \to 0
$
with $\bigoplus_{i\in E} \pi_i^*\mathcal{O}_{\mathbb{P}^1}(1)$ in the middle instead of $\bigoplus_{i\in E} \pi_i^*\mathcal{O}_{\mathbb{P}^1}(-1)$.  As a result, when we define the dual bundle $\cQ_L^\vee$, we are led to consider the orthogonal dual $L^\perp = (\kk^E/L)^\vee \subseteq \kk^E$ of the realization $L\subseteq \kk^E$ of a matroid $\M$, which realizes the dual matroid $\M^\perp$.

\begin{definition}\label{defn:augmentedTauto}
Let $L\subseteq \kk^E$ be a realization of a rank $r$ matroid $\M$ on $E$.
Setting $\ell=2$ and $L_1 = L^{\perp} \subseteq L_2 = \kk^E$ in Proposition~\ref{prop:orbitflag} supplies us with a map
\[
X_E \to X_{I(\M^\perp) + I(\U_{n,E})}\to Fl(n-r,n;V^E).
\]
Define the \emph{augmented tautological bundles} $\cS_L$ and $\cQ_L$ by
\begin{align*}
\cQ_L &=  \text{the dual of the pullback to $X_E$ of the tautological rank $n-r$ subbundle of $Fl(n-r,n;V^E)$}\\
\cS_L &= \text{the dual of the quotient bundle $\bigoplus_{i \in E} \pi_i^* \mathcal{O}(-1)/\mathcal{Q}_L^{\vee}$.}
\end{align*}
\end{definition}

That $\cQ_L^\vee$ is a subbundle of $\bigoplus_{i \in E} \pi_i^* \mathcal{O}(-1)$ follows from Lemma~\ref{lem:O1s} and the fact that Proposition~\ref{prop:orbitflag} supplies us with a commuting diagram
\[
\begin{tikzcd}
& &X_{I(\U_{n,E})} \ar[r,hook] & Gr(n;V^E)\\
&X_{E} \ar[r] \ar[ru]\ar[rd]&X_{I(\U_{n,E}) + I(\M^\perp)}\ar[u]\ar[d]  \ar[u]\ar[d]\ar[r,hook] &Fl(n-r,n;V^E) \ar[u]\ar[d]\\
& &X_{I(\M^\perp)} \ar[r,hook] & Gr(n-r;V^E).
\end{tikzcd}
\]

\begin{remark}\label{rem:SES}
By construction, we have a short exact sequence of $\mathbb{G}^E$-equivariant vector bundles 
$$0 \to \mathcal{S}_L \to \bigoplus_{i \in E} \pi_i^* \mathcal{O}_{\mathbb{P}^1}(1) \to \mathcal{Q}_L \to 0,$$
which, when restricted to the $\GG^E$-orbit $\mathbb{A}^E$, is canonically identified with
$$0 \to \mathcal{O}_{\mathbb{A}^E} \otimes L \to \mathcal{O}_{\mathbb{A}^E} \otimes \kk^E \to \mathcal{O}_{\mathbb{A}^E} \otimes \kk^E/L \to 0.$$
\end{remark}

For arbitrary matroids $\M$, we construct ($T$-equivariant) $K$-classes $[\cS_\M]$ and $[\cQ_\M]$ on $X_E$.  By Theorem~\ref{thm:localization}.\ref{localization:K}, the $T$-equivariant $K$-ring of $X_E$ is identified with a subring of the product ring $\prod_{\Sigma_E(n)} \ZZ[T_1^{\pm 1}, \dotsc, T_n^{\pm 1}]$.  So we will specify these classes by specifying their localization values at each torus-fixed point indexed by a maximal cone of $\Sigma_E$.

By Proposition~\ref{prop:stellafan}, the maximal cones of $\Sigma_E$ are in bijection with compatible pairs $I \leq \mathscr F$ where $\emptyset\subseteq I\subseteq E$ and $\mathscr F$ is a (possibly empty) maximal chain of proper subsets of $E$ containing $I$.  For a chain $\mathscr F$ containing $I$, write $\mathscr F/I$ for the new chain of subsets of $E\setminus I$ obtained by removing $I$ from each subset in the original chain.  A maximal chain $\mathscr F: \emptyset \subsetneq F_1 \subsetneq \cdots \subsetneq F_{n-1}$ orders the ground set by $F_1 < F_2\setminus F_1 < \cdots < E\setminus F_{n-1}$, and for each matroid $\M$ on $E$ we denote:
\begin{enumerate}[label = $\bullet$]\itemsep 5pt
\item $B_{\mathscr F}(\M)$ the minimal basis of $\M$ under the lexicographic ordering, and
\item $B^c_{\mathscr F}(\M)$ the complement of $B_{\mathscr F}(\M)$ in the ground set of $\M$.
\end{enumerate}

\begin{proposition}\label{prop:Kclasses}
For a matroid $\M$ on $E$, the \emph{augmented tautological classes} defined as
\begin{equation*}\begin{split}
[\mathcal S_\M]_{I \leq \mathscr F} &= \operatorname{rk}_\M(I)+ \sum_{i\in B_{\mathscr F/I}(\M/I)} T_i^{-1} \quad\text{and}\\
[\mathcal Q_\M]_{I \leq \mathscr F} &= |I| - \operatorname{rk}_\M(I) + \sum_{i\in B^c_{\mathscr F/I}(\M/I)} T_i^{-1}
\end{split}\end{equation*}
are well-defined $T$-equivariant $K$-classes on $X_E$.  Moreover, if $L$ is a realization of $\M$, then $[\cS_L] = [\cS_\M]$ and $[\cQ_L] = [\cQ_\M$].
\end{proposition}

\begin{proof}
First we check that $[\mathcal{Q}_L] = [\mathcal{Q}_\M]$. Then taking the case $L = \{0\}$ gives that 
$$[\bigoplus_{i \in E} \pi_i^* \mathcal{O}_{\mathbb{P}^1}(1)]_{I \le \mathscr{F}} = |I| + \sum_{i \in E \setminus I} T_i^{-1}.$$
As $[\mathcal{S}_L] + [\mathcal{Q}_L] = [\bigoplus_{i \in E} \pi_i^* \mathcal{O}_{\mathbb{P}^1}(1)]$, this implies that $[\mathcal{S}_L] = [\mathcal{S}_\M]$. 

Let $L \subseteq \kk^E$ be a subspace of dimension $r$. Note that the rank $n - r$ tautological subbundle $\mathcal{\mathcal{S}}$ on $Fl(n-r, n; V^E)$ is pulled back from the forgetful map $Fl(n-r, n; V^E) \to Gr(n-r; V^E)$. The image of the $T$-fixed point on $X_E$ corresponding to a maximal compatible pair $I \le \mathscr{F}$ is a $T$-fixed point $\mathrm{p}$ of $Gr(n-r; V^E)$ such that every non-zero Pl\"{u}cker has weight equal to the vertex of $I(\M^{\perp})$ on which any functional in the interior of $\sigma_{I \le \mathscr{F}}$ attains its minimum, which is $\be_{B^c_{\mathscr{F}/I}(\M/I)}$. 
Then
$$[\mathcal{S}]_{\mathrm{p}} = |I| - \rk_\M(I) + \sum_{i \in B^c_{\mathscr{F}/I}(\M/I)} T_i \in K_T(\mathrm{p}).$$
As pullbacks commute with each other, this implies that $[\mathcal{Q}_L^{\vee}]_{I \le \mathscr{F}} = [\mathcal{S}]_{\mathrm{p}} = |I| - \rk_\M(I) + \sum_{i \in B^c_{\mathscr{F}/I}(\M)} T_i$, so applying $D_K$ gives that $[\mathcal{Q}_L] = [\mathcal{Q}_\M]$. In particular, it gives the claimed formula for $[\bigoplus_{i \in E} \pi_i^* \mathcal{O}_{\mathbb{P}^1}(1)] = [\mathcal{Q}_{\{0\}}]$.

Now we check well-definedness. As $[\mathcal{S}_\M] + [\mathcal{Q}_\M] = [\bigoplus_{i \in E} \pi_i^* \mathcal{O}_{\mathbb{P}^1}(1)]$, it suffices to check that $[\mathcal{S}_\M]$ is well-defined. There are two types of codimension $1$ cones in $\Sigma_E$. The first type is given by a compatible pair $I \le \mathscr{F}$ where $I = {F}_1$ and there is some $\ell$ such that ${F}_{\ell + 1} \setminus F_{\ell} = \{i, j\}$. This cone is contained in the kernel of the functional $\be_i - \be_j$. Let $\sigma_{I \le \mathscr{F}_1}$ and $\sigma_{I \le \mathscr{F}_2}$ be the two maximal cones containing $\sigma_{I \le \mathscr{F}}$; they are obtained by inserting either $F_{\ell} \cup i$ or $F_{\ell} \cup j$ into $\mathscr{F}$. Because the normal fan of $I(\M^{\perp})$ coarsens $\Sigma_E$, the vertices of $I(\M^{\perp})$ that functionals in the interiors of $\sigma_{I \le \mathscr{F}_1}$ and $\sigma_{I \le \mathscr{F}_2}$ attain their minimum on are either identical or differ by an edge. Because $\sigma_{I \le \mathscr{F}_1}$ and $\sigma_{I \le \mathscr{F}_2}$ have the same ``$I$,'' this edge must be parallel to $\be_i - \be_j$, and so the symmetric difference of $B_{\mathscr{F}_1/I}(\M/I)$ and $B_{\mathscr{F}_2/I}(\M/I)$ is either $\{i, j\}$ or $\emptyset$. This implies that, along $\sigma_{I \le \mathscr{F}}$, $[\mathcal{S}_\M]$ satisfies the condition of Theorem~\ref{thm:localization}.

The second type of codimension $1$ cone is given by a compatible pair $I \le \mathscr{F}$ when $I \cup j = F_1$, which is contained in the kernel of $\be_j$. Then the maximal cones containing $\sigma_{I \le \mathscr{F}}$ are $\sigma_{I \cup j \le \mathscr{F}}$ and $\sigma_{I \le \tilde{\mathscr{F}}}$, where $\tilde{\mathscr{F}}$ is obtained by adding $I$ to $\mathscr{F}$. Then a similar argument to the first case shows that $B_{\mathscr{F}/I \cup j}(\M/I \cup j)$ and $B_{\tilde{\mathscr{F}}/I}(\M/I)$ either coincide or differ by $\{j\}$. 
\end{proof}

These augmented tautological bundles and classes are related to the non-augmented tautological bundles and classes introduced in \cite{BEST21} as follows. Endow $\mathcal{O}_{\underline{X}_E}^{\oplus E}$ with the inverse $T$-equivariant structure, i.e., $(t_1, \dotsc, t_n) \cdot (x_1, \dotsc, x_n) = (t_1^{-1} x_1, \dotsc, t_n^{-1} x_n)$.

\begin{definition}\label{defn:tauto}
Let $L \subseteq \kk^E$ be a realization of a matroid $\M$. Then the (non-augmented) \emph{tautological bundles} $\underline{\mathcal{S}}_L$ and $\underline{\mathcal{Q}}_L$ are the unique $T$-equivariant vector bundles on $\underline{X}_E$ that fit into a short exact sequence 
$$0 \to \underline{\mathcal{S}}_L \to \mathcal{O}_{\underline{X}_E}^{\oplus E} \to \underline{\mathcal{Q}}_L \to 0$$
where the fiber over the identity is identified with
$$0 \to L \to \kk^E \to \kk^E/L \to 0.$$
\end{definition}

One can show that the short exact sequence in the above definition is the restriction to $\underline X_E$ of the short exact sequence $0 \to \cS_L \to \bigoplus_{i\in E} \pi_i^* \mathcal{O}_{\mathbb{P}^1}(1) \to \cQ_L \to 0$.

For each matroid $\M$, the authors of \cite{BEST21} define classes $[\underline{\mathcal{S}}_\M]$ and $[\underline{\mathcal{Q}}_\M]$ in $K_T(\underline{X}_E)$. The $T$-fixed points on $\underline{X}_E$ are in bijection with complete flags $\mathscr{F}$ of subsets of $E$. The tautological classes are described by 
$$[\underline{\mathcal{S}}_\M]_{\mathscr{F}} = \sum_{i \in B_{\mathscr{F}}(\M)} T_i^{-1} \quad \text{ and }  \quad[\underline{\mathcal{Q}}_\M]_{\mathscr{F}} = \sum_{i  \in  B_{\mathscr{F}}^c(\M)} T_i^{-1}.$$
In particular, these are restrictions to $\underline X_E$ of the augmented tautological classes $[\cS_\M]$ and $[\cQ_\M]$.

\subsection{Basic properties}\label{subsec:basic}
We now develop some basic properties of augmented tautological classes. These properties and their proofs are similar to those considered in \cite[Section 5]{BEST21}.

\begin{proposition}\label{prop:firstChern}
For a matroid $\M$, we have that $[\det \cQ_\M]$ equals the $K$-class of the line bundle corresponding under Proposition~\ref{prop:nefpolymatroid} to the polymatroid $I(\M^\perp)$.
\end{proposition}

\begin{proof}
As a $T$-equivariant $K$-class, we have from Proposition~\ref{prop:Kclasses} that
\[
[\det \cQ_\M]_{I\leq \mathscr F} = \prod_{i\in B_{\mathscr F/I}^c(\M/I)}T_i^{-1}
\]
for a maximal cone $\sigma_{I\leq \mathscr F}$ of $\Sigma_E$.  Since the vertex of $I(\M^\perp)$ that minimizes the pairing with a vector in the interior of $\sigma_{I\leq \mathscr F}$ is $\be_{B_{\mathscr F/I}^c(\M/I)}$, the result follows.\\
\indent Alternatively, by appealing to Proposition~\ref{prop:tautoval} one can reduce to the case where $\M$ admits a realization $L$, in which case the diagram above Remark~\ref{rem:SES} implies that $\det\cQ_L$ defines the map $X_E \to X_{I(\M^\perp)}$ given by the line bundle $\mathcal O_{X_E}(D_{I(\M^\perp)})$.
\end{proof}

\begin{proposition}\label{prop:tautoval}
Any function that maps a matroid $\M$ to a fixed polynomial expression involving symmetric powers, exterior powers, tensor products, and direct sums of $[\mathcal{S}_\M], [\mathcal{Q}_\M], [\mathcal{S}_\M]^{\vee},$ and $[\mathcal{Q}_\M]^{\vee}$ is valuative, and similarly for a fixed polynomial expression in the Chern classes of the augmented tautological classes. 
\end{proposition}

For instance, the proposition implies that the assignments $\M\mapsto c(\cQ_\M)$ and $\M \mapsto s(\cQ_\M^\vee)$ are valuative.

\begin{proof}
Let $\ZZ^{2^E}$ be the free abelian group with the standard basis indexed by the subsets of $E$.  Consider the function
\[
\operatorname{Mat}(E) \to \bigoplus_{\Sigma_E(n)} \mathbb{Z}^{2^{E}} \text{ given by } \M \mapsto \sum_{\sigma_{I \le \mathscr{F}} \in \Sigma_E(n)} \be_{B_{\mathscr{F}/I}(\M/I)}.
\]
By Proposition~\ref{prop:faceval}, this function is valuative; see also \cite[Theorem 5.4]{AFR10}. Any fixed polynomial expression in the augmented tautological classes or their Chern classes factors through this map and is therefore valuative. 
\end{proof}

We now consider how augmented tautological classes restrict to $T$-invariant subvarieties of $X_E$.
By Corollary~\ref{cor:stars}, for a (not necessarily maximal) compatible pair $I \leq \mathscr F: F_1 \subsetneq \cdots \subsetneq F_k$, the corresponding $T$-invariant subvariety $Z_{I \leq \mathscr F} \subseteq X_E$ corresponding to the cone $\sigma_{I\leq \mathscr F}$ is naturally identified with
\[
Z_{I\leq\mathscr F} \simeq X_{F_1 \setminus I} \times \prod_{i = 1}^k \underline{X}_{{F_{i+1}\setminus F_i}}.
\]
This identification then induces isomorphisms
$$K_T(Z_{I\leq\mathscr F}) \stackrel{\sim}{\to} K_T(X_{F_1 \setminus I}) \otimes \bigotimes_{i=1}^{k} K_T(\underline{X}_{{F_{i+1}\setminus F_i}}) \text{ and } A_T^{\bullet}(Z_{I\leq\mathscr F}) \stackrel{\sim}{\to} A_T^{\bullet}(X_{F_1 \setminus I}) \otimes \bigotimes_{i=1}^{k} A_T^{\bullet}(\underline{X}_{{F_{i+1}\setminus F_i}}).$$

\begin{proposition}\label{prop:restrict}
Under the above identification, we have that
\begin{equation*}
\begin{split}
[\mathcal S_\M]|_{Z_{I\leq \mathscr F}} &= \operatorname{rk}_\M(I)[\mathcal O_{Z_{I\leq \mathscr F}}] + [\mathcal S_{\M|F_1/I}] \otimes 1^{\otimes k} + \sum_{i = 1}^k 1^{\otimes (i-1)} \otimes [\underline {\mathcal S}_{\M|F_{i+1}/F_i}] \otimes 1^{\otimes (k-i)},\quad\text{and}\\
[\mathcal Q_\M]|_{Z_{I\leq \mathscr F}} &= (|I| -\operatorname{rk}_\M(I))[\mathcal O_{Z_{I\leq \mathscr F}}] + [\mathcal Q_{\M|F_1/I}]\otimes 1^{\otimes k} + \sum_{i = 1}^k 1^{\otimes (i-1)} \otimes [\underline {\mathcal Q}_{\M|F_{i+1}/F_i}] \otimes 1^{\otimes (k-i)}.
\end{split}
\end{equation*}
In particular, when $\mathcal  F = \emptyset$, we have that $c(\mathcal S_\M)|_{Z_I} \simeq c(\mathcal S_{\M/I})$ as a class in $A^\bullet(Z_I) \simeq A^\bullet(X_{{E\setminus I}})$, and similarly for $\mathcal Q_\M$.
\end{proposition}

\begin{proof}
The fan of $Z_{I \le \mathscr{F}}$ is the star of $\sigma_{I \le \mathscr{F}}$, and the localization of an augmented tautological class to a $T$-fixed point of $Z_{I \le \mathscr{F}}$ is the same as the localization to the $T$-fixed point of $X_E$ at the corresponding maximal cone of $\Sigma_E$. \\
\indent The face of $I(\M^{\perp})$ on which functionals in the (relative) interior of $\sigma_{I \le \mathscr{F}}$ attain their minimum is naturally identified with $I((\M|F_1/I)^{\perp}) \times \prod_{i=1}^{k} P((\M|{F_{i+1}}/F_i)^{\perp})$, and this identification is compatible with the corresponding identification for $\Pi_E$. As the localizations of augmented tautological classes to a fixed point corresponding to a maximal cone of $\Sigma_{E}$ depend only on vertex of $I(\M^{\perp})$ on which any functional in the interior of that maximal cone attains its minimum, this product decomposition gives the result. 
\end{proof}

\section{Augmented wonderful varieties and Bergman classes}

\subsection{Augmented wonderful varieties}

\begin{definition}
Let $L\subseteq \kk^E$ be a linear subspace.  With $\kk^E$ identified with the toric affine chart of $X_E$ corresponding to the cone $\sigma_{E\leq\emptyset} = \RR_{\geq 0}^E$ of $\Sigma_E$, the \emph{augmented wonderful variety} $W_L$ of $L$ is defined as the closure of $L$ in $X_E$.
\end{definition}

We note an equivalent description of the augmented wonderful variety, which can be deduced from Proposition~\ref{prop:iteratedblowup}. For a flat $F\subseteq E$ of $\M$, let $L_F = L \cap (\kk^{E\setminus F} \oplus 0^F)$.  The projective completion $\mathbb{P}(L \oplus \kk)$ of $L$ contains a copy of $\mathbb{P}(L)$ as the hyperplane at infinity, and so it contains a subspace identified with $\mathbb{P}(L_F)$ for every flat $F$ of $\M$.
Under the iterated blow-up $\pi_E\colon X_E \to \PP^E$, the augmented wonderful variety $W_L$ is the strict transform of $\PP(L \oplus \kk) \subseteq \PP(\kk^E \oplus \kk) = \PP^E$, fitting into the diagram
\[
\begin{tikzcd}
&W_L \ar[r, hook] \ar[d] &X_E \ar[d]\\
&\PP(L\oplus \kk) \ar[r, hook] &\PP^E.
\end{tikzcd}
\]
This makes $W_L$ equal to the variety obtained by blowing up $\mathbb{P}(L \oplus k)$ at the linear spaces $\PP(L_F)$ corresponding to corank $1$ flats of $\M$, then blowing up at the strict transforms of linear spaces corresponding to corank $2$ flats of $\M$, and so on.

We relate augmented wonderful varieties to augmented tautological bundles as follows.

\begin{theorem}\label{thm:vanishingsection}
For a linear subspace $L\subseteq \kk^E$, the augmented wonderful variety $W_L$ is the vanishing locus of a distinguished global section of $\mathcal{Q}_L$.
\end{theorem}

We prepare to prove Theorem~\ref{thm:vanishingsection} with the following lemma.

\begin{lemma}
Let $\cQ$ be a vector bundle of rank $k$ on a smooth variety $X$, and let $L \subseteq H^0(X, \mathcal{Q})$ be a subspace which generates $\mathcal{Q}$. Suppose there exists a nonempty open $U\subseteq X$ such that for a general $s \in L$, the vanishing locus $V(s)$ is nonempty and the intersection $V(s)\cap U$ is integral of codimension $k$.  Then $V(s)$ is integral for a general $s \in L$.
\end{lemma}

\begin{proof}
Once we show that $V(s)$ is irreducible, the unmixedness theorem \cite[Corollary 18.14]{Eis95} implies that $V(s)$, which is of codimension $k$, has no embedded points, and hence is integral.  To show that $V(s)$ is irreducible, let $\cS$ be the kernel of $\mathcal O_X \otimes L \twoheadrightarrow \cQ$, and let $\AA(\cS)$ be the total space of $\cS$, which is irreducible.
We consider the map $\pi: \AA(\cS) \to X\times L \to L$.  For $s\in L$, the fiber $\pi^{-1}(s)$ is isomorphic to the vanishing locus $V(s)$.  Since $V(s)$ is nonempty for a general $s$, the map $\pi$ is a dominant map between varieties, and hence a general fiber of $\pi$ is pure-dimensional.  Now, let $Z$ be the total space of the restriction of $\cS$ to the closed subvariety $X\setminus U$.  Since $\dim Z < \dim \AA(\cS)$, we see that $Z$ cannot contain a component of a general fiber of $\pi$.  Hence, a general fiber of $\pi$ is irreducible, as desired. 
\end{proof}

\begin{proof}[Proof of Theorem~\ref{thm:vanishingsection}]
Take the vector $v = (1, \ldots, 1, 0, \ldots, 0) \in \kk^{E} \oplus \kk^E$.  Let us identify $\kk^{E} \oplus \kk^E = H^0(X_E, \bigoplus_{i\in E} \pi^*_i\mathcal O(1)) = (V^E)^\vee$.
The vector $v$ then defines a global section of $\bigoplus_{i\in E}\pi^*_i\mathcal O(1)$, and hence a global section of $\mathcal Q_L$ via the surjection $\bigoplus_{i\in E}\pi^*_i\mathcal O(1) \twoheadrightarrow \cQ_L$.
On the $\GG^E$-orbit $\mathbb{A}^E$ of $X_E$, Remark~\ref{rem:SES} identifies the restriction of $v$ with the section
\[
(x_1, \dotsc, x_n) \in \big(\kk[x_1, \dotsc, x_n]\big)^E =  H^0(\AA^E, \mathcal{O}_{\AA^E} \otimes \kk^E).
\]
So the image of $v$ in $H^0(\AA^E, \mathcal{O}_{\AA^E} \otimes \kk^E/L)$ vanishes exactly on $L$.
The $\mathbb{G}^E$-orbit of $v$ is dense in $\kk^{E} \oplus \kk^E$.  Hence, by $\mathbb{G}^E$-equivariance, the $\mathbb{G}^E$-orbit of the image of $v$ in $H^0(X_E, \mathcal Q_L)$ is dense in a subspace of $H^0(X_E, \mathcal Q_L)$ that globally generates $\mathcal Q_L$. 
In other words, the section $v$ is a sufficiently general section satisfying the conclusion of the above lemma, from which the theorem now follows.
\end{proof}

\begin{corollary}\label{cor:koszul} Let $L\subseteq \kk^E$ be a linear subspace of dimension $r$.
\begin{enumerate}\itemsep 5pt
\item The normal bundle $\mathcal N_{W_L/X_E}$ is identified with the restriction $\cQ_L |_{W_L}$.
\item The $K$-class of the structure sheaf $[\mathcal O_{W_L}]\in K(X_E)$ equals $\sum_{i = 0}^{n-r} (-1)^{i} [\textstyle{\bigwedge^i}\cQ_L^\vee]$.
\end{enumerate}
\end{corollary}

\begin{proof}
As $W_L$ is a smooth subvariety of $X_E$ of dimension $r$, that $W_L$ is the vanishing locus of a global section of $\cQ_L$ implies that the Koszul complex
\[
0 \to \textstyle{\bigwedge^{n- r} \cQ_L^\vee}\to \cdots \to  \textstyle{\bigwedge^2 \cQ_L^\vee} \to \cQ_L^\vee \to \mathcal O_{X_E}
\]
is a resolution of $\mathcal O_{W_L}$.  Both statements now follow.
\end{proof}

\subsection{Augmented Bergman classes}\label{subsec:MW}

We describe the Chern classes of augmented tautological classes and recover the augmented Bergman class as the top Chern class.  We use the language of Minkowski weights, defined as follows.

\begin{definition}\label{defn:MW}
A $d$-dimensional \emph{Minkowski weight} on a unimodular fan $\Sigma$ is a function $w \colon \Sigma(d) \to \ZZ$ such that the following balancing condition is satisfied: for every cone $\tau'\in \Sigma(d-1)$
\[
\sum_{\tau\succ \tau'} w(\tau)u_{\tau'\setminus\tau} \in \operatorname{span}(\tau')
\]
where the summation is over all cones $\tau\in \Sigma(d)$ containing $\tau'$, and $u_{\tau'\setminus\tau}$ denotes the primitive generator of the  unique ray of $\tau$ that is not in $\tau'$.  Write $\operatorname{MW}_d(\Sigma)$ for the set of $d$-dimensional Minkowski weights on $\Sigma$.
\end{definition}
Minkowski weights play the role of homology classes on smooth complete toric varieties in the following sense.

\begin{theorem}\cite[Theorem 3.1]{FultonSturmfels}\label{thm:FS}
Let $\Sigma$ be a complete unimodular fan of dimension $m$, and let $X_\Sigma$ be its toric variety.  Then, for every $0\leq d \leq m$, one has an isomorphism
\[
A^{m-d}(X_\Sigma) \overset\sim\to \operatorname{MW}_d(\Sigma) \quad\text{defined by}\quad \xi \mapsto \left( \tau \mapsto \int_X \xi \cdot [Z_\tau]\right).
\]
\end{theorem}

For a smooth complete toric variety $X_\Sigma$, when a Chow class $\xi\in A^\bullet(X_\Sigma)$ maps to a Minkowski weight $w \in \operatorname{MW}_\bullet(\Sigma)$ by the isomorphism in Theorem~\ref{thm:FS}, we say that $w$ and $\xi$ are \emph{Poincar\'e duals} of each other, which is notated by writing 
\[
\xi \cap [X_{\Sigma}] = w.
\]

We compute the Chern classes of the augmented tautological classes in terms of Minkowski weights on $\Sigma_E$.  By Theorem~\ref{thm:FS}, this amounts to computing how they intersect with the various torus-invariant strata of $X_E$, for which we use Proposition~\ref{prop:restrict} to reduce to understanding the Chern classes in the top degrees.  We hence begin by computing what happens in the top degrees.

\begin{lemma}\label{lem:degree}
We have that 
\begin{align*}
\int_{X_E} c(\mathcal{Q}_\M) &= \begin{cases} 1 & \M = \U_{0,E} \\ 0 & \text{otherwise,} \end{cases} \qquad\text{ and}\\
\int_{X_E} c(\mathcal{S}_\M) &= \begin{cases} 1 & \M = \U_{n,E} \\ 0 & \text{otherwise}. \end{cases}
\end{align*}
\end{lemma}
\begin{proof}
We do the case of $\mathcal{S}_{\M}$.  The case of $\mathcal{Q}_{\M}$ is similar. If $\M \neq \U_{n,E}$, then $\mathcal{S}_{\M}$ has rank less than $n$, so $c_{n}(\mathcal{S}_{\M}) = 0$. If $\M = \U_{n,E}$, then $\mathcal{S}_{\M} = \bigoplus_{i \in E} \pi_i^* \mathcal{O}_{\mathbb{P}^1}(1)$, so we have that $\deg c_{n}(\mathcal{S}_{\M}) = 1$. 
\end{proof}
We will also need the analogous statement for tautological bundles. 
\begin{lemma}\cite[Lemma 7.3]{BEST21}
We have that 
$$\int_{\underline{X}_{E}} c(\underline{\mathcal{Q}}_{\M}) = \begin{cases} 1 & \M = \U_{1,E} \text{ or }\M = \U_{0,1} \\ 0 & \text{otherwise}, \end{cases} \text{ and}$$
$$\int_{\underline{X}_{E}} c(\underline{\mathcal{S}}_{\M}) = \begin{cases} (-1)^{n-1} & \M = \U_{n-1,E} \text{ or } \M = \U_{1,1} \\ 0 & \text{otherwise}. \end{cases}$$
\end{lemma}
We now compute the intersection numbers of the Chern classes of $[\mathcal{S}_\M]$ and $[\mathcal{Q}_\M]$ with the boundary stata. When the minimal element of $\mathscr{F}$ is the empty set, we recover \cite[Proposition 7.4]{BEST21}.

\begin{proposition}\label{prop:minkowski}
Let $I \le \mathscr{F}: F_1 \subsetneq F_2 \subsetneq \dotsc \subsetneq F_k$ be a compatible pair, and set $\ell = \operatorname{codim} Z_{I \le \mathscr{F}}$.
As before, we set $F_{k + 1} = E$, and when $\mathscr{F}$ is empty we interpret $F_1$ as $E$.
Let $[Z_{I\leq \mathcal F}] \in A^\bullet(X_\Sigma)$ be the Chow class of the $T$-invariant subvariety $Z_{I\leq \mathcal F}$.
Then
\[
  \int_{X_E} c_{n - \ell}(\mathcal{Q}_{\M}) \cdot [Z_{I \le \mathscr{F}}] = \begin{cases}1 & \begin{aligned} &\text{$F_1 \subseteq \operatorname{cl}_{\M}(I)$, and for $i = 1, \ldots, k$, exactly $k + \operatorname{rk}_{\M}(I) - \operatorname{rk}_{\M}(\M)$ of} \\ & \text{the minors $\M|F_{i+1}/F_{i}$ are loops, and the rest are  $U_{1,F_{i+1}\setminus F_{i}}$,}\end{aligned}\\
\\
0 & \text{otherwise, and}\end{cases}
\]
\[
  \int_{X_E} c_{n - \ell}(\mathcal{S}_{\M}) \cdot [Z_{I \le \mathscr{F}}] = \begin{cases}(-1)^{\epsilon} & \begin{aligned} &\operatorname{rk}_{\M}(F_1) - \operatorname{rk}_{\M}(I) = |F_1| - |I|\text{, and for $i = 1, \ldots, k$, exactly} \\ & \text{$k + \operatorname{rk}_{\M}(\M) - \operatorname{rk}_{\M}(I) - n$ of the minors $\M|F_{i+1}/F_{i}$ are coloops,} \\ & \text{and the rest are  $U_{|F_{i+1}\setminus F_{i}| - 1 ,F_{i+1}\setminus F_{i}}$,}\end{aligned}\\
\\
0 & \text{otherwise,}\end{cases}
\]
where $\epsilon = n - k - |F_1|$. 
\end{proposition}

\begin{proof}
We do the case of $\mathcal{S}_{\M}$, the case of $\mathcal{Q}_{\M}$ is similar. By Proposition~\ref{prop:restrict}, we have that
$$c(\mathcal{S}_{\M}, u)|_{Z_{I \le \mathscr{F}}} = c(\mathcal{S}_{\M|F_1/I}, u) \otimes \bigotimes_{i=1}^{k} c(\underline{\mathcal{S}}_{\M|_{F_{i+1}}/F_{i}}, u) \in A^{\bullet}(X_{{F_1 \setminus I}}) \otimes \bigotimes_{i=1}^{k} A^{\bullet}(\underline X_{{F_{i+1} \setminus F_{i}}}).$$
Then Lemma~\ref{lem:degree} implies that the intersection number vanishes unless $\M|F_1/I$ is boolean, and each $\M|F_{i+1}/F_i$ is either a coloop or is a corank $1$ uniform matroid. Note that $\M|F_1/I$ is boolean if and only if $\operatorname{rk}_{\M}(F_1) - \operatorname{rk}_{\M}(I) = |F_1| - |I|$, and the fact that $\operatorname{rk}_{\M}(\M) = \operatorname{rk}_\M(I) + \operatorname{rk}_{\M}(\M|F_1/I) + \dotsb +  \sum \operatorname{rk}_{\M}(\M|F_{i+1}/F_i)$ implies that, if the intersection number is non-zero, then exactly $k + \operatorname{rk}_{\M}(\M) - \operatorname{rk}_{\M}(I) - n$ of the minors $\M|F_{i+1}/F_{i}$ are coloops. In this case, the intersection number is $(-1)^{\epsilon}$, where 
$$\epsilon = \sum\left(|F_{i+1}/F_i| - 1 \right),$$
where the sum is over the minors such that $\M|F_{i+1}/F_i$ is not a coloop. 
The set $E$ decomposes into a disjoint union of elements where the corresponding minor is a coloop, is in $I$, is in a non-coloop minor, or is in $F_1 \setminus I$, so
$$n = (k + \operatorname{rk}_{\M}(\M) - \operatorname{rk}_{\M}(I) - n) + |I| + (\sum |F_{i+1}/F_i|) + (|F_1| - |I|).$$
We also have that the number of non-coloops is $n + \operatorname{rk}_{\M}(I) - \operatorname{rk}_{\M}(\M)$. Substituting, we see that $\epsilon = n - k - |F_1|$. 
\end{proof}

We now define and derive certain properties of augmented Bergman fans and augmented Bergman classes.

\begin{definition}\label{defn:augmentedBergman}
For a matroid $\M$ of rank $r$ on $E$, the \emph{augmented Bergman fan}, denoted $\Sigma_\M$, is the subfan of $\Sigma_E$ consisting of cones $\sigma_{I \leq \mathscr {F}}$ where the subset $I\subseteq E$ is independent in $\M$ and the flag $\mathscr F$ consists of proper flats of $\M$.
The \emph{augmented Bergman class} $[\Sigma_\M]$ of $\M$ is the weight
\[
[\Sigma_\M] \colon \Sigma_E(r) \to \ZZ \quad\text{defined by}\quad \sigma \mapsto \begin{cases}
1 & \text{if $\sigma\in \Sigma_\M$}\\
0 & \text{otherwise.}
\end{cases}
\]
\end{definition}

\cite[Proposition 2.8]{BHMPW} states that, up to scaling, the augmented Bergman class is the unique way to assign weights to the cones of the augmented Bergman fan that results in a Minkowski weight.

\begin{corollary}\label{cor:ctop}
Let $\M$ be a matroid of rank $r$ on $E$.
\begin{enumerate}\itemsep 5pt
\item We have that $c_{n-r}(\mathcal{Q}_\M) = [\Sigma_\M]$.  In particular, the augmented Bergman class $[\Sigma_\M]$ is a well-defined Minkowski weight.
\item The assignment $\M\mapsto [\Sigma_\M]$ is valuative.
\item If $L\subseteq \kk^E$ is a realization of $\M$, then $[\Sigma_\M] = [W_L]$.
\end{enumerate}
\end{corollary}

\begin{proof}
The first statement follows from Proposition~\ref{prop:minkowski}.  The second statement follows from the first by Proposition~\ref{prop:tautoval}. The third statement follows from the first by Theorem~\ref{thm:vanishingsection}.
\end{proof}

By restricting to the permutohedral variety, we recover properties of ``non-augmented'' Bergman fans and classes as follows.  Note that for a loopless matroid $\M$, the augmented Bergman fan $\Sigma_\M$ contains the ray $\rho_\emptyset$.

\begin{definition}
The (non-augmented) \emph{Bergman fan} of a loopless matroid $\M$ on $E$ is $\underline\Sigma_\M = \operatorname{star}_{\rho_\emptyset} \Sigma_\M$.  Equivalently, it is the subfan of $\underline\Sigma_E$ consisting of cones $\underline\sigma_{\mathscr F}$ where the flag $\mathscr F$ consists of nonempty proper flats of $\M$.  The (non-augmented) \emph{Bergman class} $[\underline\Sigma_\M]$ is the Minkowski weight on $\underline\Sigma_E$ defined by assigning weight 1 to the cones of $\underline\Sigma_\M$.
\end{definition}

The Bergman class of a matroid with a loop is defined to be zero.  Since $[\cQ_\M]$ restricts to $[\underline\cQ_\M]$ on $\underline X_E$ and $[\Sigma_\M]$ restricts to $[\underline\Sigma_\M]$, Corollary~\ref{cor:ctop} recovers the properties of Bergman classes stated in \cite[Corollary 7.11]{BEST21}.

\subsection{Tropical geometry of augmented Bergman fans}\label{subsec:tropical}
The contents of this subsection are not logically necessary for the rest of the paper, but will be useful elsewhere. We explain how augmented Bergman fans are related to tropicalizations.  We point to \cite{MS15} for a background in tropical geometry.

\begin{proposition}\label{prop:trop}
Let $L\subseteq \kk^E$ be a realization of a matroid $\M$ of rank $r$.  For a general $b\in \mathbb{G}_a^E$, the tropicalization of the very affine variety $\mathring{L}_b = (L+b) \cap T$ equals the support of the augmented Bergman fan $\Sigma_\M$.
\end{proposition}

\begin{proof}
Let $\widetilde E = E\sqcup \{0\}$ and let $p \colon \ZZ^{\widetilde E}/\ZZ\be_{\widetilde E} \to \ZZ^E$ be the isomorphism described in \S\ref{subsec:refinecoarsen}.  Under the isomorphism $p$, we may identify $T$ with the projectivization $\PP \widetilde T$ of the torus $\widetilde T = (\kk^*)^{\widetilde E}$.  We show that the tropicalization of $\mathring{L}_b \subseteq \PP \widetilde T$ is the support of a subfan in $\underline\Sigma_{\widetilde E}$ that maps isomorphically under $p$ onto the augmented Bergman fan $\Sigma_\M$.\\
\indent Let $L = \{\mathbf x \in \kk^E \mid A^\perp\mathbf x = 0\}$ for an $(n-r)\times n$ matrix $A^\perp$.  For an element $b\in \mathbb{G}_a^E$, let $b'\in \mathbb{G}_a^E$ be such that $L+b = \{\mathbf x \in \kk^E \mid A^\perp \mathbf x = b'\}$.  In other words, the closure of $L+b$ in the projective completion $\PP(\kk^E \oplus \kk) = \PP(\kk^{\widetilde E})$ is the projectivization of the linear subspace $\{(\mathbf x, x_0) \in \kk^{\widetilde E} \mid A^\perp \mathbf x - b' x_0 = 0\}$.
Since $b'$ is general because $b$ was, this linear subspace is a realization of the matroid $\widetilde \M = \M \times 0$ on $\widetilde E$ called the \emph{free coextension} of $\M$, whose set of bases is defined as
\[
\{B\cup 0 \mid B \text{ a basis of $\M$}\} \cup \{S \subseteq E \mid \text{$S$ contains a basis of $\M$ and $|S| = r+1$}\}.
\]
It is a classical statement \cite{Stu02, AK06} that the tropicalization of a linear subspace is the support of the Bergman fan of the corresponding matroid.  Thus, it suffices now to show that the support of the Bergman fan of the free coextension is equal to that of the augmented Bergman fan under the isomorphism $p$.  This follows from the lemma below, which is a restatement of the discussion in \cite[\S5.1]{MM21}.
\end{proof}

\begin{lemma}
Let $\M$ be a matroid on $E$, and $\widetilde \M$ its free coextension matroid on $\widetilde E$.  The collection
\[
\cG = \{F\cup 0 \mid F\subseteq E\text{ a flat of $\M$}\} \cup \{i \in E \mid i \text{ not a loop in $\M$}\}
\]
is a building set on the lattice of flats of $\widetilde \M$ that induces the fan structure on the support $|\underline\Sigma_{\widetilde \M}| \subseteq \RR^{\widetilde E}/\RR\be_{\widetilde E}$ of the Bergman fan of $\widetilde \M$ consisting of cones
\[
\operatorname{cone}\{\overline\be_i \mid i\in I\} + \operatorname{cone}\{\overline\be_{F\cup 0} \mid F\in \mathscr F\}
\]
for each compatible pair $I\leq \mathscr  F$ with $I\subseteq E$ independent in $\M$ and $\mathscr F$ a flag of nonempty proper flats of $\M$.
\end{lemma}

We remark that the tropicalization of $(L+b) \cap T$ for a non-general $b$ can differ from the support of $\Sigma_\M$.  Nonetheless, by $\GG^E$-equivariance, the homology class of the closure $W_{L+b}$ of $L+b$ in the stellahedral variety $X_E$ is independent of $b\in \kk^E$.  Taking $b$ to be general, Proposition~\ref{prop:trop} gives an alternate proof that $[W_L] = [\Sigma_\M]$, for instance by \cite[Proposition 9.4]{Katz2009}.

\section{Exceptional isomorphisms}\label{sec:exceptIsom}

We construct the pair of isomorphisms between $K(X_E)$ and $A^\bullet(X_E)$ that were stated in Theorem~\ref{thm:exceptIsom}.  The two isomorphisms will be related via the two involutions $D_K$ and $D_A$ described in \S\ref{subsec:equivConsts}.

We begin by recalling Theorem~\ref{thm:localization}, which identifies the $T$-equivariant $K$-ring $K_T(X_E)$ with a subring of the product ring $\prod_{\sigma \in \Sigma_E(n)}\ZZ[T_1^{\pm 1}, \ldots, T_n^{\pm 1}]$ of Laurent polynomial rings, and identifies the $T$-equivariant Chow ring $A^\bullet_T(X_E)$ with a subring of the product ring $\prod_{\sigma\in \Sigma_E(n)} \ZZ[t_1, \ldots, t_n]$ of polynomial rings.
Let $A^\bullet_T(X_E)[\prod_{i\in E} (1+t_i)^{-1}]$ be the ring obtained by adjoining the inverse of the polynomial $\prod_{i\in E} (1+t_i)$ to the ring $A^\bullet_T(X_E)$.  For an element $f$ in such product rings, denote by $f_\sigma$ the (Laurent) polynomial corresponding to $\sigma \in \Sigma_E(n)$.

\begin{theorem}\label{thm:zeta}
The map $\zeta_T\colon K_T(X_E) \to A^\bullet_T(X_E)[\prod_{i\in E} (1+t_i)^{-1}]$ defined by sending
\[
f_\sigma(T_1, \dotsc, T_n)\mapsto f_\sigma(1+ t_1, \dotsc, 1+t_n) \text{ for any $\sigma\in \Sigma_E(n)$}
\]
is a ring isomorphism, which descends to a ring isomorphism $\zeta \colon K(X_E)\to A^\bullet(X_E)$.
\end{theorem}

\begin{proof}
Every edge of the stellahedron $\Pi_E$ is parallel to either $\be_i$ for some $i\in E$ or to $\be_i - \be_j$ for some $i\neq j \in E$.  Thus, the conditions $f_\sigma(T_1, \dotsc, T_n) - f_{\sigma'} (T_1, \dotsc, T_n)\equiv 0 \ \operatorname{mod}\  1-T^v$ appearing in Theorem~\ref{thm:localization}.\ref{localization:K}, in the case of $K_T(X_E)$, state that either $f_{\sigma} - f_{\sigma'} \equiv 0 \ \operatorname{mod} \ 1 - T_i$ or $f_{\sigma} - f_{\sigma'} \equiv 0 \ \operatorname{mod} \ 1 - \frac{T_i}{T_j}$.  The latter is equivalent to stating that $f_{\sigma} - f_{\sigma'} \equiv 0 \ \operatorname{mod} \ T_j - T_i$.  Under the transformation $T_i \mapsto 1+t_i$ defining $\zeta_T$, these two conditions become $f_{\sigma}(1 + t_1, \dotsc, 1 + t_n) - f_{\sigma'}(1+ t_1, \dotsc, 1 + t_n) \equiv 0 \ \operatorname{mod} \ t_i$ and $f_{\sigma}(1 + t_1, \dotsc, 1 + t_n) - f_{\sigma'}(1 + t_1, \dotsc, 1 + t_n) \equiv 0 \ \operatorname{mod} \ t_j - t_i$, which are exactly the conditions appearing in Theorem~\ref{thm:localization}.\ref{localization:A} in the case of $A_T^\bullet(X_E)$.  Hence, the map $\zeta_T$ is well-defined and is clearly an isomorphism.\\
\indent We now check that the isomorphism $\zeta_T$ descends to a ring isomorphism on the non-equivariant rings.  We recall from Theorem~\ref{thm:localization} that the kernel $I_K$ of the quotient map $K_T(X_E) \to K(X_E)$ is the ideal in $K_T(X_E)$ generated by $f - f(1,\ldots, 1)$ for $f$ a global Laurent polynomial, and that the kernel $I_A$ of the quotient map $A^\bullet_T(X_E) \to A^\bullet(X_E)$ is the ideal in $A^\bullet_T(X_E)$ generated by $f - f(0,\ldots, 0)$ for $f$ a global polynomial.
Note that the polynomial $\prod_{i\in E}(1+t_i)$ whose inverse was adjoined to $A_T^\bullet(X_E)$ maps to 1 under this quotient map.   It thus remains only to show that $\zeta_T$ maps $I_K$ isomorphically onto $I_A' = I_A[\prod_{i\in E}(1+t_i)^{-1}]$.  But both $\zeta_T(I_K) \subseteq I_A'$ and $\zeta_T(I_K) \supseteq I_A'$ are straightforward to verify by considering their generators.
\end{proof}

By conjugating $\zeta$ by the two involutions $D_K$ and $D_A$, we have the ``dual'' isomorphism.

\begin{definition}
Let $\phi \colon K(X_E)\to A^\bullet(X_E)$ be the isomorphism defined by $\phi = D_A \circ \zeta \circ D_K$.
\end{definition}

We remark that, similarly to Theorem~\ref{thm:zeta}, one can show that the map $\phi_T \colon K_T(X_E) \to A^\bullet_T(X_E)[\prod_{i\in E} (1-t_i)^{-1}]$ defined by sending
\[
f(T_1, \dotsc, T_n)\mapsto f((1-t_1)^{-1}, \dotsc, (1 - t_n)^{-1}) \text{ for a Laurent polynomial $f \in \ZZ[T_1^{\pm 1}, \dotsc, T_n^{\pm 1}]$}
\]
is an isomorphism, which descends to the non-equivariant isomorphism $\phi$.

We now show that $\zeta$ and $\phi$ behave particularly well with respect to $K$-classes with ``simple Chern roots,'' a notion introduced in \cite{BEST21}.

\begin{definition}
A $T$-equivariant $K$-class $[\mathcal{E}] \in K_T(X_E)$ has \emph{simple Chern roots} if for each maximal $\sigma \in \Sigma_E$, there is a sequence $(a_{\sigma, 0}, a_{\sigma, 1}, \ldots, a_{\sigma, n})$ such that $[\mathcal{E}]_{\sigma} = a_{\sigma, 0} + \sum_{i=1}^{n} a_{\sigma, i} T_i$. 
\end{definition}

Note that $[\mathcal{Q}_{\M}]^{\vee}$ and $[\mathcal{S}_{\M}]^{\vee}$ have simple Chern roots.

\begin{proposition}\label{prop:simpleChern} Let $[\mathcal{E}] \in K_T(X_E)$ have simple Chern roots.  With $u$ a formal variable, we have
\begin{align*}
\sum_{j \ge 0} \zeta_T( \textstyle\bigwedge^j[\mathcal{E}])u^j &= (u + 1)^{\operatorname{rk}_{\M}(\mathcal{E})} c^T\left (\mathcal{E}, \frac{u}{u + 1} \right),\\
\sum_{j \ge 0} \phi_T( \textstyle\bigwedge^j[\mathcal{E}])u^j &= (u + 1)^{\operatorname{rk}_{\M}(\mathcal{E})} s^T(\mathcal{E}^{\vee}) c^T\left (\mathcal{E}^{\vee}, \frac{1}{u + 1} \right),\\
\sum_{j \ge 0}^{} \zeta_T(\operatorname{Sym}^j [\mathcal{E}])u^j &= \frac{1}{(1 - u)^{\operatorname{rk}_{\M}(\mathcal{E})}} s^T\left (\mathcal{E}, \frac{u}{u-1}\right ), \text{ and}\\
\sum_{j \ge 0} \phi_T(\operatorname{Sym}^j [\mathcal{E}])u^j &= \frac{c^T(\mathcal{E}^{\vee})}{(1 - u)^{\operatorname{rk}_{\M}(\mathcal{E})}} s^T\left(\mathcal{E}^{\vee}, \frac{1}{1-u}\right).
\end{align*}
\end{proposition}

\begin{proof}
We prove the formulas involving $\phi$. The formulas involving $\zeta$ are similar (and the first formula follows from \cite[Proposition 10.5]{BEST21}). 
Since $[\mathcal{E}]$ has simple Chern roots, we have that $[\mathcal{E}]_{\sigma} = a_{\sigma, 0} + \sum_{i \in I_{\sigma}} T_i$ for some multiset $I_{\sigma}$. We then compute
\begin{align*}
\sum_{j \ge 0} \phi_T( \textstyle\bigwedge^j\displaystyle [\mathcal{E}])_{\sigma}u^j & = (u + 1)^{a_{\sigma, 0} + |I_{\sigma}|} \prod_{i \in I_{\sigma}}(1/(1 - t_i)) (1 - t_i/(u + 1)) \\
& = (u + 1)^{\operatorname{rk}_{\M}(\mathcal{E})} s^T(\mathcal{E}^{\vee})_{\sigma} c^T\left (\mathcal{E}^{\vee}, \frac{1}{u + 1} \right)_{\sigma},\quad\text{and}\\
\sum_{j \ge 0} \phi_T(\operatorname{Sym}^j [\mathcal{E}])_{\sigma}u^j & = \frac{1}{(1 - u)^{a_{\sigma, 0} + |I_{\sigma}|}} \prod_{i \in I_{\sigma}} \frac{1 - t_i}{1 - t_i/(1-u)} = \frac{c^T(\mathcal{E}^{\vee})_{\sigma}}{(1 - u)^{\operatorname{rk}_{\M}(\mathcal{E})}} s^T\left (\mathcal{E}^{\vee}, \frac{1}{1-u} \right)_{\sigma},
\end{align*}
as desired.
\end{proof}

We note in particular the following consequence of Proposition~\ref{prop:simpleChern}.

\begin{corollary}\label{cor:phizetaontauto}
Let $\M$ be a matroid of rank $r$ on $E$.  Let $D_{I(\M^\perp)}$ be the $T$-invariant divisor associated to $I(\M^\perp)$ as discussed above Example~\ref{eg:twomaps}.
\begin{enumerate}\itemsep 5pt
\item One has $\phi([\mathcal O_{X_E}(D_{I(\M^\perp)})]) = c(\cQ_{\M})$ and  $\zeta([\mathcal O_{X_E}(D_{I(\M^\perp)})]) = s(\cQ_{\M}^\vee)$.
\item If $L\subseteq \kk^E$ realizes $\M$, then $\zeta([\mathcal O_{W_L}]) = [W_L]$.
\end{enumerate}
\end{corollary}

\begin{proof}
Applying $\zeta = D_A \circ \phi \circ D_K$ to the first formula in the proposition gives
\[
\sum_{j \ge 0} \phi( \textstyle\bigwedge^j[\mathcal{E}]^\vee)u^j = (u + 1)^{\operatorname{rk}_{\M}(\mathcal{E})} c(\mathcal{E}, -\frac{u}{u + 1})
\]
for $[\mathcal E]\in K(X_E)$ with simple Chern roots.  Since $[\cQ_\M]^\vee$ has simple Chern roots with $\operatorname{rk}(\cQ_\M) = n - r$, and since $[\bigwedge^{n-r}\cQ_\M] = [\det\cQ_\M] = [\mathcal O_{X_E}(D_{I(\M^\perp)})]$ by Proposition~\ref{prop:firstChern}, the first statement now follows by setting $[\mathcal E] = [\cQ_\M]^\vee$ and noting that $c(\mathcal E,-u) = c(\mathcal E^\vee,u)$.
The second statement follows from the first formula in the proposition and Corollary~\ref{cor:koszul}.
\end{proof}

\begin{example}\label{eg:alphays}
Note that $[\det \cQ_{\U_{n-1,E}}] = [\mathcal O_{X_E}(D_{I(\U_{1,E})})]$ and $[\det \cQ_{\U_{0,E}}]  = [\mathcal{O}_{X_E}(D_{I(\U_{n,E})})]$.  Because the line bundles $\mathcal O_{X_E}(D_{I(\U_{1,E})})$ and $\mathcal O_{X_E}(D_{I(\U_{n,E})})$ induce the maps $\pi_E\colon X_E \to \PP^E$ and $\pi_{1^E}\colon X_E \to (\PP^1)^E$, respectively, we have
\[
\phi([\mathcal{O}_{X_E}(D_{I(\U_{1,E})})]) = 1+\alpha \quad\text{and}\quad \phi([\mathcal{O}_{X_E}(D_{I(\U_{n,E})})]) = \prod_{i\in E} (1+y_i) =  c\big(\bigoplus_{i\in E}\pi_i^*\mathcal O_{\PP^1}(1)\big).
\]
Here, recall the notation that $\alpha = c_1(\pi_E^* \mathcal O_{\PP^E}(1))$ and $y_i = c_1(\pi_i^* \mathcal O_{\PP^1}(1))$.
\end{example}

\begin{remark}\label{rem:compareHRR}
Let us remark on how the maps $\phi$ and $\zeta$ here are related to the exceptional isomorphism for permutohedral varieties given in \cite[Theorem D]{BEST21}. Just as for augmented tautological bundles, classes, and Bergman classes, the first relation comes from considering $\underline X_E$ as a $T$-fixed divisor on $X_E$: The restriction of $\zeta$ to $\underline X_E$ recovers the isomorphism $\underline\zeta$ between $K(\underline X_E)$ and $A^\bullet(\underline X_E)$ in \cite[Theorem D]{BEST21}.\\
\indent Let us now sketch a different relation.  Let $\widetilde E = E \sqcup \{0\}$ as in \S\ref{subsec:refinecoarsen}, where we noted that the stellahedral fan $\Sigma_E$ can be considered as a coarsening of the permutohedral fan $\underline\Sigma_{\widetilde E}$.  In other words, we have a $T$-equivariant birational map $p \colon \underline{X}_{\widetilde E} \to X_E$.  One can show that there is a commuting diagram
\[
\begin{tikzcd}
K(X_{E}) \arrow[r, "\zeta"] \arrow[d, hook]
& A^{\bullet}(X_{E}) \arrow[d, hook] \\
K(\underline X_{\widetilde E}) \arrow[r, "\underline{\zeta}"]
& A^{\bullet}(\underline X_{\widetilde E})
\end{tikzcd}
\]
where the two vertical maps are the respective pullback maps, and one has similar commuting diagrams for $\phi$ and the $T$-equivariant versions of $\zeta$ and $\phi$.  Both Theorem~\ref{thm:exceptIsom} and Theorem~\ref{thm:fakeHRR} can then be deduced from the commutativity of the diagrams and \cite[Theorem D]{BEST21}.
\end{remark}

\section{Valuative group, homology, and the intersection pairing}\label{SectionValuativeGroup}

\subsection{The polytope algebra and the proof of Theorem~\ref{mainthm:equivalences}}\label{subsec:polytopealgebra}
For the proof of Theorem~\ref{mainthm:equivalences}, the last remaining ingredient is the polytope algebra introduced in \cite{McM89}.
For a polytope $Q\subseteq \RR^E$, define the function $\mathbf 1_Q \colon \RR^E \to \ZZ$ by $\mathbf 1_Q(u) = 1$ if $u\in P$ and 0 otherwise.  Recall that a (lattice) polytope $P$ is said to be a (lattice) deformation of $Q$ if its normal fan $\Sigma_{P}$ coarsens that of $Q$.

\begin{definition}
Let $\Sigma$ be the normal fan of a smooth polytope $Q\subseteq \RR^E$.  Let $\mathbb I(\Sigma)$ be the subgroup of $\ZZ^{\RR^E}$ generated by $\{\mathbf 1_P \mid P \text{ a lattice deformation of $Q$}\}$, and let $\operatorname{transl}(\Sigma)$ to be the subgroup of $\mathbb I(\Sigma)$ generated by $\{\mathbf 1_P - \mathbf 1_{P+u} \mid u \in \ZZ^E\}$.  We define the \emph{polytope algebra} to be the quotient
\[
\overline{\mathbb I}(\Sigma) = \mathbb I(\Sigma) / \operatorname{transl}(\Sigma).
\]
\end{definition}

For a lattice deformation $P$, let us denote by $[P]$ its class in the polytope algebra $\overline{\mathbb I}(\Sigma)$.  The polytope algebra, as the terminology suggests, is a ring with multiplication induced by Minkowski sum, that is, by $[P]\cdot[P'] = [P+ P']$.
It was well-known among experts that the polytope algebra is naturally identified with $K(X_{\Sigma})$; this is realized in Theorem~\ref{thm:folklore}. When we apply the theorem to the stellahedral variety, noting that deformations of the stellahedron are exactly polymatroids (Proposition~\ref{prop:nefpolymatroid}), we deduce the following.

\begin{theorem}\label{thm:polytopealgebrastella}
The map sending an integral polymatroid $P$ on $E$ to $[\mathcal O_{X_E}(D_P)]$ defines an isomorphism $\overline{\mathbb I}(\Sigma_E) \simeq K(X_E)$.
\end{theorem}

We now prove Theorem~\ref{thm:valuativeiso} by showing that we have a sequence of isomorphisms
\[
\bigoplus_{r = 0}^{n} \mathrm{Val}_r(E) \simeq \overline{\mathbb I}(\Sigma_E) \simeq K(X_E) \simeq A^\bullet(X_E).
\]
We prepare for the first isomorphism in the sequence with the following lemma.

\begin{lemma}\label{lem:polymatspan}
The intersection of an integral polymatroid with an integral translate of the boolean cube $[0,1]^E$, if nonempty, is a translate of the independence polytope of a matroid.
\end{lemma}

\begin{proof}
For $i\in E$ and $a\in \ZZ$, let us define the hyperplane $H_{i,a} = \{u\in \RR^E \mid \langle \be_i, u\rangle = a\}$ and its half-spaces $H^{\pm}_{i,a} = \{u\in \RR^E \mid \langle \pm\be_i, u\rangle \geq \pm a\}$.
It follows from Definition~\ref{defn:polymat} that a polymatroid intersected with any half-space $H^{+}_{i,a}$ or $H^{-}_{i,a}$ is a translate of a polymatroid if it isn't empty.
So, the intersection of an integral polymatroid with an integer translate of the boolean cube is a translate of a polymatroid if nonempty.  By Example~\ref{eg:indep}, it now suffices to verify that this polymatroid is integral.\\
\indent By \cite[(35)]{edmonds1970submodular}, the intersection of two integral polymatroids is a polytope whose vertices lie in $\ZZ^E$. By intersecting an integral polymatroid $P$ with integral polymatroids of the form $\prod_{i=1}^{n} [0, a_i]$, for $a_i \in \mathbb{Z}_{\ge 0}$, we see that all vertices of the intersection of $P$ with an integral translate of the boolean cube are in $\mathbb{Z}^E$. 
\end{proof}

\begin{proposition}\label{prop:val&ind}
The map $\bigoplus_{r = 0}^{n} \mathrm{Val}_r(E) \to \overline{\mathbb I}(\Sigma_E)$ defined by $\M \mapsto [I(\M^\perp)]$ is an isomorphism.
\end{proposition}

\begin{proof}
To see that the given map is well-defined, note that the base polytope of the dual $P(\M^\perp)$ is $-(P(\M) - \be_E)$, and that the independence polytope $I(\M^\perp)$ is the intersection with $[0,1]^E$ of the Minkowski sum $P(\M^\perp) + [-1,0]^E$.  Each of these operations---translation, negation, Minkowski sum, and intersection---preserves valuative relations.  Surjectivity of the map is immediate from Lemma~\ref{lem:polymatspan}, since given an integral polymatroid $P$, by tiling $\RR^E$ with integer translates of the boolean cube, we can express $[P]\in \overline{\mathbb  I}(\Sigma_E)$ as a linear combination of the classes of independence polytopes of matroids.\\
\indent For injectivity, first we show that the only relations between indicator functions of translates of independence polytopes come from valuativity. 
Suppose we have $\sum_{i=1}^k a_i \mathbf{1}_{I(\M_i) + u_i} = 0$ for $a_i\in \ZZ$, $u_i \in \mathbb{Z}^n$, and $\M_i$ a matroid on $E$.  We show that then $\sum_{i=1}^k a_i \mathbf 1_{I(\M_i)} = 0$ as an element in $\ZZ^{\RR^E}$.
By Proposition~\ref{prop:faceval}, this implies that $\sum_{i=1}^k a_i \mathbf 1_{P(\M_i)} = 0$ because each $I(\M_i)$ has $P(\M_i)$ as the face maximizing the pairing with $\be_E$.  For a subset $S\subseteq E$, let $\ell_S$ be the subset of $\{\M_1, \ldots, \M_k\}$ consisting of matroids whose set of loops is equal to $S$, or equivalently, the smallest coordinate subspace containing the independence polytope of the matroid is $\RR^S\subseteq \RR^E$.
Let us pick a linear ordering $(S_0 = \emptyset, S_1, S_2, \ldots, S_{2^{n}} = E)$ of the subsets of $E$ that refines the partial order by inclusion. We claim by induction that $\sum_{\M_j \in \ell_{S_i}} a_j \mathbf 1_{I(\M_j)} = 0$.  In the base case $S_0 = \emptyset$, the polytopes $I(\M_j)$ for all $\M_j\in \ell_{S_0}$ nontrivially intersect the interior of the boolean cube $[0,1]^E$, whereas none of those of $\M_{j'} \in \ell_{S_i}$ for $i>0$ do.  Hence that $\sum_{i=1}^k a_i \mathbf 1_{I(\M_i) + u_i} = 0$ implies that $\sum_{\M_j \in \ell_{S_0}} a_j \mathbf 1_{I(\M_j)} = 0$.  For the induction step at $S_i$, we may assume that $\ell_{S_0}, \ldots, \ell_{S_{i-1}}$ are empty.  
Then, we repeat the argument with ``the interior of the boolean cube'' replaced by ``the relative interior of the cube $[0,1]^{S_i} \times \{0\}^{E\setminus S_i}$''.  That is, the polytopes $I(\M_j)$ for all $\M_j\in \ell_{S_i}$ nontrivially intersect the relative interior of the cube $[0,1]^{S_i} \times \{0\}^{E\setminus S_i}$, whereas none of those of $\M_{j'} \in \ell_{S_{i'}}$ for $i'>i$ do.
Hence, again we conclude $\sum_{\M_j \in \ell_{S_i}} a_j \mathbf 1_{I(\M_j)} = 0$ from $\sum_{i=1}^k a_i \mathbf 1_{I(\M_i) + u_i} = 0$, completing the induction.\\
\indent Now suppose that $\sum_{i=1}^k a_i [I(\M_i)] = 0$ for $a_i \in \mathbb{Z}$ and $\M_i$ a matroid on $E$. This means that
$$\sum_{i=1}^k a_i \mathbf{1}_{I(\M_i)} + \sum_{P, m} b_{P, m} (\mathbf{1}_{P + m} - \mathbf{1}_{P}) = 0$$
for some collection of polymatroids $P$, vectors $m \in\ZZ^n$, and integers $b_{P,m}$.
Using Lemma~\ref{lem:polymatspan}, we can rewrite this as 
$$\sum_{i=1}^{k} a_i \mathbf{1}_{I(\M_i)} + \sum_{j=1}^{\ell} c_{j} (\mathbf{1}_{I(\M'_j) + m_{j}} - \mathbf{1}_{I(\M'_j)}) = 0$$
for some collection of matroids $\M'_j$ and vectors $m_j \in \mathbb{Z}^n$. Then the previous discussion implies that equality still holds when we remove the second sum, as desired.
\end{proof}

\begin{proof}[Proof of Theorem~\ref{thm:valuativeiso}]
In Proposition~\ref{prop:val&ind}, we have constructed an isomorphism $\bigoplus_{r = 0}^{n} \mathrm{Val}_r(E) \to \overline{\mathbb I}(\Sigma_E)$ defined by $\M \mapsto [I(\M^\perp)]$.  Now, composing the isomorphism $\overline{\mathbb I}(\Sigma_E) \simeq K(X_E)$ in Theorem~\ref{thm:polytopealgebrastella} with the isomorphism $\phi \colon K(X_E) \to A^\bullet(X_E)$ in \S\ref{sec:exceptIsom}, we obtain an isomorphism $\overline{\mathbb I}(\Sigma_E) \to A^\bullet(X_E)$, which by Corollary~\ref{cor:phizetaontauto} maps $[I(\M^\perp)]$ to $c(\cQ_\M)$ for a matroid $\M$.  By Corollary~\ref{cor:ctop}, the top nonvanishing degree part $c_{n-\operatorname{rk}(\M)}(\cQ_\M)$ of $c(\cQ_\M)$ is the augmented Bergman class $[\Sigma_\M]$, so we conclude from the graded structure of $A^\bullet(X_E)$ that $\bigoplus_{r = 0}^{n} \mathrm{Val}_r(E) \to A^\bullet(X_E)$ defined by $\M \mapsto [\Sigma_\M]$ is an isomorphism of abelian groups.
\end{proof}

With Theorem~\ref{thm:valuativeiso}, we can now complete the proof of Theorem~\ref{thm:exceptIsom}.

\begin{proof}[Proof of Theorem~\ref{thm:exceptIsom}]
That $\zeta$ and $\phi$ are ring isomorphisms was proved in Section \ref{sec:exceptIsom}, and that they satisfy the stated properties is Corollary~\ref{cor:phizetaontauto}.  To verify that the stated properties characterize the maps, note first that $A^1(X_E)$ generates $A^\bullet(X_E)$ as a ring, and that the augmented Bergman classes of matroids of rank $n-1$ span $A^1(X_E)$ because $\operatorname{Val}_{n-1}(E) \simeq A^1(X_E)$ by Theorem~\ref{thm:valuativeiso}.  The result now follows because every matroid of rank $n-1$ is realizable over any field, and if $L\subset \kk^E$ realizes a matroid $\M$ of rank $n-1$ then $[W_L]= [\Sigma_\M]$ and $c(\cQ_L) = 1+c_1(\cQ_L) = 1 + [\Sigma_\M]$ by Corollary~\ref{cor:ctop}.
\end{proof}

We now prove Theorem~\ref{thm:intersectionaugmented} by using Lemma~\ref{lem:polymatspan} with Corollary~\ref{cor:phizetaontauto} and Corollary~\ref{cor:ctop}.

\begin{proof}[Proof of Theorem~\ref{thm:intersectionaugmented}]
If $\operatorname{crk}(\M) + \operatorname{crk}(\M') > n \ge \operatorname{crk}(\M \wedge \M')$, then the result vacuously holds, so we may assume that $\operatorname{crk}(\M) + \operatorname{crk}(\M') \le n$. 
Note that, by Corollary~\ref{cor:ctop}, the degree $\operatorname{crk}(\M) + \operatorname{crk}(\M')$ part of $c(\mathcal{Q}_\M)c(\mathcal{Q}_{\M'})$ is $[\Sigma_\M] \cdot [\Sigma_{\M'}]$, so by Corollary~\ref{cor:phizetaontauto} it suffices to compute the degree  $\operatorname{crk}(\M) + \operatorname{crk}(\M')$ part of $\phi([I(\M^{\perp})] \cdot [I(\M'^{\perp})])$. By Lemma~\ref{lem:polymatspan}, we may write $[I(\M^{\perp})] \cdot [I(\M'^{\perp})] = [I(\M^{\perp}) + I(\M'^{\perp})]$ as a sum of the classes of independence polytopes of matroids by intersecting it with the tiling of $\mathbb{R}^E$ by translates of the boolean cube and using inclusion-exclusion on the faces. This gives an expression for non-equivariant $K$-class $[I(\M^{\perp})] \cdot [I(\M'^{\perp})]$ as a sum of the $K$-classes of independence polytopes of matroids. \\
\indent The intersection of $I(\M^{\perp}) + I(\M'^{\perp})$ with the boolean cube is $I((\M \wedge \M')^{\perp})$. The image of $[I((\M \wedge \M')^{\perp})]$ under $\phi$ is $[\Sigma_{\M \wedge \M'}]$ in degree $\operatorname{crk}(\M \wedge \M')$. 
Therefore, it suffices to show that the images under $\phi$ of all of the other terms in the expression of $[I(\M^{\perp})+I(\M'^{\perp})]$ as a sum of the classes of independence polytopes of matroids are zero in degrees at least $\operatorname{crk}(\M) + \operatorname{crk}(\M')$.
Every other polytope appearing requires a nontrivial translation towards the origin to realize it as an independence polytope, since an independence polytope always contains the origin.
As the lattice distance from the origin of any vertex of $I(\M^{\perp}) + I(\M'^{\perp})$ is bounded by $\operatorname{crk}(\M) + \operatorname{crk}(\M')$, this means that, after translating one of these polytopes so that it is the independence polytope of a matroid, that matroid has rank at most $\operatorname{crk}(\M) + \operatorname{crk}(\M') - 1$.
Then the result follows from Proposition~\ref{prop:minkowski}.
\end{proof}

We showed in the discussion following Corollary~\ref{cor:ctop} that $[\Sigma_\M]$ restricts to $[\underline\Sigma_\M]$ on $\underline X_E$.  Hence, by restricting to $\underline X_E \subseteq X_E$, we obtain Corollary~\ref{cor:intersectbergman} from Theorem~\ref{thm:intersectionaugmented}.  We also deduce that if $\M, \M'$, and $\M \wedge \M'$ are loopless, then $\operatorname{crk}(\M) + \operatorname{crk}(\M') = \operatorname{crk}(\M\wedge \M')$.

\subsection{A Schubert basis}\label{subsec:schubert}

For a total order $<$ on $E$ and two subsets $I = \{i_1 < \cdots < i_r\}$ and $J = \{j_1 < \cdots <j_r\}$ of $E$ with same cardinality, let us say that $I\leq J$ if $i_k \leq j_k$ for all $k = 1, \ldots, r$.

\begin{definition}
A \textbf{Schubert matroid} on $E$ of rank $r$ is a matroid whose set of bases is
\[
\{B \subseteq E \mid |B| = r \text{ and }B \leq I \}
\]
for some total order $<$ on $E$ and a subset $I\subseteq E$ with $|I| = r$.
\end{definition}

Because $I\leq J$ if and only if $(E\setminus I) \geq (E\setminus J)$, the dual of a Schubert matroid is a Schubert matroid.
We note the following equivalent description of the bases of a Schubert matroid.

\begin{remark}\label{rem:schubertequiv}
Let $<$ be a total order on $E$, and $I = \{i_1 < \cdots < i_r\}$.  Define
\[
I_\text{jumps} = \{i_j \in I \mid j=r \text{ or there exists $e\in E$ such that $i_j < e < i_{j+1}$}\}.
\]
Writing $I_\text{jumps} = \{\ell_1< \cdots < \ell_k\}$, define a chain $F_1, \ldots, F_k$ of subsets of $E$ and positive integers $d_1, \ldots, d_k$ by
\[
F_j = \{e \in E \mid e\leq \ell_j\} \quad\text{and}\quad d_1+ \cdots + d_j = |F_j \cap I| \qquad\text{for $j = 1, \ldots, k$}.
\]
Note that by construction, we have $d_1 \leq |F_1|$ and $d_j < |F_j \setminus F_{j-1}|$ for all $j = 2, \ldots, k$.  The set $\{B\subseteq E \mid |B| = r \text{ and } B\leq I\}$ of the bases of the Schubert matroid associated to $<$ and $I$ then can be described equivalently as the set
\[
\{ B= \{b_1 < \cdots < b_r\} \subseteq E \mid \{b_1, \ldots, b_{d_1 + \cdots + d_j}\} \subseteq F_j \text{ for all $j = 1, \ldots, k$}\}.
\]
\end{remark}

Schubert matroids appear in the literature under various other guises such as nested matroids \cite{Ham17}, Bruhat interval polytopes \cite{TW15}, generalized Catalan matroids \cite{BdM06}, and shifted matroids \cite{Ard03}.

\begin{theorem}\label{thm:schubert}
The augmented Bergman classes of Schubert matroids on $E$ form a basis for $A^\bullet(X_E)$.
\end{theorem}

We prepare the proof with the following lemma.

\begin{lemma}\label{lem:schubert}
For $\emptyset\subsetneq F \subseteq E$, denote by $h_F$ the divisor $D_{I(\U_{1,F} \oplus \U_{0,E\setminus F})}$ corresponding to $I(\U_{1,F} \oplus \U_{0,E\setminus F})$ under Proposition~\ref{prop:nefpolymatroid}.
Then, the set of monomials
\[
\left\{h_{F_1}^{d_1}\cdots h_{F_k}^{d_k} \ \middle| \ \emptyset \subsetneq F_1 \subsetneq \cdots \subsetneq F_k \subseteq E,\ d_1 \leq |F_1|,\ d_i < |F_{i}\setminus F_{i-1}|\ \forall i = 2, \ldots, k \right\}.
\]
form a basis for the Chow cohomology ring $A^\bullet(X_E)$.
\end{lemma}

\begin{proof}
Let $\cG = \{S\cup 0 \mid S\subseteq E\} \cup E$ be the building set on $\widetilde E = E \sqcup \{0\}$ in Proposition~\ref{prop:stellacoarsen}, and let $\Sigma_{\mathcal{G}}$ denote the corresponding fan.  Then, \cite[Corollary 2]{FY04} states that the Chow cohomology ring of $\Sigma_\cG$ has a presentation
\[
A^\bullet(\Sigma_\cG) = \frac{\ZZ[z_X \mid X\in \cG]}{\left\langle z_{X_1}\cdots z_{X_k} \mid \{X_1, \ldots, X_k\} \text{ not a face of $\mathcal N$}\right\rangle +\left \langle \sum_{X \ni i} z_X \mid i \in \widetilde E \right\rangle},
\]
and moreover, \cite[Corollary 1]{FY04} states that the set of monomials
\[
\left\{z_{F_1\cup 0}^{d_1}\cdots z_{F_k\cup 0}^{d_k} \ \middle| \ \emptyset \subsetneq F_1 \subsetneq \cdots \subsetneq F_k \subseteq E,\ d_1 \leq |F_1|,\ d_i < |F_{i}\setminus F_{i-1}|\ \forall i = 2, \ldots, k \right\}
\]
form a basis for $A^\bullet(\Sigma_\cG)$.  We modify this basis by performing an upper triangular linear change of variables as follows.
For $\emptyset\subsetneq F \subseteq E$, let
\[
\widetilde h_F = \sum_{F\subseteq G \subseteq E} -z_{G\cup 0}.
\]
When $\cG$ is given any total order that refines the partial order by inclusion, replacing $z_{F\cup 0}$ by $\widetilde h_F$ is an upper triangular linear change of variables.  Hence, we have that 
\[
\left\{\widetilde h_{F_1}^{d_1}\cdots \widetilde h_{F_k}^{d_k} \ \middle| \ \emptyset \subsetneq F_1 \subsetneq \cdots \subsetneq F_k \subseteq E,\ d_1 \leq |F_1|,\ d_i < |F_{i}\setminus F_{i-1}|\ \forall i = 2, \ldots, k \right\}
\]
is a basis of $A^\bullet(\Sigma_\cG)$.
It remains only to verify that for any $\emptyset\subsetneq F \subseteq E$, the element $\widetilde h_F \in A^1(\Sigma_\cG)$ corresponds to $h_F \in A^\bullet(X_E)$ under the isomorphism $p\colon \Sigma_\cG \to \Sigma_E$ of Proposition~\ref{prop:stellacoarsen}.\\
\indent In the presentation of $A^1(\Sigma_\cG)$ above, for $\emptyset\subseteq S \subsetneq E$, the variable $z_{S\cup 0}$ represents the torus-invariant divisor associated to the ray $\operatorname{cone}(\overline \be_{S\cup 0})$ of $\Sigma_\cG$, which under the isomorphism $p\colon \Sigma_\cG \to \Sigma_E$ in Proposition~\ref{prop:stellacoarsen} maps to the ray $\rho_{S}$ of $\Sigma_E$.  Moreover, it follows from the linear relation $\sum_{X\ni 0}z_X = 0$ in $A^\bullet(\Sigma_\cG)$ that the expression $\sum_{F\subseteq G \subseteq E} -z_{G\cup 0}$ for $\widetilde h_F$ can be rewritten as
\[
\widetilde h_F= \sum_{\substack{\emptyset\subseteq S \subsetneq E\\ F\not\subseteq S}} z_{S\cup 0}.
\]
Hence, the isomorphism $p\colon \Sigma_\cG \to \Sigma_E$ maps $\widetilde h_F$ to the element 
\[
\sum_{\substack{\emptyset\subseteq S \subsetneq E\\ F\not\subseteq S}} [D_S] \in A^1(X_E),
\]
which by Proposition~\ref{prop:nefpolymatroid} corresponds to $I(\U_{1,F}\oplus \U_{0,E\setminus F})$ because the rank function $\rk$ of the matroid $\U_{1,F}\oplus \U_{0,E\setminus F}$ is given by $\rk(E\setminus S) = 1$ if $F\not\subseteq S$ and 0 otherwise.
\end{proof}

For matroids $\M$ and $\M'$ on $E$, there is a dual notion to matroid intersection, \emph{matroid union}, defined by  ${\M}\vee \M' := (\M^{\perp} \wedge \M'^{\perp})^{\perp}$. The bases of $\M \vee \M'$ are the maximal elements among the unions of the basis of $\M$ and $\M'$.

\begin{proof}[Proof of Theorem~\ref{thm:schubert}]
For $\emptyset\subsetneq F \subseteq E$, let $\mathrm{H}_F$ be the corank 1 matroid whose unique circuit is $F$.  Equivalently, its dual matroid $\mathrm{H}_F^\perp$ is the matroid $\U_{1,F} \oplus \U_{0,E\setminus F}$.  We note from Proposition~\ref{prop:firstChern} and Corollary~\ref{cor:ctop} that
\[
h_F = [D_{I(\mathrm{H}_F^\perp)}] = c_1(\cQ_{\mathrm{H}_F}) = [\Sigma_{\mathrm{H}_F}].
\]
Now, applying Theorem~\ref{thm:intersectionaugmented} to Lemma~\ref{lem:schubert} yields the theorem once we show the following: For an element $h_{F_1}^{d_1}\cdots h_{F_k}^{d_k}$ in the monomial basis of $A^\bullet(X_E)$ given in Lemma~\ref{lem:schubert}, the matroid intersection
\[
\H_{F_1}^{\wedge d_1} \wedge \cdots \wedge \H_{F_k}^{\wedge d_k} = 
\underbrace{\H_{F_1} \wedge\cdots\wedge \H_{F_1}}_{d_1\ \textnormal{times}} \wedge \cdots \wedge  \underbrace{\H_{F_k} \wedge\cdots\wedge \H_{F_k}}_{d_k\ \textnormal{times}}
\]
is a Schubert matroid of corank $d_1 + \cdots + d_k$, and every Schubert matroid arises in this way.
Since the dual of a Schubert matroid is a Schubert matroid, we may instead prove the dual statement that the matroid union
\[
\underbrace{\H_{F_1}^\perp \vee\cdots\vee \H_{F_1}^\perp}_{d_1\ \textnormal{times}} \vee \cdots \vee  \underbrace{\H_{F_k}^\perp \vee\cdots\vee \H_{F_k}^\perp}_{d_k\ \textnormal{times}}
\]
is a Schubert matroid, and that every Schubert matroid of rank $d_1 + \cdots + d_k$ arises in this way.\\
\indent Since every matroid in the above matroid union is of rank 1, a basis of the matroid union is obtained by selecting $d_i$ elements of $F_i$ for each $i = 1, \ldots, k$ such that the union of all the selected elements has as large cardinality as possible.
By Remark~\ref{rem:schubertequiv}, we see that such matroid union are exactly the Schubert matroids of rank $d_1 + \cdots + d_k$.
\end{proof}

Combining Theorem~\ref{thm:valuativeiso} with Theorem~\ref{thm:schubert} recovers the following result of Derksen and Fink \cite[Theorem 5.4]{DerksenFink}.

\begin{corollary}\label{cor:DerksenFink}
Schubert matroids on $E$ of rank $r$ form a basis for $\operatorname{Val}_r(E)$.
\end{corollary}

Because Schubert matroids are realizable over any infinite field, combining Corollary~\ref{cor:ctop} and Corollary~\ref{cor:phizetaontauto} with Theorem~\ref{thm:schubert} also yields the following.

\begin{corollary}\label{cor:Lspan}
The $K$-classes $[\mathcal O_{W_L}]$ of augmented wonderful varieties span $K(X_E)$ as an abelian group.
\end{corollary}

\section{Numerical properties}

\subsection{The Hirzebruch--Riemann--Roch-type formulas}\label{subsec:fakeHRR}

We now prove Theorem~\ref{thm:fakeHRR} using Corollary~\ref{cor:Lspan}.  While one can prove Theorem~\ref{thm:fakeHRR} by mimicking the proof of \cite[Theorem D]{BEST21}, we present a proof that avoids the use of the Atiyah--Bott localization formula.  Recall the notation $\alpha = \pi_E^* c_1(\mathcal O_{\PP^E}(1))$.

\begin{proof}[Proof of Theorem~\ref{thm:fakeHRR}]
We first verify the formula involving the $\zeta$ map, i.e.,\ that 
\[
\chi \big([\mathcal E] \big) = \int \zeta\big([\mathcal E] \big)\cdot (1+\alpha + \cdots + \alpha^{n})
\]
for any $[\mathcal E]\in K(X_E)$.  Corollary~\ref{cor:Lspan} implies that it suffices to show this for the case $[\mathcal E] = [\mathcal O_{W_L}]$ for any linear subspace $L\subseteq \kk^E$.  Now, we have $\chi([\mathcal O_{W_L}]) = 1$ since $W_L$ is obtained from a projective space by a sequence of blow-ups along smooth centers.  On the other hand, using Corollary~\ref{cor:phizetaontauto} and applying the projection formula to $\pi_E$ gives that $$\int_{X_E} [W_L] \cdot (1 + \alpha + \dotsb + \alpha^n) = \int_{\mathbb{P}^E} c_1(\mathcal{O}_{\mathbb{P}^E}(1))^{\dim L} \cdot (1 + c_1(\mathcal{O}_{\mathbb{P}^E}(1)) + \dotsb + c_1(\mathcal{O}_{\mathbb{P}^E}(1))^{n}) = 1.$$  
\indent Having established the formula involving $\zeta$, we now use Serre duality to derive the formula involving $\phi$, i.e.,\
\[
\chi \big([\mathcal E] \big) = \int \phi\big([\mathcal E] \big)\cdot c \big(\bigoplus_{i\in E}\pi_i^*\mathcal O_{\PP^1}(1)\big).
\]
First, by \cite[Theorem 8.1.6]{CLS11}, the anti-canonical divisor of $X_E$ is the $\sum_{S\subsetneq E} D_S + \sum_{i\in E} D_i$, where $D_S$ denotes the torus-invariant divisor of the ray $\rho_S$, and $D_i$ that of the ray $\rho_i$ in $\Sigma_E$.  By Proposition~\ref{prop:nefpolymatroid}, one checks that $\sum_{S\subsetneq E} D_S = D_{I(\mathrm \U_{1,E})}$ and $ \sum_{i\in E} D_i = D_{I(\mathrm \U_{n,E})}$.  In summary, we have that the anti-canonical bundle $\omega_{X_E}^\vee$ of $X_E$ is
\[
\omega_{X_E}^\vee = \mathcal O_{X_E}(D_{I(\mathrm \U_{1,E})} + D_{I(\mathrm \U_{n,E})}).
\]
Corollary~\ref{cor:phizetaontauto}, in the form of Example~\ref{eg:alphays}, thus gives $\phi([\omega_{X_E}^\vee]) = (1+\alpha)\cdot c \big(\bigoplus_{i\in E}\pi_i^*\mathcal O_{\PP^1}(1)\big)$.
Applying Serre duality, along with the definition that $\zeta = D_A \circ \phi \circ D_K$, we conclude
\begin{align*}
\chi([\mathcal E]) &= (-1)^{n} \chi([\mathcal E]^\vee\cdot[\omega_{X_E}])\\
&= (-1)^{n} \int_{X_E} \zeta\big([\mathcal E]^\vee\cdot[\omega_{X_E}] \big)\cdot (1+\alpha + \cdots + \alpha^{n})\\
&= (-1)^{n} \int_{X_E} D_A\big(\phi([\mathcal E]) \cdot \phi([\omega_{X_E}^\vee])\big) \cdot (1+\alpha + \cdots + \alpha^{n})\\
&= (-1)^{n} \int_{X_E} D_A\Big(\phi([\mathcal E]) \cdot (1+\alpha) \cdot c \big(\bigoplus_{i\in E}\pi_i^*\mathcal O_{\PP^1}(1)\big) \Big) \cdot (1+\alpha + \cdots + \alpha^{n})\\
&= (-1)^{n} \int_{X_E} D_A\Big(\phi([\mathcal E]) \cdot c \big(\bigoplus_{i\in E}\pi_i^*\mathcal O_{\PP^1}(1)\big) \Big)\\
&= \int_{X_E} \phi([\mathcal E]) \cdot c \big(\bigoplus_{i\in E}\pi_i^*\mathcal O_{\PP^1}(1)\big),
\end{align*}
as desired.
\end{proof}

\subsection{Tutte polynomial formulas}
We show that two specializations of the Tutte polynomial arise as volume polynomials of augmented tautological classes. The first is the rank-generating function of a matroid, i.e., $T_\M(u + 1, v + 1)$. This computation does not show that the rank-generating function has any log-concavity property because it involves the Chern class of $[\mathcal{S}_\M]$, and Proposition~\ref{prop:minkowski} shows that $c(\mathcal{S}_\M)$ is rarely nef or anti-nef.
We also compute the intersection numbers of a second set of classes, which gives a more complicated specialization of the Tutte polynomial. This computation can be used to show that the result is Lorentzian and therefore has log-concavity properties.  Recall the notation that $y_i = \pi_i^*(c_1(\mathcal{O}_{\mathbb{P}^1}(1)))$ for $i\in E$, and let $u^I = \prod_{i \in I} u_i$ for $I \subseteq E$. 

\begin{theorem}\label{thm:SQintersect}
Let $\M$ be a matroid on $E$ of rank $r$. For $I\subseteq E$, we have
\[
\int_{X_E} c(\mathcal S_\M,z) \cdot w^{n-r} \cdot c(\mathcal Q_\M,w^{-1})  \cdot \prod_{i \in I} y_i = z^{r - \operatorname{rk}_\M(I)} w^{|I| - \operatorname{rk}_\M(I)}.
\]
In particular, summing over all $I \subseteq E$, we have that
\[
\int_{X_E} c(\mathcal S_\M,z) \cdot w^{n-r} \cdot c(\mathcal Q_\M,w^{-1}) \cdot \prod_{i = 1}^{n} (1 + y_i u_i) = \sum_{I \subseteq E} z^{r-\operatorname{rk}_\M(I)}w^{|I| -\operatorname{rk}_\M(I)} u^I.
\]
\end{theorem}

\begin{proof}
By Proposition~\ref{prop:restrict}, the restriction of the Chern classes of augemented tautological classes to $X_{{E \setminus I}}$ are the Chern classes of the augmented tautological classes of the contraction $\M/I$.  Now one notes that $\int_{X_E} c(\mathcal S_\M,z) \cdot c(\mathcal Q_\M, w) = \int_{X_E} c_r(\mathcal S_\M) \cdot z^r \cdot c_{n-r}(\mathcal Q_\M) \cdot w^{n-r} = z^rw^{n-r}$ since $[\mathcal S_\M] + [\mathcal Q_\M] = [\bigoplus_{i \in E} \pi_i^* \mathcal{O}_{\mathbb{P}^1}(1)]$.
\end{proof}

Theorem~\ref{thm:tutterank} is immediate from Theorem~\ref{thm:SQintersect}.
We now prove Theorem~\ref{thm:tutteintersection}. The proof uses the Hirzebruch--Riemann--Roch-type formulas for both $\zeta$ and $\phi$ to obtain the equality of certain intersection numbers.
We first state a combinatorial lemma that will be used twice in the proof of Theorem~\ref{thm:tutteintersection}.

\begin{lemma}\label{lem:tuttecontraction}
Let $\M$ be a matroid of rank $r$ on $E$. Then
$$\sum_{I \subseteq E} a^{|I|} b^{r - \operatorname{rk}_\M(I)} c^{n - |I| - r + \operatorname{rk}_\M(I)} T_{\M/I}\left(\frac{d}{b}, \frac{b+c}{c}\right) = (a + b)^r c^{n-r} T_\M\left(\frac{a + d}{a+b}, \frac{a + b + c}{c}\right).$$
\end{lemma}

\begin{proof}
Using the rank generating function for the Tutte polynomial, we compute
\begin{equation*}
\begin{split}
&\sum_{I \subseteq E} a^{|I|} b^{r - \operatorname{rk}_\M(I)} c^{n - |I| - r + \operatorname{rk}_\M(I)} T_{\M/I}\left(\frac{d}{b}, \frac{b+c}{c}\right) \\
&= \sum_{I \subseteq E} a^{|I|} b^{r - \operatorname{rk}_\M(I)} c^{n - |I| - r + \operatorname{rk}_\M(I)} \sum_{J \supseteq I} \left(\frac{d - b}{b} \right)^{r - \operatorname{rk}_\M(J)} \left(\frac{b}{c} \right)^{|J| - |I| - \operatorname{rk}_\M(J) + \operatorname{rk}_\M(I)} \\
&= \sum_{I \subseteq J \subseteq E} a^{|I|} b^{|J| - |I|} c^{n - r - |J| + \operatorname{rk}_\M(J)} (d-b)^{r - \operatorname{rk}_\M(J)} \\
&= \sum_{J \subseteq E} b^{|J|} c^{n - r - |J| + \operatorname{rk}_\M(J) } (d-b)^{r - \operatorname{rk}_\M(J)} \sum_{I \subseteq J} a^{|I|}b^{-|I|} \\
& = \sum_{J \subseteq E} b^{|J|} c^{n - r - |J| + \operatorname{rk}_\M(J) } (d-b)^{r - \operatorname{rk}_\M(J)} \left(\frac{a + b}{b} \right)^{|J|} \\
&= (a + b)^{r} c^{n-r} \sum_{J \subseteq E} \left( \frac{d -b}{a + b}\right)^{r - \operatorname{rk}_\M(J)} \left(\frac{a + b}{c} \right)^{|J| - \operatorname{rk}_\M(J)} \\
&= (a + b)^r c^{n-r} T_\M\left(\frac{a + d}{a+b}, \frac{a + b + c}{c}\right),
\end{split}
\end{equation*}
as desired.
\end{proof}

\begin{proof}[Proof of Theorem~\ref{thm:tutteintersection}]
Note that $s(\pi_E^* \mathcal{O}_{\mathbb{P}^E}(-1), x) = 1 + \alpha x + \alpha^2 x^2 + \dotsb$.  We prove the result in three steps.\\
\\
\emph{Step 1}: We show that
\begin{equation}\label{eq:step1}
\int_{X_E} s(\mathcal Q_\M^\vee,z) \cdot c(\mathcal Q_\M,w) = z^r w^{n-r} T_\M(0, 1 +  \textstyle\frac{z}{w}).
\end{equation}
As $[\mathcal{S}_\M] + [\mathcal{Q}_\M] = [\bigoplus_{i \in E} \pi_i^* \mathcal{O}_{\mathbb{P}^1}(1)]$, we have $s(\mathcal Q_\M^\vee, z) =  c(\bigoplus_{i \in E} \pi_i^* \mathcal O(-1), z)^{-1}\cdot c(\mathcal S_\M^\vee, z) =  c(\bigoplus_{i \in E} \pi_i^*\mathcal{O}_{\mathbb{P}^1}(1),z) \cdot c(\mathcal S_\M^\vee, z)$.  We compute

\begin{equation*}
\begin{split}
\int_{X_E} s(\mathcal Q_\M^\vee,z) \cdot c(\mathcal Q_\M,w) &= \int_{X_E} c(\bigoplus_{i \in E} \pi_i^* \mathcal{O}_{\mathbb{P}^1}(1), z) \cdot c(\mathcal S_\M^\vee,z)  \cdot c(\mathcal Q_\M,w)\\
&= \int_{X_E} \sum_{I \subseteq E} (\prod_{i\in I}y_i) \cdot z^{|I|} \cdot c(\mathcal S_\M,-z) \cdot c(\mathcal Q_\M,w)\\
&= \sum_{I\subseteq E} z^{|I|} \int_{X_{{E\setminus I}}} c(\mathcal S_{\M/I},-z) \cdot c(\mathcal Q_{\M/I},w)\\
&= \sum_{I\subseteq E} z^{|I|} (-z)^{r-\operatorname{rk}_\M(I)} w^{n-|I| - (r-\operatorname{rk}_\M(I))}\\
&= z^rw^{n-r} \sum_{I\subseteq E} (-1)^{r-\operatorname{rk}_\M(I)} (z/w)^{|I|-\operatorname{rk}_\M(I)} = z^rw^{n-r} T_\M(0,1+ \textstyle\frac{z}{w}).
\end{split}
\end{equation*}
\\
\emph{Step 2}:  We show that
\begin{equation}\label{eq:step2}
\int_{X_E} (1+\alpha x + \alpha^2x^2 + \dotsb) \cdot s(\mathcal Q_\M^\vee, z) \cdot  c(\mathcal Q_\M, w)  = z^r(x + w)^{n-r} T_\M\Big(\frac{x}{z}, \frac{x + z + w}{x + w}\Big).
\end{equation}
As the result is homogeneous, it suffices to prove the claimed formula after evaluating $x = 1$. We compute $\chi((\sum_{i \ge 0} \wedge^i [\mathcal{Q}_\M]^{\vee} w^i)(\sum_{j \ge 0} \operatorname{Sym}^j [\mathcal{Q}_\M]^{\vee} z^j ))$ in two different ways, using Proposition~\ref{prop:simpleChern} and the Hirzebruch--Riemann--Roch-type formulas for both $\zeta$ and $\phi$. We then get that
\begin{equation*} \begin{split} &\int_{X_E} (1 + \alpha + \alpha^2 +  \dotsb) \cdot (w + 1)^{n-r} \cdot c \left(\mathcal{Q}_\M^{\vee}, \frac{w}{w + 1}\right ) \cdot (1-z)^{r -n} \cdot s \left (\mathcal{Q}_\M^{\vee}, \frac{z}{z-1} \right) \\
&= \int_{X_E} c(\bigoplus_{i \in E} \pi_i^* \mathcal{O}_{\mathbb{P}^1}(1)) \cdot (w + 1)^{n-r} \cdot c \left(\mathcal{Q}_\M, \frac{1}{w + 1} \right) \cdot (1-z)^{r-n} \cdot s\left (\mathcal{Q}_\M, \frac{1}{1-z} \right).\end{split}\end{equation*}
Replacing $w$ by $-w/(w + 1)$ and $z$ by $z/(z-1)$ and cancelling common terms, we obtain that
\begin{equation*} \begin{split} &\int_{X_E} (1 + \alpha + \dotsb) \cdot c(\mathcal{Q}_\M, w) \cdot  s(\mathcal{Q}_\M^{\vee}, z) 
= \int_{X_E} c(\bigoplus_{i \in E} \pi_i^* \mathcal{O}_{\mathbb{P}^1}(1)) \cdot c(\mathcal{Q}_\M, w + 1) \cdot s(\mathcal{Q}_\M, 1-z).\end{split}\end{equation*}
Now we apply (\ref{eq:step1}), noting $s(\cQ_\M,1-z) = s(\cQ_\M^\vee,z-1)$, to obtain that
$$\int_{X_E} c(\mathcal{Q}_\M, w + 1) \cdot s(\mathcal{Q}_\M^{\vee}, z-1) = (z -1)^r (w+1)^{n-r} T_\M \left (0, \frac{z + w}{w + 1} \right).$$
Arguing as in Step 1 and using Proposition~\ref{prop:restrict}, the above equation implies that 
\begin{equation*}\begin{split}
& \int_{X_E} \prod_{i\in E}(1 + y_i u_i) \cdot c(\mathcal{Q}_\M, w + 1) \cdot s(\mathcal{Q}_\M^{\vee}, z-1)  
\\ &= \sum_{I \subseteq E} u^I (z - 1)^{r - \operatorname{rk}_\M(I)} (w+1)^{n- |I| - r + \operatorname{rk}_\M(I)} T_{\M/I}\left(0, \frac{z + w}{w + 1} \right).
\end{split}\end{equation*}
Setting each $u_i$ to  $1$ and using that $\prod(1 + y_i) = c(\bigoplus \pi_i^* \mathcal{O}_{\mathbb{P}^1}(1))$, we get that
\begin{equation*}\begin{split}
&\int_{X_E} c(\bigoplus_{i \in E} \pi_i^* \mathcal{O}_{\mathbb{P}^1}(1)) \cdot c(\mathcal{Q}_\M, w + 1) \cdot s(\mathcal{Q}_\M^{\vee}, z-1) \\ &= \sum_{I \subseteq E}  (z - 1)^{r - \operatorname{rk}_\M(I)} (w+1)^{n- |I| - r + \operatorname{rk}_\M(I)} T_{\M/I} \left (0, \frac{z + w}{w + 1} \right).
\end{split}\end{equation*}
Applying Lemma~\ref{lem:tuttecontraction} with $a = 1$, $b = z - 1$, $c = w + 1$, and $d = 0$, we obtain (\ref{eq:step2}).\\
\\
\emph{Step 3}: We finish the computation. We have that
\begin{equation*}\begin{split}
&\int_{X_E} (1 + \alpha x + \alpha^2 x^2 + \dotsb) \cdot  c(\bigoplus_{i \in E} \pi_i^* \mathcal{O}_{\mathbb{P}^1}(1), y) \cdot s(\mathcal{Q}_\M^{\vee}, z) \cdot c(\mathcal{Q}_\M, w) \\
&= \sum_{I \subseteq E} y^{|I|}\int_{X_{{E \setminus I}}}(1 + \alpha x + \alpha^2 x^2 + \dotsb) \cdot s(\mathcal{Q}_{\M/I}^{\vee}, z) \cdot c(\mathcal{Q}_{\M/I}, w) \\
&= \sum_{I \subseteq E} y^{|I|} z^{r - \operatorname{rk}_\M(I)}(x+w)^{n - |I| -r + \operatorname{rk}_\M(I)} T_{\M/I}\Big(\frac{x}{z}, \frac{x+ z+ w}{x+w}\Big).
\end{split} \end{equation*}
Then the result follows from Lemma~\ref{lem:tuttecontraction} with $a = y$, $b = z$, $c = x + w$, and $d = x$. 
\end{proof}

\subsection{Positivity properties}\label{subsec:positive}

We now use Theorem~\ref{thm:tutteintersection} to prove Theorem~\ref{thm:tuttelorentzian}, which states that the 4-variable transformation of the Tutte polynomial in Theorem~\ref{thm:tutteintersection} is a denormalized Lorentzian polynomial.  Let us begin by reviewing the language of Lorentzian polynomials developed in \cite{BH20}.

For a homogeneous degree $d$ polynomial $f = \sum_{u \in \ZZ_{\geq0}^m} a_{u} x^{u} \in \RR[x_1, \ldots, x_m]$, its \emph{normalization} is $N(f) = \sum_{u \in \ZZ_{\geq0}^{m}} a_{u} \frac{x^{u}}{u!}$ where $u! = u_1! \cdots u_m!$.  The polynomial $f$ is said to be the \emph{denormalization} of $N(f)$.  The polynomial $f$ is a \emph{strictly Lorentzian polynomial} if every monomial of degree $d$ has a positive coefficient and every $(d-2)$-th coordinate partial derivative of $f$ is a quadric form with signature $(+,-,-, \ldots, -)$.  It is a \emph{Lorentzian polynomial} if $f$ is a limit of strictly Lorentzian polynomials.
Lorentzian polynomials satisfy a strong log-concavity property \cite[Example 2.26]{BH20} and are preserved under nonnegative linear change of variables \cite[Theorem 2.10]{BH20}.  Polynomials whose normalization is Lorentzian, called \emph{denormalized Lorentzian polynomials}, share similar properties \cite[\S4.3]{BLP}.

We now place the strategy used in the proof of \cite[Theorem 9.13]{BEST21} into an axiomatic framework and use the framework to deduce the theorem.
The key tool will be the theory of \emph{Lefschetz fans}, a notion introduced in \cite[Definition 1.5]{ADH22}. Lefschetz fans are certain (possibly non-complete) simplicial quasi-projective balanced fans whose Chow ring satisfies an analogue of the K\"{a}hler package. We summarize their fundamental properties.

\begin{theorem}\label{thm:lef}
The following hold.
\begin{enumerate}[label = (\arabic*)]\itemsep 5pt
    \item\label{lef:supp} \cite[Theorem 1.6]{ADH22} If $\Sigma$ is a Lefschetz fan, then any quasi-projective simplicial fan with the same support as $\Sigma$ is Lefschetz. 
    \item\label{lef:prod} \cite[Lemma 5.27]{ADH22} A product of Lefschetz fans is Lefschetz. 
    \item \label{lef:Bergman} \cite[Theorem 8.9]{AHK18} The Bergman fan of a loopless matroid is Lefschetz. 
    \item \label{lef:lorentzian} \cite[Theorem 4.6]{BH20}, \cite[Theorem 5.20]{ADH22}, see also \cite[Lemma 9.12]{BEST21} Let $\Sigma$ be an $\ell$-dimensional smooth projective fan, and let $\Sigma'$ be a $d$-dimensional subfan that is Lefschetz and defines the Minkowski weight $[\Sigma'] \in A^{\ell-d}(X_{\Sigma})$ as a balanced fan. Then, for any base-point-free divisors $D_1, \ldots, D_m \in A^1(X_{\Sigma})$, the polynomial 
    $$\sum_{i_1 + \dotsb + i_m = d} \left(\int_{X_{\Sigma}} D_1^{i_1} \cdot \dotsb \cdot D_m^{i_m} \cdot [\Sigma']\right)x_1^{i_1} \dotsb x_m^{i_m}$$
    is denormalized Lorentzian. 
\end{enumerate}
\end{theorem}

Let us now set up the axiomatic framework.  For a finite set $S$, denote
\[
\mathsf{Mat}^\circ_{S} = \text{the set of loopless and coloopless matroids with ground set $S$.}
\]
We say that a map $\varphi \colon \mathsf{Mat}^\circ_S \to G$ taking values in an abelian group $G$ is valuative if it is a restriction to $\mathsf{Mat}^\circ_S$ of a valuative map on the set of all matroids on $S$.
Let $N$ be a nonnegative integer that depends on $n$ (e.g.\ $N = 2n$), and let $[N] = \{1, \ldots, N\}$.
Our framework consists of three objects $(F,T,X)$:
\begin{enumerate}[label = $\bullet$]\itemsep 5pt
\item a map $F_{(\cdot)}\colon \mathsf{Mat}^\circ_{E} \to \mathsf{Mat}^\circ_{[N]}$,
\item a torus $T$ with an action on $\kk^N$ via a map $\varphi \colon T \to \mathbb{G}_{m}^N$, and
\item a smooth projective $T$-variety $X$ with a dense open $T$-orbit $\overline T$ (which is a quotient torus of $T$), such that $\varphi$ naturally descends to $\overline \varphi \colon \overline T\to \GG_m^N/\GG_m$.
\end{enumerate}

We require that these objects satisfy the following properties:
\begin{enumerate}[label = (\roman*)]\itemsep 5pt
\item The assignment $\M\mapsto [\underline{\Sigma}_{F_\M}]$, sending a matroid $\M$ on $E$ to the Bergman class of the matroid $F_\M$ on $[N]$, is valuative.
\item There is a map
\[
F^\kk_{(\cdot)}\colon \coprod_{r=0}^n Gr(r;E)(\kk) \to \coprod_{R=0}^N Gr(R;[N])(\kk)
\]
such that for any realization $L\subseteq \kk^{E}$ of $\M\in \mathsf{Mat}_E^\circ$, the matroid $F_{\M}$ equals the matroid on $[N]$ realized by  $F^\kk_L$.  We often abuse notation and write $F$ for $F^\kk$ also.
\item For any $L\subseteq \kk^{E}$,  specifying the fibers over $\overline t \in \overline T$ to be $\varphi(t^{-1}) F_L$ defines a $T$-equivariant vector subbundle $\mathcal F_L$ of $\mathcal{O}_{X}^{\oplus N}$ on $X$.
\item The Segre class $s(\mathcal F_L) \in A^\bullet(X)$ depends only on the matroid that $L$ realizes.
\item The assignment $\M\mapsto s(\mathcal F_{(L\text{ realizing } \M)})$ from the set of $\kk$-realizable matroids in $\mathsf{Mat}^\circ_E$ to $A^\bullet(X)$ is valuative.
\end{enumerate}
Because every matroid in $\mathsf{Mat}^\circ_E$ is valuatively equivalent to a linear combination of $\kk$-realizable matroids in $\mathsf{Mat}^\circ_E$ \cite[Lemma 5.9]{BEST21}, 
the conditions (iv) and (v) imply that we have a unique valuative extension $\M\mapsto s(\mathcal F_\M) \in A^\bullet(X)$ such that $s(\mathcal F_\M) = s(\mathcal F_L)$ whenever $L$ realizes $\M$.
Thus, we may define the following.

\begin{definition}
With $F$, $T$, and $X$ satisfying the conditions above, for a matroid $\M\in \mathsf{Mat}^\circ_E$ we define $[\mathbb{P}(\mathcal F_\M)] \in A^\bullet(X\times \mathbb{P}^{N-1})$ by
\[
[\mathbb{P}(\mathcal F_\M)] = \sum_{i = 0}^{N-R} s_i(\mathcal F_\M) \delta^{N-R-i}
\]
where $R$ is the rank of $F_\M$ and $\delta = c_1(\mathcal O(1))$ is the hyperplane class of $\mathbb{P}^{N-1}$ pulled back to $X\times \mathbb{P}^{N-1}$.
\end{definition}

When $\M$ is realized by $L \subseteq \kk^E$, then $[\mathbb{P}(\mathcal{F}_\M)] = [\mathbb{P}(\mathcal{F}_L)]$ by \cite[Proposition 9.13]{EH}.

\begin{example}
In the setting of \cite{BEST21}, we let $n=N$ with $T = \mathbb{G}_{m}^E$ acting on $X = \underline X_{E}$ naturally via $T \to \mathbb{P} T$, and acting on $\kk^{E}$ by the inverse standard action.  If we set $F$ to be the identity map, which satisfies the conditions listed above, we then have $\mathcal F_L = \underline{\mathcal S}_L$.   If we set $F$ to be the matroid duality map (i.e.,\ $\M\mapsto \M^\perp$ and $L\mapsto L^\perp$), which also satisfies the conditions listed above, we then have $\mathcal F_L = \underline{\mathcal Q}^\vee_L$.
\end{example}

\begin{example}
Let $N = 2n$, and let $T = \mathbb{G}_{m}^E$ act on $\kk^{E} \times \kk^E$ by $t \cdot (x,y) = (t^{-1}x,y)$, and act on $X_E$ as its open dense torus.  Let $pre'F$ be the map that adds parallel element to each element in a matroid $\M$ on $E$ to get a matroid $pre'F_\M$ on $E\sqcup E$.  Note that $\M \mapsto [\underline\Sigma_{pre'F_\M}]$ is valuative, since $[\underline\Sigma_{pre'F_\M}]$ is the image of $[\underline\Sigma_\M]$ under the diagonal embedding $x \mapsto (x,x)$.  In fact, the map $\M\mapsto pre'F_\M$ itself is valuative.  If we set $F$ to be $pre'F$ precomposed with matroid duality map, we then have $\mathcal F_L = \mathcal Q_L^\vee$.  If we set $F$ to be $pre'F$ precomposed and then post-composed with matroid duality maps (note that one duality takes place on $E$ and the other on $E \sqcup E$), we get $\mathcal F_L = \mathcal K_L$, where $\mathcal K_L$ is defined by the exact sequence
$$0 \to \mathcal{K}_L \to \mathcal{O}_{X_{E}}^{\oplus E} \oplus \mathcal{O}_{X_{E}}^{\oplus E} \to \mathcal{Q}_L \to 0.$$
Note that the $K$-class $[\mathcal{K}_L]$ depends only on the matroid that $L$ represents because $[\mathcal{K}_L] = [\mathcal{O}_{X_{E}}^{\oplus E} \oplus \mathcal{O}_{X_{E}}^{\oplus E}] - [\mathcal{Q}_L]$. Note also that $s(\mathcal{K}_L) = c(\mathcal{Q}_L)$.

\end{example}

\begin{theorem}\label{thm:constructLefschetz}
Under the conditions above, there exists a smooth projective $(\overline T \times \mathbb{G}_{m}^N/\mathbb{G}_{m})$-toric variety $Y_\Sigma$ with a birational toric morphism $\pi \colon Y_\Sigma \to X\times \mathbb{P}^{N-1}$ such that for every matroid $\M\in \mathsf{Mat}^\circ_E$, there exists a Lefschetz subfan $\Sigma_{X,F_{\M}}$ of $\Sigma$ such that $\pi_*[\Sigma_{X,F_{\M}}] = [\mathbb{P}(\mathcal F_{\M})]$, where $[\Sigma_{X,F_{\M}}]$ denotes the Chow cohomology class on $Y_\Sigma$ that is Poincar\'{e} dual to the Minkowski weight of constant weight 1 on the Lefschetz fan $\Sigma_{X,F_{\M}}$.
\end{theorem}

\begin{proof}
First, we set the birational toric morphism $\pi$ restricted to the tori to be given by $(t,t') \mapsto (t, \varphi(t) t')$.  Now, we can take $\Sigma$ to be any unimodular projective fan inside $\operatorname{Cochar}(\overline T)_\mathbb{R} \times (\mathbb{R}^N/\mathbb{R})$ such that it refines $\text{(the fan of }X) \times \underline\Sigma_{[N]}$ and makes $Y_\Sigma \to X\times \mathbb{P}^{N-1}$ into a valid toric morphism.\\
\indent We take $\Sigma_{X,F_\M}$ to be the subfan of $\Sigma$ with support $\operatorname{Cochar}(\overline T)_\mathbb{R} \times \underline\Sigma_{F_\M}$.  By Theorem~\ref{thm:lef}.\ref{lef:Bergman}, the support of the fan $\Sigma_{X, F_\M}$ is equal to the support of a product of two Lefschetz fans, and hence by Theorem~\ref{thm:lef}.\ref{lef:supp} and \ref{lef:prod}, $\Sigma_{X, F_\M}$ is a Lefschetz fan.  By the assumptions, the assignment $\M\mapsto [\mathcal F_\M]$ and the assignment $\M\mapsto [\mathbb{P}(\mathcal F_\M)]$ are valuative.  On the other hand, the assumption that $\M\mapsto [\underline\Sigma_{F_\M}]$ is valuative implies that $\M\mapsto [\Sigma_{X,F_\M}]$ is also valuative.  Thus, for the desired equality $\pi_*[\Sigma_{X,F_\M}] = [\mathbb{P}(\mathcal F_\M)]$, it suffices to show it when $\M$ has a $\kk$-realization $L$.\\ 
\indent For a loopless matroid $\M'$ on a set $E'$ realized by a linear subspace $L'\subseteq \kk^{E'}$, the Minkowski weight with constant weight 1 on the Bergman fan $\Sigma_{\M'}$ is the tropicalization of $\PP(L') \cap \GG_m^{E'}/\GG_m$ \cite{Stu02, AK06}.  Hence, the Minkowski weight with constant weight 1 on $\Sigma_{X,F_\M}$ is the tropicalization of $\overline T \times (\mathbb{P}(F_L) \cap \mathbb{G}_{m}^N/\mathbb{G}_{m})$, so the Chow class $[\Sigma_{X,F_\M}]$ equals the class of the closure of $\overline T \times (\mathbb{P}(F_L) \cap \mathbb{G}_{m}^N/\mathbb{G}_{m})$ inside $Y_\Sigma$.  On the other hand, by construction the map $\pi$ bijectively maps $\overline T \times (\mathbb{P}(F_L) \cap \mathbb{G}_{m}^N/\mathbb{G}_{m})$ to an open subset of $\mathbb{P}(\mathcal F_L)$, an irreducible subvariety of $X\times \mathbb{P}^{N-1}$. Then the result follows.
\end{proof}

\begin{remark}
If there are several maps $F^{(1)}, \ldots, F^{(k)}$ from $\mathsf{Mat}^\circ_{E}$ to $\mathsf{Mat}^\circ_{[N^{(k)}]}$, each satisfying the conditions listed above with a common $X$ and $T$ fixed throughout, the theorem easily generalizes to the multi-projectivization $[\mathbb{P}(\mathcal F_\M^{(1)})\times_X \cdots \times_X \mathbb{P}(\mathcal F_{\M}^{(k)})]$.
\end{remark}

\begin{proof}[Proof of Theorem~\ref{thm:tuttelorentzian}]
First we assume that $\M$ is loopless and coloopless. 
Note the $\mathcal{Q}_L^{\vee}$ embeds into $\mathcal{O}_{X_E}^{\oplus E \sqcup E}$ because $\bigoplus_{i \in E} \pi_i^* \mathcal{O}(-1)$ does, and we can apply Theorem~\ref{thm:constructLefschetz} to this embedding. 
Therefore there is a smooth projective toric variety $Y_{\Sigma}$ with torus $\mathbb{G}_m^E \times \mathbb{G}_m^{E \sqcup E}/\mathbb{G}_m \times \mathbb{G}_m^{E \sqcup E}/\mathbb{G}_m$, 
a map $\pi \colon Y \to X_E \times \mathbb{P}^{2n - 1} \times \mathbb{P}^{2n - 1}$, and a Lefschetz subfan $\Sigma_{X_E, \M}$ of $\Sigma$ such that $\pi_*[\Sigma_{X_E, \M}] = [\mathbb{P}(\mathcal{K}_\M)  \times_{X_E} \mathbb{P}(\mathcal{Q}^{\vee}_\M)]$. 
Let $\delta$ and $\epsilon$ be the first Chern classes of the pullbacks of $\mathcal{O}(1)$ to $X_E \times \mathbb{P}^{2n - 1} \times \mathbb{P}^{2n - 1}$ from the two projective spaces. Then, with the shorthand $\frac{1}{1-a} = 1 + a + a^2 + \cdots + a^n$, we have
\begin{equation*}\begin{split}
&\int_{X_E} \frac{1}{1-\alpha x} \cdot c(\bigoplus_{i \in E} \pi_i^* \mathcal{O}_{\mathbb{P}^1}(1), y) \cdot s(\mathcal{Q}_\M^{\vee}, z) \cdot c(\mathcal{Q}_\M, w) \\
&= \int_{X_E \times \mathbb{P}^{2n - 1} \times \mathbb{P}^{2n - 1}}\frac{1}{1-\alpha x} \cdot c(\bigoplus_{i \in E} \pi_i^* \mathcal{O}_{\mathbb{P}^1}(1), y) \cdot\frac{\delta^{n + r - 1}}{1 - \delta z}  \cdot
\frac{\epsilon^{n - r - 1}}{1 - \epsilon w}\cdot[\mathbb{P}(\mathcal{K}_\M) \times_{X_E} \mathbb{P}(\mathcal{Q}_\M^{\vee})] \\
&= \int_{Y_{\Sigma}}\frac{1}{1 - \pi^*\alpha x}\cdot \pi^*c(\bigoplus_{i \in E} \pi_i^* \mathcal{O}_{\mathbb{P}^1}(1), y)\cdot\frac{\pi^* \delta^{n + r -1}}{1 -  \pi^*\delta z}\cdot \frac{\pi^* \epsilon^{n - r - 1}}{1 - \pi^*\epsilon w  } \cdot [\Sigma_{X_E, \M}], 
\end{split}\end{equation*}
where we have used $\alpha$ and $c(\bigoplus_{i \in E} \pi_i^* \mathcal{O}_{\mathbb{P}^1}(1))$ to refer also to their pullbacks to $X_E \times \mathbb{P}^{2n - 1} \times \mathbb{P}^{2n - 1}$. Then the result follows from Theorem~\ref{thm:lef}.\ref{lef:lorentzian}, using that $c(\bigoplus_{i \in E} \pi_i^* \mathcal{O}(1))$ is the Chern class of a direct sum of nef line bundles.\\
\indent Any matroid $\M$ of rank $r$ on $E$ can be written as the direct sum of matroids $\U_{0, j} \oplus \U_{\ell, \ell} \oplus \M'$, where $\M'$ is a loopless and coloopless of rank $r - \ell$ on a ground set of size $n - j - \ell$. Because the Tutte polynomial is multiplicative for direct sums of matroids, we have that 
\begin{multline*}
(y + z)^{r} (x + w)^{n-r} T_\M\left( \frac{x + y}{y + z}, \frac{x + y + z + w}{x + w}\right) = \\
(x + y + z + w)^j (x + y)^\ell (y + z)^{r - \ell } (x + w)^{n- j - r} T_{\M'}\left( \frac{x + y}{y + z}, \frac{x + y + z + w}{x + w}\right).
\end{multline*}
By \cite[Corollary 3.8]{BH20}, products of denormalized Lorentzian polynomials are denormalized Lorentzian, which implies the result. 
\end{proof}

\begin{remark}\label{rmk:strengthen}
One can obtain stronger log-concavity results by replacing $c(\bigoplus_{i \in E} \pi_i^* \mathcal{O}_{\mathbb{P}^1}(1), y)$ with $\prod_{i \in E}(1 + y_i u_i)$ to obtain a Lorentzian polynomial in $n + 3$ variables $x, z, w, u_1, \dotsc, u_n$. 
Using that specializations of Lorentzian polynomials are Lorentzian \cite[Theorem 2.10]{BH20}, we obtain that the polynomial $t_{\M}(x, y,z,w)$ in Theorem~\ref{thm:tuttelorentzian} is Lorentzian after each $x^a y^b z^c w^d$ term is replaced by $\frac{x^a y^b z^c w^d}{a!c!d!}$. By setting $x=z=0$, this gives a new proof of \cite[Corollary 9]{HSW}. 
\end{remark}

\section{Chern--Schwartz--MacPherson classes}\label{sec:csm}

\subsection{Log tangent bundles}
There is a natural log structure on $X_E$ obtained by viewing it as a simple normal crossings (snc) compactification of $\mathbb{A}^E$; let $\partial X_{E}$ denote the boundary divisor. Note that this is not the usual log structure on a toric variety. We obtain a log structure on $W_L$ for any linear space $L$ by declaring the inclusion $W_L \hookrightarrow X_E$ to be strict. Equivalently, we view $W_L$ as an snc compactification of $L$.  Let $\partial W_L$ be the boundary divisor of $W_L$; note that $\partial W_L = \partial X_{E} \cap W_L$. For an snc pair $(X, D)$ (i.e., a smooth variety $X$ with an snc divisor $D$) over $\kk$, we use $\Omega^1_X(\log D)$ to denote the log cotangent bundle of $(X, D)$ over $\kk$, and $\mathcal{T}_X(- \log D) := \Omega^1_X(\log D)^{\vee}$ to denote the log tangent bundle. 
Recall that we identified $\mathcal{Q}_L|_{W_L}$ with $N_{W_L/X_E}$ in Corollary~\ref{cor:koszul}.

\begin{lemma}\label{lemma:normalsequence}
 Let $\iota \colon Y \hookrightarrow X$ be an inclusion of smooth varieties over $\kk$, and let $D$ be an snc divisor on $X$ such that $(Y, D \cap Y)$ is an snc pair. Then there is an exact sequence
$$0 \to \mathcal{T}_{Y}(- \log D|_{Y}) \to \iota^* \mathcal{T}_{X}(-\log D) \to N_{Y/X} \to 0,$$
where $N_{Y/X}$ is the normal bundle of $Y \hookrightarrow X$. If a group scheme $G$ acts on $X$ preserving $D$ and $Y$, then this is an exact sequence of $G$-equivariant sheaves. 
\end{lemma}

\begin{proof}
By \cite[{1.1(iii)}]{Olsson}, we have that $L_{Y/S}, L_{X/S}$ are $\Omega^1_{Y}(- \log D|_{Y})$, $\Omega^1_{X}(- \log D)$. By \cite[{1.1(ii)}]{Olsson}, $L_{Y/X}$ can be identified with $N^{\vee}_{Y/X}[1]$. Then the result follows from \cite[{1.1(v)}]{Olsson} and dualizing. The last statement follows from functoriality. 
Alternatively, one can deduce the lemma from the map of short exact sequences
\[
\begin{tikzcd}
0 \ar[r] &\Omega_X|_Y \ar[r] \ar[d]&\Omega_X(\log D)|_Y \ar[r] \ar[d]&\bigoplus_i \mathcal O_{D_i}|_Y \ar[d,equal]\ar[r] & 0\\
0 \ar[r] &\Omega_Y \ar[r] &\Omega_Y(\log D|_Y) \ar[r] &\bigoplus_i \mathcal O_{D_i|_Y} \ar[r] & 0
\end{tikzcd}
\]
by applying the snake lemma.
\end{proof}

\begin{theorem}\label{thm:logtangent}
As an $L$-equivariant sheaf, $\mathcal{T}_{W_L}(-\log \partial W_L)$ can be identified with $\mathcal{S}_L|_{W_L}$, in such a way that the exact sequence $0 \to \mathcal{S}_L \to \bigoplus_{i \in E} \pi_i^* \mathcal{O}_{\mathbb{P}^1}(1) \to \mathcal{Q}_L \to 0$ restricts to the exact sequence $0 \to \mathcal{T}_{W_L}(-\log \partial W_L) \to \iota^* \mathcal{T}_{X_E}(-\log \partial X_{E}) \to N_{W_L/X_E} \to 0$. 
\end{theorem}

Theorem~\ref{thm:logtangent} is closely related to \cite[Theorem 8.8]{BEST21}. The $\mathbb{G}_m$-equivariant structure on $S_L|_{W_L}$ is different from the $\mathbb{G}_m$-equivariant structure on $\mathcal{T}_{W_L}(-\log \partial W_L)$ in general.

\begin{proof}
First we do the case of $n=1$, in which case the stellahedron $\Pi_1$ is the interval $[0,1]$.  In other words, we have $\mathbb{P}^1$ with the log structure given by the divisor $\partial \mathbb{P}^1 = \infty$, where $\infty$ is the point $[1:0]\in \PP^1$.  The exact sequence
$$0 \to \mathcal{O}(-2) \to \Omega^1_{\mathbb{P}^1}(\log \partial \mathbb{P}^1) \to \mathcal{O}_{\infty} \to 0$$
implies that $\mathcal{T}_{\mathbb{P}^1}(-\log \partial \mathbb{P}^1)$ is isomorphic to $\mathcal{O}_{\mathbb{P}^1}(1)$. By \cite[Proposition 2.3]{HassettTschinkel}, there is a unique $\mathbb{G}_a$-equivariant structure on $\mathcal{O}_{\mathbb{P}^1}(1)$, so $\mathcal{T}_{\mathbb{P}^1}(-\log \partial \mathbb{P}^1)$ is isomorphic to $\mathcal{O}_{\mathbb{P}^1}(1)$ with the $\mathbb{G}_a$-equivariant structure described in \S\ref{subsec:flagvar}.
As the formation of the log tangent bundle behaves well with respect to products, the log tangent bundle of $(\mathbb{P}^1)^E$ (viewed as a compactification of $\mathbb{A}^E$) is $\boxplus_{i \in E} \mathcal{O}_{\mathbb{P}^1}(1)$, with the induced $\mathbb{G}_a^E$-equivariant structure. 
Now, since $X_E \to (\mathbb{P}^1)^E$ is a composition of blow-ups at the boundary, the pullback $\bigoplus_{i \in E} \pi_i^* \mathcal{O}_{\mathbb{P}^1}(1)$ of $\boxplus_{i \in E} \mathcal{O}_{\mathbb{P}^1}(1)$ is isomorphic to the log-tangent bundle of $X_E$ as $\mathbb{G}_a^E$-equivariant sheaves (see, for example, the proof of \cite[Lemma 2.1]{Brion2009}).\\
\indent Now we do the general case. By Lemma~\ref{lemma:normalsequence}, it suffices to see that the following square commutes, as that will identify $\mathcal{S}_L|_{W_L}$ with the kernel of the map $\mathcal{T}_{X_E}(- \log \partial X_{E})|_{W_L} \to N_{W_L/X_E}$.
\begin{center}
\begin{tikzcd}
\bigoplus_{i \in E} \pi_i^* \mathcal{O}_{\mathbb{P}^1}(1)|_{W_L} \arrow[r] \arrow[d]
& \mathcal{Q}_L|_{W_L} \arrow[d] \\
\mathcal{T}_{X_E}(- \log \partial X_{E})|_{W_L} \arrow[r]
& N_{W_L/X_E} 
\end{tikzcd}
\end{center}
It suffices to check that this diagram commutes after restricting to a dense open subset. As the top and bottom maps are maps of $L$-equivariant sheaves, it suffices to note that this diagram commutes on the fiber over $0 \in \mathbb{A}^E$. At the fiber over $0$, both horizontal maps can be identified with the natural projection $\kk^E \to \kk^E/L$, and the vertical maps with the identity. 
\end{proof}

\subsection{Chern--Schwartz--MacPherson classes of matroid Schubert varieties}\label{ssec:CSM}
First we review the theory of Chern--Schwartz--MacPherson (CSM) classes. 
As CSM classes are defined only for varieties over a field of characteristic zero, we fix $\kk = \mathbb{C}$ and work with singular homology instead of Chow. Then, for any locally closed subset $Z$ of a proper variety $X$, there is a homology class $c_{SM}(\mathbf{1}_Z) \in H_{\bullet}(X, \mathbb{Z})$. If $X$ is smooth and $Z = X$, then the CSM class agrees with the Poincar\'{e} dual of the total Chern class of the tangent bundle. Together with its functorial properties, this property completely determines the CSM class of any variety. If $f \colon X \to Y$ is a morphism between proper varieties that restricts to an isomorphism over $Z$, then $f_*(c_{SM}(\mathbf{1}_Z)) = c_{SM}(\mathbf{1}_{f(Z)})$.

We now prove Theorem~\ref{thm:csm}.
Let $L \subseteq \kk^E$ be a linear space of dimension $r$, and let $Y_L$ be the closure of $L$ in $(\mathbb{P}^1)^E$, the matroid Schubert variety of $L$. 
Recall from the introduction that the singular homology $H_{2k}(Y_L, \mathbb{Z})$ has a basis labeled by the flats of rank $k$. For a flat $F$, set $L^F = L/L_F$. The closure of a cell labeled by $F$ can be identified with the matroid Schubert variety of the linear space $L^F$. For a flat $F$, let $y_F \in H_{2k}(Y_L, \mathbb{Z})$ denote the class of the closure of the cell corresponding to $F$. Because $(\mathbb{P}^1)^E$ is the Schubert variety for the boolean matroid, in particular we obtain a basis for the singular homology of $(\mathbb{P}^1)^E$, where each $I\subseteq E$ defines the class $y_I \in H_{2|I|}((\PP^1)^E,\ZZ)$.
Note that the product $\prod_{i\in I} y_i$ of the divisor classes in Definition~\ref{defn:alphayis} is Poincar\'e dual to $y_I$ in the sense that for $I'\subseteq E$, we have $(\prod_{i\in I'} y_i) \cap y_I = 1$ if $I = I'$ and is $0$ otherwise.

\begin{lemma}\label{lem:push}
The pushforward $H_{\bullet}(Y_L, \mathbb{Z}) \to H_{\bullet}((\mathbb{P}^1)^E, \mathbb{Z})$ sends $y_F$ to $\sum_{I} y_I$, where the sum is over bases of $\M|F$. 
\end{lemma}

\begin{proof}
In degree $r$, this follows from \cite[Theorem 1.3c]{ArdilaBoocher}. The general case then follows from the identification of the closure of the cell indexed by $F$ with the matroid Schubert variety of $L^F$. 
\end{proof}

\begin{proof}[Proof of Theorem~\ref{thm:csm}]
Because the $W_L$ is an snc compactification of $L$, the CSM class of $L$ in $W_L$ is $c(\mathcal{T}_{W_L}(- \log \partial W_L)) \cap [W_L]$ by \cite[Theorem 1]{Aluffi99}. Let $\iota \colon W_L \to X_E$ be the inclusion. As $\mathcal{T}_{W_L}(- \log \partial W_L) = \iota^* \mathcal{S}_L$ and $[W_L] = c_{n - r}(\mathcal{Q}_L)$, the projection formula implies that 
$$\iota_* (c(\mathcal{T}_{W_L}(- \log \partial W_L)) \cap [W_L]) = c(\mathcal{S}_L) \cup c_{n - r}(\mathcal{Q}_L) \cap [X_E].$$ 
Using Theorem~\ref{thm:SQintersect} and Theorem~\ref{thm:logtangent}, one can show that 
$$\int_{X_E} \iota_* c(\mathcal{T}_{W_L}(- \log \partial W_L)) \cdot \prod_{i \in I}y_i = \begin{cases} 1, & I \text{ independent} \\ 0, & \text{otherwise}.\end{cases}$$
Therefore, the pushforward of $c_{SM}(\mathbf{1}_L) \in H_{\bullet}(W_L, \mathbb{Z})$ to $H_{\bullet}((\mathbb{P}^1)^E, \mathbb{Z})$ is $\sum_{I \text{ independent}}y_I$. The functoriality of CSM classes implies that this is the pushforward of the CSM class of $L$ in $Y_L$. From Lemma~\ref{lem:push}, we note that the pushforward on homology from $Y_L$ to $(\mathbb{P}^1)^E$ is injective, and $\sum_{F} y_F$ pushes forward to the claimed class. 
\end{proof}

\begin{remark}
Using the stratification of $Y_L$ by cells which are identified with matroid Schubert varieties for restrictions to flats of $\M$, Theorem~\ref{thm:csm} implies that
$$c_{SM}(\mathbf{1}_{Y_L}) = \sum_{F \in \mathscr L(\M)} |\{G\in \mathscr L(\M) \mid G\supseteq F\}| \cdot y_F.$$
\end{remark}

\appendix
\section{Polytope algebras and K-rings of toric varieties}

The notion of valuativity and the polytope algebra both have many variants, sometimes equivalent and sometimes not.  In this mostly expository appendix, we collect these together, and record their relationship to the K-ring of toric varieties.

\subsection{Variants of valuativity}

Valuative functions have been studied extensively as combinatorial generalizations of measures.
We point to \cite{McM93b} and \cite[\S6]{Sch14} as references and give a brief summary here.

\medskip
For $S\subseteq \RR^n$ (or $\QQ^n$), denote its indicator function by $\one_S \colon \RR^n \text{ (or $\QQ^n$)} \to \ZZ$ defined as 
$$
\one_S(x) = \begin{cases}
1 & \textnormal{if }x\in S\\
0 & \textnormal{otherwise}.
 \end{cases}
$$
Let $\mathscr S \subseteq 2^{\RR^n}$ be a collection of nonempty\footnote{Some authors allow $\emptyset \in \mathscr S$ and then impose by convention a triviality for $\emptyset$, such as $f(\emptyset) = 0$ for a function $f$ on $\mathscr S$.  See for instance \cite{Sal68, McM89}.  Here, we prefer to begin with collections of nonempty subsets.}
subsets of $\RR^n$.  We write
    $$
    \ind(\mathscr {S}) := \ZZ\{\one_S \mid S\in\mathscr{S}\}
    $$
for the $\ZZ$-module generated by the indicator functions of elements of $\mathscr {S}$.
For a hyperplane $H \subseteq \RR^n$, let $H^+$ and $H^-$ denote the two closed half-spaces that it defines.
The notion of valuative functions on $\mathscr S$ has many variants:

\begin{definition}\label{defn:vals}
For an abelian group $A$, we say a function {$f\colon \mathscr S{\cup \{\emptyset\}} \to A$
{with $f(\emptyset) = 0$}} is 
\begin{enumerate}\itemsep 5pt
    \item[(a)] \textbf{weakly valuative} if $f(S) = f(S \cap H^+) + f(S\cap H^-) - f(S \cap H)$
 for any $S \in \mathscr S$ and hyperplane $H$ such that $S\cap H^+, S \cap H^-, S \cap H \in \mathscr  S$,
    
    \item[(b)] (when $\mathscr S$ consists of polyhedra) satisfies the \textbf{weak inclusion-exclusion principle} if for any polyhedral subdivision $S = \bigcup_{i=1}^k S_i$ such that $S\in \mathscr S$ and $\bigcap_{j\in J}S_j \in \mathscr S {\cup \{\emptyset\}}$ 
    for every $J \subseteq \{1,\ldots, k\}$, the inclusion-exclusion relation $f(S) = \sum_{J \subseteq \{1,\ldots, k\}} (-1)^{|J|-1} f(\bigcap_{j\in J}S_j) $ holds, 
    
\item[(c)] is \textbf{additive} (a.k.a.\ \textbf{valuative}) if $f(S_1 \cup S_2) + f(S_1 \cap S_2) = f(S_1) + f(S_2)$ for any pair $S_1, S_2 \in \mathscr S$ such that $S_1\cup S_2, S_1\cap S_2 \in \mathscr S{\cup \{\emptyset\}}$,
\item[(d)] satisfies the \textbf{inclusion-exclusion principle} if for any union $S = \bigcup_{i=1}^k S_i$ such that $S\in \mathscr S$ and $\bigcap_{j\in J}S_j \in \mathscr S{\cup \{\emptyset\}}$ for every $J \subseteq \{1,\ldots, k\}$, the inclusion-exclusion relation $f(S) = \sum_{J \subseteq \{1,\ldots, k\}} (-1)^{|J|-1} f(\bigcap_{j\in J}S_j) $ holds,
\item[(e)] is \textbf{strongly valuative} if there exists a (unique) map of $\ZZ$-modules $\widehat f  \colon \ind(\mathscr S) \to A$ such that $f(S) = \widehat f(\one_S)$ for all $S \in \mathscr S$.
\end{enumerate}
\end{definition}

The following implications between the various notions of valuativity are immediate.
$$
\xymatrix
{
&(c)  \ar@{=>}[d] &(d)\ar@{=>}[l] \ar@{=>}[d] &(e) \ar@{=>}[l]\\
&(a) &(b) \ar@{=>}[l]
}
$$
Whether some or all of the implications can be reversed in the diagram for a given collection $\mathscr S$ is a difficult problem in general.  We collect some previous results here.

\begin{theorem}\label{thm:valequiv} As before, let $\mathscr S$ be a collection of nonempty subsets of $\RR^n$.
\begin{enumerate}\itemsep 5pt
\item \cite{Gro78} If $\mathscr S$ is intersection-closed, i.e.,\ $S_1, S_2 \in \mathscr S \implies S_1\cap S_2 =\emptyset \text { or } S_1\cap S_2\in \mathscr S$, then we have $(c) \iff (d) \iff (e)$.  For example, the family of all convex bodies in $\RR^n$ is intersection closed.
\item \cite{Sal68, Vol57} If $\mathscr S = \mathscr P$, the family of all polytopes in $\RR^n$ (which is intersection-closed) then we further have $(a) \iff (c)$ so all five notions are equivalent.  A minor modification of the proof also shows that the same holds for $\mathscr Q$, the family of all polyhedra in $\RR^n$ (see \cite[\S3.2]{McM09} for an explicit proof).
\item \cite{McM09} If $\mathscr S =\mathscr Q_\Lambda$ or $\mathscr P_\Lambda$, where $\mathscr Q_\Lambda$ is the family of all $\Lambda$-polyhedra in $\RR^n$ for a rank $n$ lattice $\Lambda\subseteq \RR^n$ (similarly $\mathscr P_\Lambda$ is the family of all $\Lambda$-polytopes), then we have $(c) \iff (d) \iff (e)$.  Note that $\mathscr Q_\Lambda$ and $\mathscr P_\Lambda$ are not intersection-closed.
\end{enumerate}
\end{theorem}

When $\mathscr S$ is the family of extended generalized permutohedra, i.e.,\ lattice polyhedra in $\RR^n$ whose normal fans coarsen (possibly convex subfans of) the normal fan of the  standard permutohedron of dimension $n-1$ in $\RR^n$, Derksen and Fink showed that $(b) \iff (e)$ \cite[Theorem 3.5]{DerksenFink}.
We ask whether the equivalence holds more generally:

\begin{question}\label{ques:defvalequiv}
How are the different variants of valuativity in Definition~\ref{defn:vals} related to each other when $\mathscr S$ is the set of all (lattice) polytopes whose normal fans coarsen a fixed complete (smooth and/or projective) rational fan?
\end{question}

We record here a useful consequence of Theorem~\ref{thm:valequiv} that taking faces of polytopes is a strongly valuative operation.  For a vector $v\in \RR^n$ and a polytope $P \subset \RR^n$, let $\operatorname{face}(P,v)$ be the face of $P$ on which the standard inner product with $v$ is minimized.

\begin{proposition}\label{prop:faceval}
Let $P_1, \ldots, P_k$ be (lattice) polytopes in $\RR^n$, and suppose  $\sum_{i=1}^k a_i \one_{P_i} = 0$ for some $a_1, \ldots, a_k \in \ZZ$.  Then, for any $v\in \RR^n$, one has $\sum_{i = 1}^k a_i \one_{\operatorname{face}(P_i,v)} =0$.
\end{proposition}

\begin{proof}
In other words, we need show that the function on the set of all (lattice) polytopes sending $P$ to $\one_{\operatorname{face}(P,v)}$ is strongly valuative.  By Theorem~\ref{thm:valequiv}, it suffices to show that this function is additive in the sense of Definition~\ref{defn:vals}(c), and this additivity is an immediate consequence of \cite[Theorem 4.6]{McM09}.
\end{proof}

\subsection{Variants of polytope algebras}

Fix a positive integer $n$.    
For a family $\mathscr S$ of nonempty subsets in $\RR^n$, let
\[
Z(\mathscr S) := \Big\{\sum_{S\in \mathscr S} a_S S \mid a_S\in \ZZ \text{ all but finitely many non-zero}\Big\}
\]
be the free abelian group generated by the set $\mathscr S$.  Define the following subgroups of $Z(\mathscr S)$:
\begin{align*}
\operatorname{val}(\mathscr S) = & \text{ the subgroup generated by the additive (a.k.a.\ valuative) relations, i.e.,}\\
 & P + Q - P\cup Q - P\cap Q \text{ whenever $P,Q,P\cap Q,P\cup Q\in \mathscr S$,}\\
\operatorname{stVal}(\mathscr S) = & \text{ the kernel of the map $Z(\mathscr S) \to \ind(\mathscr S)$ defined by $S \mapsto \one_S$, and}\\
\operatorname{transl}(\mathscr S) = &\text{ the subgroup generated by translation invariance relations, i.e.,}\\
& P - (P+v) \text{ whenever $P$ and $P+v\in \mathscr S$ for $v\in \RR^n$}.
\end{align*}
We may consider the following four quotient groups
\begin{align*}
\Pi(\mathscr S) &= Z(\mathscr S)/\operatorname{val}(\mathscr S),\\
\overline\Pi(\mathscr S) &= Z(\mathscr S)/(\operatorname{val}(\mathscr S)+\operatorname{transl}(\mathscr S)),\\
\ind(\mathscr S) &= Z(\mathscr S)/\operatorname{stVal}(\mathscr S), \text{ and}\\
\overline\ind(\mathscr S) &= Z(\mathscr S)/(\operatorname{stVal}(\mathscr S) +\operatorname{transl}(\mathscr S)).
\end{align*}
In each these four cases, for an element $P\in \mathscr S$ we denote by $[P]$ its image in the quotient group.  For a commutative ring $A$, we write $\Pi_A = \Pi \otimes A$, and similarly for $\overline\Pi$, $\ind$, and $\overline\ind$.\\
\indent We now consider the case where $\mathscr S$ is a family of polytopes.  In good cases, one may give these quotients groups a ring structure as in the following lemma, which is a minor variation of \cite[Lemma 6]{McM89}.
In this appendix, we use $\uplus$ for the Minkowski sum of polytopes when it is helpful to distinguish it notationally from the addition in $Z(\mathscr S)$.

\begin{lemma}
Suppose $\mathscr S$ is a Minkowski-sum-closed family of polytopes in $\RR^n$.  That is, if $P$ and $Q$ are polytopes in $\mathscr S$, then so is their Minkowski sum $P \msum Q$.    Then, for the quotient groups $\Pi(\mathscr S)$ and $\overline\Pi(\mathscr S)$, the multiplication given by
\[
[P]\cdot[Q] = [P\msum Q] \text{ for $P,Q\in \mathscr S$, and extended linearly to the whole group},
\]
is well-defined.  In particular, if further $\mathscr S$ contains the origin $\mathbf o$ of $\RR^n$, then the quotient groups are unital commutative rings with $[\mathbf o]$ the unit.
\end{lemma}

\begin{proof}
\cite[1.2.2]{Had57} shows that if $Q_1$ and $Q_2$ are polytopes such that $Q_1\cup Q_2$ is a polytope, then
\[
P\msum(Q_1\cup Q_2) = (P\msum Q_1) \cup (P\msum Q_2) \quad\text{and}\quad
P\msum(Q_1\cap Q_2) = (P\msum Q_1) \cap (P\msum Q_2)
\]
for any polytope $P \subseteq \RR^n$.  Hence, the multiplication via Minkowski sum is well-defined.  
\end{proof}

For a subring $R$ of $\mathbb{R}$, let $\mathscr P_R$ be the set of all nonempty $R$-polytopes in $\mathbb{R}^n$, i.e.,\ the polytopes that have vertices in $R^n$.  Usually $R$ will be either $\ZZ$, $\QQ$, or $\RR$.  When $R$ is $\QQ$ or $\RR$, Theorem~\ref{thm:valequiv}.(1) implies that $\Pi(\mathscr P_R) = \ind(\mathscr P_R)$, and hence $\overline\Pi(\mathscr P_R) = \overline\ind(\mathscr P_R)$ also.  The same conclusion holds when $R = \ZZ$ by Theorem~\ref{thm:valequiv}.(3).  The ring $\overline\Pi_\RR(\mathscr P_\RR)$ is what is often called McMullen's \emph{polytope algebra} as defined in \cite{McM89, McM93a}.

\medskip
For polytopes $P$ and $Q$, one says that $Q$ is a \emph{weak Minkowski summand} of $P$ if there is a polytope $Q'$ and $\lambda>0$ such that $\lambda Q \msum Q' = P$.  It is straightforward to show that this is equivalent to stating that the normal fan of $Q$ coarsens that of $P$.

\begin{definition}
Given a complete fan $\Sigma$ in $\RR^n$, we define the subfamily $\mathscr P_{R,\Sigma}\subseteq \mathscr P_R$ to be the set of $R$-polytopes whose normal fan coarsens $\Sigma$.  Let us define
\[
\Pi(R,\Sigma) = \text{the image of $Z(\mathscr P_{R,\Sigma})\subseteq Z(\mathscr P_R)$ in $\Pi(\mathscr P_R)$},
\]
and likewise for $\overline\Pi(R,\Sigma)$, $\ind(R,\Sigma)$, and $\overline\ind(R,\Sigma)$.  
\end{definition}

Note that, per Question~\ref{ques:defvalequiv}, it is unclear whether $\Pi(\mathscr P_{R,\Sigma}) = \Pi(R,\Sigma)$.
It is clear however that $\ind(R,\Sigma) = \ind(\mathscr P_{R,\Sigma})$, and also that $\operatorname{transl}(\mathscr P_{R,\Sigma}) = Z(\mathscr P_{R,\Sigma}) \cap \operatorname{transl}(\mathscr P_R)$, so that $\overline\ind(R,\Sigma) = \overline\ind(\mathscr P_{R,\Sigma})$.  Thus, when $R$ is $\ZZ$, $\QQ$, or $\RR$, the equivalence of additivity and strong valuativity, as noted in Theorem~\ref{thm:valequiv}(3), yields the following.

\begin{proposition}
When $R$ is $\ZZ$, $\QQ$, or $\RR$, one has
\[
\Pi(R,\Sigma) = \ind(R,\Sigma) = \ind(\mathscr P_{R,\Sigma}) \quad\text{and}\quad \overline\Pi(R,\Sigma) = \overline\ind(R,\Sigma) = \overline\ind(\mathscr P_{R,\Sigma}).
\]
\end{proposition}

We conclude this section with another variant of the polytope algebra given in \cite{Mor93c}.
Given a complete rational fan $\Sigma$, Morelli defines rings $L_\Sigma(\ZZ^n)$ and $\mathscr L_\Sigma(\ZZ^n)$ as follows.  For a point $p\in \RR^n$ and a polytope $P$, if $p\in P$ then define $TC_p(P) = \RR_{\geq 0}\{P-p\}$ to be the tangent cone of $P$ at $p$, and if $p\notin P$ define by convention $TC_p(P) = \emptyset$.  Let $\mathscr C$ be the collection of cones (always centered at the origin) in $\RR^n$, and let $\mathscr C_\Sigma = \{C \subseteq \RR^n \mid C^\vee \in \Sigma\}$ be the collection of cones which are duals of the cones in $\Sigma$.
Linearly extending the map $P \mapsto \one_{TC_p(P)}$, we obtain a map $\theta_p \colon \ind(\mathscr P_\ZZ) \to \ind(\mathscr C)$ for any point $p\in \ZZ^n$.  We then define
\begin{align*}
L_\Sigma(\ZZ^n) &= \text{the subgroup generated by $f\in \ind(\mathscr P_\ZZ)$ such that $\theta_p(f) \in \ind(\mathscr C_\Sigma)$ for all $p\in \ZZ^n$, and}\\
\mathscr L_\Sigma(\ZZ^n) &=  \text{the image of $L_\Sigma(\ZZ^n)$ in $ \overline\ind(\mathscr P_\ZZ)$}.
\end{align*}
In the paragraph preceding \cite[Theorem 10.46]{BG09}, the wording is somewhat ambiguous so as to assume implicitly that $\mathscr L_\Sigma(\ZZ^n)$ is equal to $\overline\ind(\mathscr P_{\ZZ,\Sigma})$.  We ask explicitly:

\begin{question}\label{ques:morelli}
For which complete fans $\Sigma$ is $L_\Sigma(\ZZ^n) = \ind(\mathscr P_{\ZZ,\Sigma})$ and/or $\mathscr L_\Sigma(\ZZ^n) = \overline\ind(\mathscr P_{\ZZ,\Sigma})$?
\end{question}
In \cite{FujinoPayne}, the authors give examples of smooth proper toric varieties which admit no nontrivial nef line bundles, so $\overline\ind(\mathscr P_{\ZZ,\Sigma})= \mathbb{Z}$, which gives examples of smooth fans for which both equalities in the question fail.
We will later prove Theorem~\ref{thm:folklore} which, when combined with a result of Morelli (Theorem~\ref{thm:Morelli} here), implies that for smooth projective fans $\Sigma$ we have that $L_\Sigma(\ZZ^n) = \ind(\mathscr P_{\ZZ,\Sigma})$ and $\mathscr L_\Sigma(\ZZ^n) = \overline\ind(\mathscr P_{\ZZ,\Sigma})$.

\subsection{Relation to (operational) Chow rings}

Let $R = \ZZ$ or $\QQ$ from this section onwards, so that we may consider toric varieties and their ($\QQ$-)divisor classes associated to polytopes.  Let $\Sigma$ be a complete rational fan and $X_\Sigma$ be its toric variety.  We point to \cite{Ful93} for basic facts on toric varieties.  Recall that a lattice polytope $Q\in \mathscr P_{\ZZ,\Sigma}$ defines a nef $T$-equivariant line bundle $\mathcal{O}_{X_{\Sigma}}(D_Q)$ in $X_\Sigma$, with the property that its divisor class $[D_Q]\in \operatorname{Pic}(X_\Sigma)$ does not change when we translate $Q$.
See \cite[Chapter 6]{CLS11} for a discussion of polytopes and line bundles.
We collect some results of Fulton and Sturmfels.

\begin{theorem}\label{thm:FSPoly}
Let $\Sigma$ be a complete rational fan, and let $A^\bullet(X_\Sigma)$ be the operational Chow cohomology ring of the toric variety $X_\Sigma$.  Then, we have:
\begin{enumerate}\itemsep 5pt
\item \cite[Theorem 3.1]{FultonSturmfels} The operational Chow ring is isomorphic (as a graded ring) to the ring of Minkowski weights on the fan $\Sigma$ with product structure coming from the fan displacement rule.
\item \cite[Theorem 5.1]{FultonSturmfels}  If $\Sigma$ is projective, the exponential map, sending $[Q] \mapsto \exp([D_Q])$, defines an injection of rings $\overline\ind_\QQ(\mathscr P_{\QQ,\Sigma}) \to A^\bullet(X_\Sigma)_\QQ$ whose image is the subring generated by $A^1(X_\Sigma)_\QQ = \operatorname{Pic}_\QQ(X_\Sigma)$.  The exponential map is an isomorphism when $\Sigma$ is further simplicial.
\item \cite[Theorem 5.2]{FultonSturmfels} The exponential map defines an isomorphism between $\overline\ind_\QQ(\mathscr P_\QQ)$ and the direct limit $\varinjlim A^\bullet(X_\Sigma)_\QQ$ over all complete fans.
\end{enumerate}
\end{theorem}

The image $\exp([D_Q])$ of the exponential map applied to $Q$ can be described in terms of Minkowski weights as follows:  The cone dual to a face $F$ of $Q$ gets weight equal to the lattice volume of $F$ (in the lattice of the affine span of $F$).
For the case when $R = \RR$, after a suitable modification of the definitions for the ring of Minkowski weights and the exponential map above, one has a similar injective map \cite[Theorem 2]{McM89} that is an isomorphism when $\Sigma$ is further simplicial \cite[Theorem 5.1]{McM93a}.  See also \cite{Bri97}.

\subsection{Relation to K-rings}

Let $K(X)$ be the Grothendieck ring of vector bundles on a smooth complete variety $X$.  For a smooth complete $\CC$-variety $X$, the Hirzebruch--Riemann--Roch theorem gives that the Chern character map $ch \colon K(X)_\QQ \to A(X)_\QQ$, defined on classes of line bundles by $[\mathcal L] \mapsto \exp(c_1(\mathcal L))$, 
is a ring isomorphism.  Comparing this to the second statement in Theorem~\ref{thm:FSPoly}, one concludes that there is an isomorphism $\overline \ind_\QQ(\mathscr P_{\QQ,\Sigma}) \simeq K(X_\Sigma)_\QQ$ determined by $[Q] \mapsto [\mathcal O_{X_\Sigma}(D_Q)]$ when $\Sigma$ is projective and smooth.
Obtaining this isomorphism not only over $\QQ$ but over $\ZZ$ is the topic of this section.  In particular, we prove the following.

\begin{theorem}\label{thm:folklore}
Let $\Sigma$ be a smooth projective fan, and let $K_T(X_\Sigma)$ be the Grothendieck ring of torus-equivariant vector bundles on $X_\Sigma$.  Then, there is a ring isomorphism
\[
\psi_T \colon \ind(\mathscr P_{\ZZ,\Sigma}) \overset\sim\to K_T(X_\Sigma)
\]
determined by the property $[P] \mapsto [\mathcal O_{X_{\Sigma}}(D_P)]$ for any $P\in \mathscr P_{\ZZ,\Sigma}$. This descends to an isomorphism $\psi \colon \overline\ind(\mathscr P_{\ZZ,\Sigma}) \overset\sim\to K(X_\Sigma)$. 
\end{theorem}

Morelli proved a similar result for any smooth complete (not necessarily projective) fan;  the following theorem collects \cite[Theorems 5, 6, and 8]{Mor93c}.
For $k \in \mathbb{Z}_{>0}$, let $\Psi^k$ be the $k$-th Adams operation, which is a ring endomorphism of $K_{(T)}(X_{\Sigma})$ that satisfies $\Psi^k[\mathcal{L}] = [\mathcal{L}^{\otimes k}]$ for $\mathcal{L}$ a ($T$-equivariant) line bundle.
For $m \in \mathbb{Z}^n$ and $[\mathcal{E}] \in K_T(X_{\Sigma})$, let $\chi(X_{\Sigma}, [\mathcal{E}])_m$ be the weight $m$ Euler characteristic.

\begin{theorem}\label{thm:Morelli}
Let $\Sigma$ be a smooth complete fan.
\begin{enumerate}\itemsep 5pt
\item The map $\textbf{I}_T \colon K_T(X_\Sigma) \stackrel{\sim}{\to} L_{\Sigma}(\mathbb{Z}^n) \subseteq \ZZ^{\QQ^n}$ given by $[\mathcal E] \mapsto \big(m/k \mapsto \chi(X_\Sigma;\Psi^k[\mathcal E])_m 
\big)$ is a well-defined ring isomorphism.
\item The map $\textbf I_T$ descends to an isomorphism $\textbf I \colon K(X_\Sigma) \stackrel{\sim}{\to} \mathcal{L}_{\Sigma}(\mathbb{Z}^n)$.
\end{enumerate}
\end{theorem}

However, in light of Question~\ref{ques:morelli}, it is unclear whether this proves Theorem~\ref{thm:folklore}.
We conclude with our proof of Theorem~\ref{thm:folklore} in the form of two lemmas. The proof of the second lemma uses ideas of Morelli.

\begin{lemma}\label{lem:surjInd}
There is a surjective ring homomorphism $\psi_T\colon \ind(\mathscr P_{\ZZ,\Sigma}) \to K_T(X_\Sigma)$ determined by the property $[P] \mapsto [\mathcal O_{X_{\Sigma}}(D_P)]$ for any $P\in \mathscr P_{\ZZ,\Sigma}$. It descends to a surjective ring homomorphism $\psi\colon \overline \ind(\mathscr P_{\ZZ,\Sigma}) \to K(X_\Sigma)$.
\end{lemma}

\begin{proof}
First we show that $\psi_T$ is well-defined.  We use the localization theorem for the torus-equivariant $K$-theory of smooth complete toric varieties \cite[Theorem 3.2]{Nie74}, which embeds $K_T(X_\Sigma)$ as a subring of $\prod_{\mathrm{pt} \in X^T_{\Sigma}} K_T(\mathrm{pt})$.  For each fixed maximal cone $\sigma \in \Sigma$, which corresponds to a point in $X^T_{\Sigma}$, the class of $[\mathcal{O}_{X_{\Sigma}}(D_P)]$ is sent to $T^{-v_{\sigma}}$, where $v_{\sigma}$ is the vertex of $P$ on which any functional in the interior of $\sigma$ achieves its minimum. That this is well-defined follows from Proposition~\ref{prop:faceval}.  To see that $\psi_T$ is a ring homomorphism, note that if $P$ and $Q$ are polytopes, then the vertex of $P \msum Q$ on which any functional in the interior of $\sigma$ achieves its minimum is the sum of the corresponding vertices of $P$ and $Q$. \\
\indent For the surjectivity of $\psi_T$, first note that for a complete smooth toric variety $X_\Sigma$, the ring $K_T(X_\Sigma)$ is generated as a ring by the classes of $T$-equivariant line bundles \cite[Corollary 1]{Kly84} (see also \cite[Lemma 2.2]{AndersonPayne}).  If $\Sigma$ is further projective, any $T$-equivariant line bundle is isomorphic to $\mathcal L^\vee \otimes \mathcal  M$ for some ample $T$-equivariant lines bundles $\mathcal L$ and $\mathcal M$. 
Since $\psi_T$ surjects onto the classes of $T$-equivariant ample line bundles, it suffices now to show that for a $T$-equivariant ample line bundle $\mathcal L$, its inverse class $[\mathcal L^\vee]$ is a sum of powers of $[\mathcal{L}]$ (possibly with different equivariant structures).  Concretely, suppose we have a lattice polytope $P \subset\RR^n$ whose normal fan $\Sigma_P$ equals $\Sigma$. Let $N$ be the number of lattice points in $P$.
Denoting $p_S = \sum_{p\in S} p$ for a subset $S\subseteq P \cap \ZZ^n$,
we claim that
\[
[\mathcal O_{X_{\Sigma}}(-D_P)] = \sum_{k=1}^{N} (-1)^{k-1} \sum_{\substack{S\subseteq P\cap \ZZ^n\\ |S|= k}} [ \mathcal O_{X_{\Sigma}}(D_{(k-1)P-p_S})] \quad\text{as elements in $K_T(X_\Sigma)$}.
\]
By multiplying $[\mathcal O_{X_{\Sigma}}(D_P)]$, we equivalently check that
\[
\sum_{k = 0}^{N}(-1)^k \sum_{\substack{S\subseteq P\cap \ZZ^n\\ |S| = k}} [\mathcal O_{X_{\Sigma}}(D_{kP - p_S})] = 0.
\]
Here the $k = 0$ term should be interpreted as $[\mathcal{O}]$ with the trivial equivariant structure. At each $T$-fixed point $x$ of $X_\Sigma$ corresponding to a vertex $v$ of $P$, the localization value of the left-hand-side is zero since $[\mathcal O_{X_{\Sigma}}(D_{(|S|+1)P-p_{S\cup v}})]_x = [\mathcal O_{X_{\Sigma}}(D_{|S|P-p_S})]_x$ for any $S\subseteq (P\cap\ZZ^n)\setminus v$.\\
\indent Finally, we note that for $Q\in \mathscr P_{\ZZ,\Sigma}$, the divisor class $[D_Q]$ is invariant under translation of $Q$, so translation invariance is clear. Therefore $\psi_T$ descends to a map $\psi\colon \overline \ind(\mathscr P_{\ZZ,\Sigma}) \to K(X_\Sigma)$, which is surjective because $K_T(X_{\Sigma}) \to K(X_{\Sigma})$ is surjective.
\end{proof}

\begin{lemma}
The maps $\psi_T$ and $\psi$ given in the previous lemma are injective.
\end{lemma}

\begin{proof}
For $[\mathcal{E}] \in K_T(X_{\Sigma})$, consider the function $\QQ^n \to \ZZ$ defined by
\[
m/k \mapsto \chi(X_\Sigma;\Psi^k[\mathcal E])_m \quad\text{for $m\in \ZZ^n$ and $k \in \ZZ_{>0}$}.
\]
In order to see that this is a well-defined function, we need to check that
$$\chi(X_{\Sigma}; \Psi^{k}[\mathcal{E}])_m = \chi(X_{\Sigma}; \Psi^{nk}[\mathcal{E}])_{nm} \quad\text{for any $n\in \ZZ_{>0}$}.$$
By Lemma~\ref{lem:surjInd} and because the classes of the polytopes $P\in \mathscr P_{\ZZ,\Sigma}$ generate $\ind(\mathscr P_{\ZZ,\Sigma})$,  it suffices to check that
\[
\chi(X_{\Sigma}; \Psi^k[\mathcal{O}_{X_{\Sigma}}(D_P)])_m = \chi(X_{\Sigma}; \Psi^{nk}[\mathcal{O}_{X_{\Sigma}}(D_P)])_{nm} \quad\text{for any $n\in \ZZ_{>0}$}
\]
for an arbitrary polytope $P \in \mathscr{P}_{\mathbb{Z}, \Sigma}$.
This then follows from the fact that for any positive integer $\ell$ and $m\in \ZZ^n$, one has
\[
\chi(X_\Sigma, \Psi^\ell[\mathcal O_{X_{\Sigma}}(D_P)])_m = \begin{cases}
1 & \text{if $m\in \ell P$}\\
0 &\text{otherwise}.
\end{cases}
\]
Indeed, $\Psi^\ell [\mathcal{O}_{X_{\Sigma}}(D_P)] = [\mathcal{O}_{X_{\Sigma}}(D_{\ell P})]$, we can identify $H^0(X_{\Sigma}; \mathcal{O}_{X_{\Sigma}}(D_{\ell P}))$ with the vector space spanned by lattice points in $\ell P$, and the  higher cohomology of base-point-free line bundles on toric varieties vanishes \cite[\S3.4 \& \S3.5]{Ful93}.

We now construct a map $K_T(X_{\Sigma}) \to \ind(\mathscr P_{\ZZ, \Sigma})$.
By Lemma~\ref{lem:surjInd}, every class $[\mathcal{E}] \in K_T(X_{\Sigma})$ is of the form $[\mathcal{E}] = \sum_i a_i [\mathcal{O}_{X_{\Sigma}}(D_{P_i})]$ for some $P_i \in \mathscr{P}_{\mathbb{Z}, \Sigma}$. We send $[\mathcal{E}]$ to $\sum_i a_i [P_i] \in \ind(\mathscr P_{\ZZ, \Sigma})$. 
The construction above recovers the evaluations of $\sum a_i [P_i]$ at points in $\mathbb{Q}^n$. Because two finite sums of indicator functions of lattice polytopes are equal if they agree on $\mathbb{Q}^n$, this map is well-defined.
It is clearly a left-inverse of $\psi_T$ which descends to a left-inverse of $\psi$. 
\end{proof}

\small
\bibliography{sources}
\bibliographystyle{alpha}

\end{document}